\documentclass[10pt, letterpaper,reqno]{amsart}
\usepackage[left=1in,right=1in,bottom=0.5in,top=0.8in]{geometry}
\usepackage{amsfonts}
\usepackage{amsmath, amssymb}
\usepackage{graphicx}
\usepackage[font=small,labelfont=bf]{caption}
\usepackage{epstopdf}
\usepackage{xcolor}
\usepackage{amsthm}
\usepackage{float}
\usepackage{enumitem}
\usepackage{pgfplots}
\usepackage{listings}
\usepackage{longtable}
\usepackage{mathrsfs}
\usepackage{dsfont}
\usepackage{mathtools}
\usepackage{dsfont}
\usepackage{stmaryrd}
\usepackage[pdfpagelabels,hyperindex]{hyperref}
\hypersetup{linkbordercolor=green}

\newcommand*{\rom}[1]{\expandafter\@slowromancap\romannumeral #1@}
 \colorlet{lgray}{white!80!black}
\colorlet{lred}{white!85!red}
\colorlet{lgreen}{white!60!green}
\colorlet{dgreen}{black!30!green}
\colorlet{lpurple}{white!60!purple}
\colorlet{lblue}{white!60!blue}
\definecolor{green}{rgb}{0.1,0.8,0.1}
\definecolor{yellow}{rgb}{1.0,0.85,0.25}
\definecolor{purple}{rgb}{1.0, 0, 1.0}
\definecolor{blue}{rgb}{0, 0, 1.0}
\tikzstyle{unfused}=[lgray, line width=1.5pt, ->]
\tikzstyle{fused}=[lgray, line width=4pt, ->]
\tikzstyle{dual}=[black, line width=1pt, dashed]
\tikzstyle{lightdual}=[black, line width=0.5pt, dashed]
\tikzstyle{cut}=[black, line width=1.0pt]

\theoremstyle{plain}
\newtheorem{theorem}{Theorem}[section]
\newtheorem{lemma}[theorem]{Lemma}
\newtheorem{proposition}[theorem]{Proposition}
\newtheorem{corollary}[theorem]{Corollary}

\theoremstyle{remark}
\newtheorem{definition}[theorem]{Definition}
\newtheorem{remark}[theorem]{Remark}

\colorlet{shadecolor}{gray!20}
\pgfplotsset{compat=1.9}
\usetikzlibrary{shapes.multipart}
\usetikzlibrary{patterns}
\usetikzlibrary{shapes.multipart}
\usetikzlibrary{arrows}
\usetikzlibrary{decorations.markings}
\usepgflibrary{decorations.shapes}
\usetikzlibrary{decorations.shapes}
\usepgflibrary{shapes.symbols}
\usetikzlibrary{shapes.symbols}
\usetikzlibrary{decorations.pathreplacing}
\tikzstyle{fleche}=[>=stealth', postaction={decorate}, thick]
\tikzstyle{axis}=[->, >=stealth', thick, gray]
\tikzstyle{paths}=[>->, >=stealth', thick]
\tikzstyle{path}=[->, >=stealth', thick]
\tikzstyle{grille}=[dotted, gray]
\DeclareFontFamily{U}{mathx}{}
\DeclareFontShape{U}{mathx}{m}{n}{<-> mathx10}{}
\DeclareSymbolFont{mathx}{U}{mathx}{m}{n}
\DeclareMathAccent{\widehat}{0}{mathx}{"70}
\DeclareMathAccent{\widecheck}{0}{mathx}{"71}

\usepackage{romannum}
\usepackage{upgreek}

\usepackage[scr=boondoxo]{mathalfa}

\newcommand{\ber}{\mathrm{Ber}}
\newcommand{\lrw}{\mathsf{LRW}}
\newcommand{\bg}{\boldsymbol{\gamma}}
\newcommand{\bs}{\boldsymbol{\sigma}}

\renewcommand{\a}{\mathsf{a}}
\renewcommand{\b}{\mathsf{b}}
\def \lc {\mathbf{loc}}
\def \lBr {\llbracket}
\def \rBr {\rrbracket}
\def \lb {\left(}
\def \rb {\right)} 
\def \llb {\llbracket}
\def \rrb {\rrbracket}
\def \ral {\longrightarrow}
\def \lal {\longleftarrow}

\newcommand{\TV}[1]{{\left\lVert #1 \right\rVert}_{\normalfont
\text{TV}}}

\numberwithin{equation}{section}

\usepgflibrary{fpu}
\makeatletter
\let\NAT@parse\undefined
\makeatother

\author{Dominik Schmid}
\address
{
Dominik Schmid, 
University of Augsburg, Germany 
}
\email{d.schmid@uni-a.de} 
\author{Zongrui Yang}
\address
{
Zongrui Yang, 
University of California, Berkeley
}
\email{zongrui@berkeley.edu} 

\title[Approximation for the stationary measure of the open ASEP]{Sharp approximations for the stationary measures of the open ASEP I: The shock region}  
\begin{document}
\pagenumbering{arabic}

\begin{abstract} We investigate the role of asymmetry in the stationary measures of integrable probabilistic models with boundary interactions. This paper is the first in a series of three focusing on the open asymmetric simple exclusion process (ASEP). The stationary measure of open ASEP admits a two-layer Gibbs representation in which a correction term, which we call the Bryc polynomial, encodes the dependence on the asymmetry parameter. 
We give a probabilistic interpretation of the Bryc polynomial in terms of a birth–death-type Markov chain. As an application, we obtain sharp total variation approximations of the stationary measure by Bernoulli product measures in the high- and low-density phases of the shock region. These approximations yield further applications, including moderate deviation estimates for the stationary measure of open ASEP, as well as asymptotic results for the related model of open ASEP with light particles.
\end{abstract}
\maketitle

%
\section{Introduction and main results}

\subsection{Preface}

The simple exclusion process is a central model in non-equilibrium statistical physics and integrable probability. Introduced in the biology literature in 1968 by MacDonald, Gibbs and Pipkin \cite{macdonald1968kinetics}, and in the mathematical literature in 1970 by Spitzer \cite{SPITZER1970246}, it since then serves as a prototypical interacting particle system studied in various fields; see \cite{liggett1999stochastic,wood2020combinatorial,blythe2007nonequilibrium,corwin2022some} for a selection of expository work from combinatorial, probabilistic, and statistical mechanical perspectives. The one-dimensional simple exclusion process with bias parameter $q \geq 0$ has a simple description: particles attempt $q$-biased random walks on the integer lattice, i.e., they independently attempt to jump at rate $1$ to the right and rate $q$ to the left. The moves are subject to an exclusion rule, suppressing all jumps of particles to already occupied sites. 

The asymmetric simple exclusion process with $q\neq1$ is one of the few models rigorously known to belong to the Kardar--Parisi--Zhang (KPZ) universality class \cite{KPZ86}. Its totally asymmetric special case, $q=0$, known as TASEP, is particularly important because of its connection to the exactly solvable random growth model of exponential last passage percolation \cite{rost1981non}. For TASEP, quantities such as particle currents and space-time correlations admit explicit formulas through connections with representation theory and random matrix theory. In particular, KPZ scaling limits for the ASEP were first determined when $q=0$, and in turn used to construct universal limit objects for models in the KPZ universality class, notably the KPZ fixed point \cite{johansson2000shape,matetski2021kpz}. For general asymmetry $q\in(0,1)$, the analysis is substantially more involved, and analogous process-level KPZ scaling limits have been established only recently \cite{aggarwal2024kpz,aggarwal2024scaling}.

The goal of this work is to better understand the role of the asymmetry parameter $q$ in models in the KPZ universality class with open boundaries. To this end, we study the open ASEP, in which particles perform $q$-biased random walks on a finite segment subject to the exclusion constraint and may enter or exit at the endpoints at prescribed rates. The open ASEP exhibits a phase diagram governed by the entry and exit rates and reflected in the structure of its stationary measures; see \cite{liggett1975ergodic,derrida1993exact,uchiyama2004asymmetric,Bryc2017,wang2023askey}.

Our starting point is Bryc's two-layer Gibbs representation of the stationary measure of the open ASEP under Liggett's condition on the boundary rates \cite{bryc2024stationary}. This representation builds on the framework of two-layer Gibbs measures for integrable models with two-sided open boundaries developed in \cite{barraquand2024stationary}. The stationary measure is realized as a marginal of a reweighted law of a pair of independent Bernoulli random walk paths. The Radon--Nikodym derivative contains exponential weights and polynomial correction factors, which we call \emph{Bryc polynomials}, that encode the dependence on the asymmetry parameter $q$. These polynomials thus relate the stationary measures of the open ASEP and the open TASEP through their two-layer representations. For the latter \cite{bryc2024two,bryc2024limitCoex} provide effective tools for asymptotic analysis.

Our main contribution is a probabilistic interpretation of the Bryc polynomials as probability generating functions of the terminal states of inhomogeneous birth--death Markov chains. This representation makes the polynomials amenable to asymptotic analysis through probabilistic estimates for the birth--death chains. In particular, it allows us to control their effect along typical reweighted random walk paths and thereby transfer stationary approximation results from the open TASEP to the open ASEP. In a broader sense, it provides a way to understand the effect of asymmetry in integrable models with open boundaries through polynomial corrections to two-layer Gibbs representations; see Section~\ref{sec:ScopePerspective} for further discussion.

As our main application, we establish sharp Bernoulli product approximations in total variation distance for the stationary measures of the open ASEP in the high- and low-density phases of the shock region, under Liggett's condition. The estimates apply to joint occupation variables away from the left boundary in the high-density phase and away from the right boundary in the low-density phase. The errors decay exponentially in the distance from that boundary. These results refine earlier approximation results \cite{liggett1975ergodic,nestoridi2024approximating,yang2025limits} and yield further applications, including moderate deviation estimates for stationary particle counts and asymptotic results for the open ASEP with light particles.

\subsection{The open ASEP}\label{subsec:model and results}

We define the open ASEP as a continuous-time Markov chain $(\tau(t))_{t \geq 0}$ on $\{0,1\}^{N}$ with parameters
\begin{equation}\label{eq:conditions open ASEP}
\alpha,\beta>0,\quad \gamma,\delta\geq 0,\quad\text{and}\quad  0\leq q<1 . 
\end{equation}
We write $\tau=(\tau_1,\tau_2,\dots,\tau_N) \in \{0,1\}^N$ for configurations of the open ASEP, where $\tau_i$ is the indicator for the $i^{\text{th}}$ site in $\tau$ being occupied. On sites $\llbracket 2,N\rrbracket$ particles attempt jumps according to independent rate $q$ Poisson clocks to the left, and on sites $\llbracket 1,N-1\rrbracket$ to the right according to independent rate $1$ Poisson clocks, subject to the constraint that a jump is performed if and only if the target site is vacant. Here, we write $\llbracket a,b\rrbracket := [a,b] \cap \mathbb Z$ for all  $a<b$.  In addition, particles can enter at the left boundary at rate $\alpha$, exit at the right boundary at rate $\beta$, exit at the left boundary at rate $\gamma$, and enter at the right boundary at rate $\delta$; see Figure~\ref{fig:openASEP} for an illustration. Intuitively, the assumptions in \eqref{eq:conditions open ASEP} ensure that particles have a drift to the right side while new particles are created at site $1$ at a linear rate. We will frequently refer to the special case $q=\gamma=\delta=0$,  called the \textit{open TASEP}. 

In the following, we consider the function
\begin{equation} \label{eq:phi}
\phi_{\pm}(x,y):=\frac{1}{2x}\left(1-q-x+y\pm\sqrt{(1-q-x+y)^2+4xy}\right), \quad \mbox{for }\; x>0\mbox{ and }y\geq 0,
\end{equation}
and define the parameters
\begin{equation}\label{eq:defining ABCD}
\a=\phi_+(\alpha,\gamma),\quad\b=\phi_+(\beta,\delta),\quad \mathsf{c}=\phi_-(\alpha,\gamma), \quad \mathsf{d}=\phi_-(\beta,\delta).
\end{equation}
We remark that 
under our assumptions \eqref{eq:conditions open ASEP}, 
\begin{equation}\label{eq:conditions qABCD}
\a,\b\geq0,\quad -1<\mathsf{c},\mathsf{d}\leq 0 . 
\end{equation}
The parameters $\a,\b$ allow us to describe the boundary effects in the open ASEP.
More precisely, we define $\rho_{\textup{L}}:=(1+\a)^{-1}$, respectively $\rho_{\textup{R}}:=\b(1+\b)^{-1}$ as the left side, respectively right side \textit{effective densities}. The parameters $\a$ and $\b$ allow us to partition the open ASEP into different regimes. When $\a > \max(\b,1)$, we say that the open ASEP belongs to the \textit{low density phase}, for $\b > \max(\a,1)$ the open ASEP belongs to the \textit{high density phase}, and for $\a,\b \leq 1$ to the \textit{maximal current phase}. The special case $\a=\b>1$ is called the \textit{coexistence line}. Moreover, we say that the open ASEP belongs to the \textit{fan region} if $\a \b < 1$, and to the \textit{shock region} if $\a \b >1$. The different phases and regions are illustrated in Figure~\ref{fig:Phase}. Heuristically, the particle density under the stationary measure will be given by $\rho_{\textup{L}} < \frac{1}{2}$ in the low density phase, by $\rho_{\textup{R}} > \frac{1}{2}$ in the high density phase, and by $\frac{1}{2}$ in the maximal current phase \cite{liggett1975ergodic,uchiyama2004asymmetric}. 
The terms fan and shock region originate from the creation of fans and shocks in Burgers equation, where solutions can be constructed from the ASEP via hydrodynamic limits \cite{kipnis1999scaling}.

Throughout this article, we will work under \textit{Liggett's condition} from \cite{liggett1975ergodic} that
\begin{equation} \label{eq:Liggett's condition}
\gamma=q(1-\alpha),\quad\delta=q(1-\beta),\quad \alpha,\beta\in(0,1] .
\end{equation}
Note that under \eqref{eq:Liggett's condition}, the parameters $\a,\b,\mathsf{c},\mathsf{d}$ simplify to
\begin{equation}\label{eq:parameters under Liggett condition}
\a=\frac{1-\alpha}{\alpha},\quad\b=\frac{1-\beta}{\beta},\quad \mathsf{c}=-q,\quad\mathsf{d}=-q.
\end{equation}
\begin{figure}
\centering
\begin{tikzpicture}[scale=0.85]
\def\spiral[#1](#2)(#3:#4:#5){
\pgfmathsetmacro{\domain}{pi*#3/180+#4*2*pi}
\draw [#1,
       shift={(#2)},
       domain=0:\domain,
       variable=\t,
       smooth,
       samples=int(\domain/0.08)] plot ({\t r}: {#5*\t/\domain})
}
\def\particles(#1)(#2){
  \draw[black,thick](-3.9+#1,0.55-0.075+#2) -- (-4.9+#1,0.55-0.075+#2) -- (-4.9+#1,-0.4-0.075+#2) -- (-3.9+#1,-0.4-0.075+#2) -- (-3.9+#1,0.55-0.075+#2);
\node[shape=circle,scale=0.6,fill=red] (Y1) at (-4.15+#1,0.2-0.075+#2) {};
  	\node[shape=circle,scale=0.6,fill=red] (Y2) at (-4.6+#1,0.35-0.075+#2) {};
  	\node[shape=circle,scale=0.6,fill=red] (Y3) at (-4.2+#1,-0.2-0.075+#2) {};
   	\node[shape=circle,scale=0.6,fill=red] (Y4) at (-4.45+#1,0.05-0.075+#2) {};
  	\node[shape=circle,scale=0.6,fill=red] (Y5) at (-4.65+#1,-0.15-0.075+#2) {}; }
\def\annhil(#1)(#2){	  \spiral[black,thick](9.0+#1,0.09+#2)(0:3:0.42);
\draw[black,thick](8.5+#1,0.55+#2) -- (9.5+#1,0.55+#2) -- (9.5+#1,-0.4+#2) -- (8.5+#1,-0.4+#2) -- (8.5+#1,0.55+#2); }
\node[shape=circle,scale=1.5,draw] (B) at (2.3,0){} ;
	\node[shape=circle,scale=1.5,draw] (C) at (4.6,0) {};
	\node[shape=circle,scale=1.2,fill=red] (CB) at (2.3,0) {};
    \node[shape=circle,scale=1.5,draw] (A) at (0,0){} ;
 	\node[shape=circle,scale=1.5,draw] (D) at (6.9,0){} ; 
 	 	\node[shape=circle,scale=1.5,draw] (Z) at (-2.3,0){} ;
	\node[shape=circle,scale=1.2,fill=red] (YZ) at (-2.3,0) {};
   \node[line width=0pt,shape=circle,scale=1.6] (B2) at (2.3,0){};
	\node[line width=0pt,shape=circle,scale=2.5] (D2) at (6.9,0){};
		\node[line width=0pt,shape=circle,scale=2.5] (Z2) at (-2.3,0){};
		
	\node[line width=0pt,shape=circle,scale=2.5] (X10) at (6.8,0){};
		\node[line width=0pt,shape=circle,scale=2.5] (X11) at (-2.2,0){};	
			\node[line width=0pt,shape=circle,scale=2.5] (X12) at (8.4,0){};
		\node[line width=0pt,shape=circle,scale=2.5] (X13) at (-3.8,0){};		
\draw[thick] (Z) to (A);	
	\draw[thick] (A) to (B);	
		\draw[thick] (B) to (C);	
  \draw[thick] (C) to (D);

\particles(0)(0);
\particles(6.9+4.9+1.6)(0);


\draw [->,line width=1pt]  (B2) to [bend right,in=135,out=45] (C);
 \draw [->,line width=1pt] (6.9+1.6,-0.475) to [bend right,in=135,out=45] (D);	
   \draw [->,line width=1pt] (B2) to [bend right,in=-135,out=-45] (A);
     \draw [->,line width=1pt] (Z) to [bend right,in=135,out=45] (-3.9,-0.475);
    \node (text1) at (3.5,1){$1$} ;    
	\node (text2) at (1.1,1){$q$} ;   
	\node (text3) at (-2.3-1,1){$\alpha$}; 
	\node (text4) at (6.9+1,1){$\beta$};   
    \node (text3) at (-2.3-1,-1){$\gamma$}; 
	\node (text4) at (6.9+1,-1){$\delta$};   

\node[scale=0.9] (textL) at (-5.2,0){$\rho_{L}$} ;
\node[scale=0.9] (textR) at (9.9,0){$\rho_{R}$} ;

    \node[scale=0.9] (text1) at (-2.2,-0.7){$1$} ;    
    \node[scale=0.9] (text1) at (6.8,-0.7){$N$} ;   
  	
  \draw [->,line width=1pt] (-3.9,0.475) to [bend right,in=135,out=45] (Z);
   \draw [->,line width=1pt] (D) to [bend right,in=135,out=45] (6.9+1.6,0.475);

	\end{tikzpicture}	
\caption{\label{fig:openASEP} Visualization of the open ASEP for parameters $(q,\alpha,\beta,\gamma,\delta)$  on a segment of length $N$.}
\end{figure}
For the open ASEP on a segment of length $N$, we denote its unique stationary measure by $\mu_N=\mu^{(q,\a,\b)}_N$. Our main goal is to investigate the structure of the stationary measure $\mu^{(q,\a,\b)}_N$ for the open ASEP on a segment of length $N$ as $N\rightarrow \infty$.
\begin{figure}
    \centering
    \begin{tikzpicture}[scale=0.75]
 \draw[scale = 1,domain=6.7:11,smooth,variable=\x,dotted,thick] plot ({\x},{1/((\x-7)*1/3+2/3)*3+5});
 \draw[->] (5,5) to (5,10.2);
 \draw[->] (5.,5) to (11,5);
   \draw[dashed] (5,8) to (8,8);
  \draw[dashed] (8,8) to (8,5);
   \draw[dashed] (8,8) to (10.2,10.2);
   \node [left] at (5,8) {$1$};
   \node[below] at (8,5) {$1$};
     \node [below] at (11,5) {$\a$};
   \node [left] at (5,10) {$\b$};
 \draw[dashed] (8,4.9) to (8,5.1);
  \draw[dashed] (4.9,8) to (5.1,8);
 \node [below] at (5,5) {$(0,0)$};
    \node [above] at (6.5,8.5) {LD};\node [above] at (8,8.7) {LD};
    \node at (10.6,10) {\small coexistence line }; 
    \node [below] at (10,6) {HD};  \node [below] at (10,8) {HD}; 
 \node [below] at (6.5,6) {MC}; 
 \node at (6.4,10) {\small{fan}};
 \node at (7.6,10) {\small{shock}};
 \end{tikzpicture} 
        \caption{    \label{fig:Phase} Visualization of the different regimes: Low density (LD), high density (HD), and maximal current (MC) phase. Moreover, the curve $\a \b =1$ separates the fan and shock region, while the case $\a=\b>1$ is called the coexistence line. }
\end{figure}
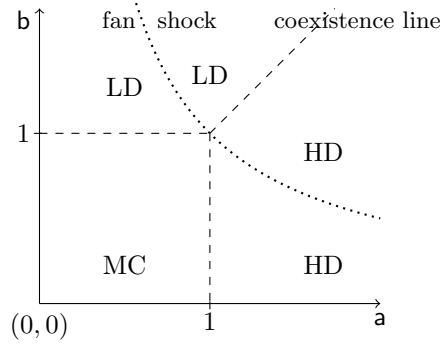

\subsection{Statement of the main characterization result} 
Fixing $q\in[0,1)$, we will define the Bryc polynomial $W_{\bs}(z)$ in Definition \ref{def:Bryc poly}, which (see Remark \ref{rmk:Bryc is polynomial}) is a polynomial with a single variable $z$ with non-negative coefficients depending on $q$.
The Bryc polynomial $W_{\bs}(z)$ is indexed by a \emph{composition} $\bs=(\sigma_0,\dots,\sigma_r)$ of $N+1$, where $r\in\mathbb{Z}_{\geq0}$. Here, we call a sequence $\bs=(\sigma_0,\dots,\sigma_r)$ a composition of $N+1$ if $\sigma_i\in\mathbb{N}$, $i\in\llbracket0,r\rrbracket$ and $\sigma_0+\dots+\sigma_r=N+1$. The order in which the elements appear in a composition is significant, which distinguishes it from the notion of a partition.  
This polynomial $W_{\bs}(z)$ was introduced by Bryc in \cite{bryc2024stationary} to characterize the stationary measure of the open ASEP. We next explain this characterization, and defer the precise construction of the polynomials $W_{\bs}(z)$ from \cite{bryc2024stationary} to Section~\ref{sec:TwoLayer}.

We begin by introducing some notation.
A \emph{lazy random walk path} is given by $\bg=(\gamma_0,\dots,\gamma_N)$ such that $\gamma_i\in\mathbb{Z}$ for $i\in\llbracket0,N\rrbracket$ and $\gamma_i-\gamma_{i-1}\in\{-1,0,1\}$ for $i\in\llbracket1,N\rrbracket$. We denote by $\lrw_N$ the space of lazy random walk paths $\bg$ of length $N$ started at $\gamma_0=0$. For each $\bg\in\lrw_N$, we write
\begin{equation*}
    \min(\bg) = \min\left\{\gamma_i \colon i \in \llbracket 0,N \rrbracket\right\} \quad \text{ and } \quad \max(\bg) = \max\left\{\gamma_i \colon i \in \llbracket 0,N \rrbracket\right\} . 
\end{equation*}
For a lazy random walk path $\bg\in\lrw_N$, we define the composition $\bs(\bg)$ associated to $\bg$ to be the composition $\bs(\bg)=(\sigma_0,\dots,\sigma_r)$ of $N+1$ given by 
\begin{equation*}
    r:=\max(\bg)-\min(\bg)\in\mathbb{Z}_{\geq0}
\end{equation*}
 and
\begin{equation} \label{eq:Signature}
\sigma_i:=\#\{j:\gamma_j=\min(\bg)+i\}\in\mathbb{N}\quad\mbox{ for } i\in\llbracket0,r\rrbracket.
\end{equation}
Here and below, for a set $\mathcal{S}$ we write $\#\mathcal{S}\in\mathbb{Z}_{\geq0}\cup\{\infty\}$ to denote the number of elements in it.
We note that $\bs(\bg)$ given by \eqref{eq:Signature} is indeed a composition of $N+1$: since $\bg$ is integer valued and has increments $\{-1,0,1\}$, every integer value in $\llbracket\min(\bg), \max(\bg)\rrbracket$ is visited at least once, and hence $\sigma_i\in\mathbb{N}$ for $i\in\llbracket0,r\rrbracket$; since the level sets $\{j:\gamma_j=\min(\bg)+i\}$ in \eqref{eq:Signature} partition the time indices  $\llbracket0,N\rrbracket$, we have $\sigma_0+\dots+\sigma_r=N+1$.

Let $\a,\b\geq0$ and write $H(\bg)$ for the number of horizontal steps in a lazy random walk path $\bg$, i.e.,
\[
H(\bg)=\#\{i\in\llbracket1,N\rrbracket, \gamma_i=\gamma_{i-1}\} \quad\mbox{ for }\bg\in\lrw_N.
\]  We define the probability measure $\mathbb{P}^{(q,\a,\b)}$ on $\lrw_N$ by 
\begin{equation}\label{eq:PNab}
\mathbb{P}^{(q,\a,\b)}(\bg)=\frac{1}{Z_N^{(q,\a,\b)}}\frac{2^{H(\bg)}\b^{\gamma_N}}{(\a\b)^{\min(\bg)}} W_{\bs(\bg)}(\a\b)\quad\mbox{ for }\bg\in\lrw_N,
\end{equation}
 where  we set
\begin{equation}\label{eq:normalization}
 Z_N^{(q,\a,\b)}=\sum_{\bg\in\lrw_N} \frac{2^{H(\bg)}\b^{\gamma_N}}{(\a\b)^{\min(\bg)}}W_{\bs(\bg)}(\a\b).
\end{equation}
We mention that $W_{\bs(\bg)}(\a\b)$ is the Bryc polynomial indexed by the composition $\bs(\bg)$ of $N+1$ associated to $\bg$ (as in \eqref{eq:Signature}), evaluated at $z=\a\b$. Since $\min(\bg)\leq\min(0,\gamma_N)$, we know that 
$ 2^{H(\bg)}\b^{\gamma_N} (\a\b)^{-\min(\bg)}  W_{\bs(\bg)}(\a\b)$ (and hence $Z_N^{(q,\a,\b)}$ given by \eqref{eq:normalization})
is a polynomial in $\a$ and $\b$ with non-negative coefficients depending on $q$.   When $\a=0$ or $\b=0$, \eqref{eq:PNab} is well-defined with the convention that $0^0=1$.
We also note that the normalization constant $Z_N^{(q,\a,\b)}$ given by \eqref{eq:normalization} is strictly positive for any $\a,\b\geq0$, which follows from considering the horizontal path $\bg\equiv0\in\lrw_N$ and noting that $W_{\bs(\bg)}(0)>0$; see Remark \ref{rmk:Bryc is polynomial}.

\begin{theorem}[c.f. Theorem 1.1 and (3.2), (3.3) in \cite{bryc2024stationary}]\label{thm:characterization stationary measure ASEP}
Consider the open ASEP on the lattice $\llbracket1,N\rrbracket$ with jump rates $\alpha,\beta,\gamma,\delta$ and $q$ satisfying \eqref{eq:conditions open ASEP}.  Assume that Liggett's condition \eqref{eq:Liggett's condition} holds, and recall that $\a=(1-\alpha)/\alpha$ and $\b=(1-\beta)/\beta$. 
Consider the random lazy random walk path $\bg\in\lrw_N$ with law $\mathbb{P}^{(q,\a,\b)}$ defined as in \eqref{eq:PNab}. Suppose $\kappa_i$, $i\in\llbracket1,N\rrbracket$ are i.i.d. Bernoulli random variables with mean $\frac{1}{2}$ that are independent of $\bg$. 
We define for all $i\in\llbracket1,N\rrbracket$ 
\[
    \tau_i:=\begin{cases}
        1 & \mbox{ if } \gamma_i-\gamma_{i-1}=1,\\
        \kappa_i & \mbox{ if } \gamma_i-\gamma_{i-1}=0 ,\\ 
        0 & \mbox{ if } \gamma_i-\gamma_{i-1}=-1.
    \end{cases}
\]
Then $(\tau_1,\dots,\tau_N)\in\{0,1\}^N$ has the law of the stationary measure $\mu_N^{(q,\a,\b)}$ of the open ASEP on $\llbracket1,N\rrbracket$.
\end{theorem}

One can observe from Theorem \ref{thm:characterization stationary measure ASEP} that the stationary measure $\mu_N^{(q,\a,\b)}$ of the open ASEP is determined from the measure $\mathbb{P}^{(q,\a,\b)}$ on lazy random walk paths $\bg\in\lrw_N$ in a simple and direct way. Specifically, at each site $i\in\llbracket1,N\rrbracket$, if $\bg$ increases at step $i$ we place a particle, if $\bg$ decreases at step $i$ we place a hole, and if $\bg$ stays put at step $i$, we flip a fair coin to determine whether we place a particle or a hole. The definition \eqref{eq:PNab} of the measure $\mathbb{P}^{(q,\a,\b)}$ is also explicit except for the Bryc polynomial term $W_{\bs(\bg)}(\a\b)$. 

We next introduce our main result, which provides a probabilistic characterization of the Bryc polynomials.

\begin{definition}\label{def:characterization measure}
Let $q\in[0,1)$, $N\in\mathbb{Z}_{\geq0}$, and $\boldsymbol{\mathsf{S}}=(\mathsf{S}(0),\dots,\mathsf{S}(N))\in\{0,1\}^{N+1}$ be a $0$-$1$ sequence. We define a time-inhomogeneous birth-death Markov chain $X^{\boldsymbol{\mathsf{S}}}=\{X^{\boldsymbol{\mathsf{S}}}(i)\}_{i=0}^{N+1}$ indexed by this $0$-$1$ sequence $\boldsymbol{\mathsf{S}}$ through $X^{\boldsymbol{\mathsf{S}}}(0)=X^{\boldsymbol{\mathsf{S}}}(1)=0$ and that for each $i\in\llbracket1,N\rrbracket$ and $j\in\llbracket0,i-1\rrbracket$,
\begin{enumerate}
    \item [(\romannum{1})] if $\mathsf{S}(i)=1$, we let $X^{\boldsymbol{\mathsf{S}}}(i+1)-X^{\boldsymbol{\mathsf{S}}}(i)\in\{-1,0\}$ with 
    \begin{equation}\label{eq:etget1}
    \mathbb{P}\left(X^{\boldsymbol{\mathsf{S}}}(i+1)=j-1 \Big{|}X^{\boldsymbol{\mathsf{S}}}(i)=j\right)=\frac{1-q^{j}}{1-q^{i+1}},
    \end{equation}
    \item [(\romannum{2})] if $\mathsf{S}(i)=0$, we let $X^{\boldsymbol{\mathsf{S}}}(i+1)-X^{\boldsymbol{\mathsf{S}}}(i)\in\{0,1\}$ with 
    \begin{equation}\label{eq:etget2}
    \mathbb{P}\left(X^{\boldsymbol{\mathsf{S}}}(i+1)=j+1 \Big{|}X^{\boldsymbol{\mathsf{S}}}(i)=j\right)=\frac{q^{j+1}-q^{i+1}}{1-q^{i+1}}.
    \end{equation}
\end{enumerate} 
We mention that the right sides of \eqref{eq:etget1} and \eqref{eq:etget2} lie in $[0,1]$ since $j\in\llbracket0,i-1\rrbracket$ and $q\in[0,1)$. 
By an induction argument, we have $X^{\boldsymbol{\mathsf{S}}}(i)\in\llbracket0,i-1\rrbracket$ for all $i\in\llbracket1,N+1\rrbracket$ and $\boldsymbol{\mathsf{S}}$, and hence this Markov chain is well-defined. When $q=0$, we use again the convention $0^0=1$.
\end{definition}

\begin{definition}\label{def:associated 01 seq}
    Let $N\in\mathbb{Z}_{\geq0}$ and $\bs=(\sigma_0,\dots,\sigma_r)$, $r\in\mathbb{Z}_{\geq0}$ be a composition of $N+1$. 
We define the $0$-$1$ sequence associated to composition $\bs$ to be $\boldsymbol{\mathsf{S}}^{\bs}=(\mathsf{S}^{\bs}(0),\dots,\mathsf{S}^{\bs}(N))\in\{0,1\}^{N+1}$ such that,  for  $i\in\llbracket0,N\rrbracket$,
\begin{equation}\label{eq:01seq}
  \mathsf{S}^{\bs}(i)=\begin{cases}
        1 &\quad\mbox{ if }i\in\{\sigma_r,\sigma_r+\sigma_{r-1},\dots,\sigma_r+\dots+\sigma_1\},\\  
        0 &\quad\mbox{ otherwise}.
    \end{cases} 
\end{equation}
\end{definition}

We have the following probabilistic characterization of the Bryc polynomials. 

\begin{theorem}\label{thm:prob interpretation of Bryc polynomials}
Let $q\in[0,1)$, $N\in\mathbb{Z}_{\geq0}$ and $\bs=(\sigma_0,\dots,\sigma_r)$, $r\in\mathbb{Z}_{\geq0}$ be a composition of $N+1$. We have
    \begin{equation}\label{eq:rep of Bryc polynomial}
    W_{\bs}(z)=\mathbb{E}\left[ z^{X^{\boldsymbol{\mathsf{S}}^{\bs}}(N+1)}\right]. 
    \end{equation}
Here, on the left side, $W_{\bs}(z)$ is given by Definition \ref{def:Bryc poly}; and on the right side,
$\boldsymbol{\mathsf{S}}^{\bs}$ is the $0$-$1$ sequence associated to the composition $\bs$ as in Definition \ref{def:associated 01 seq}, and $X^{\boldsymbol{\mathsf{S}}^{\bs}}=\{X^{\boldsymbol{\mathsf{S}^{\bs}}}(i)\}_{i=0}^{N+1}$ is the birth-death Markov chain indexed by this $0$-$1$ sequence $\boldsymbol{\mathsf{S}}^{\bs}$ as in Definition \ref{def:characterization measure}. 
\end{theorem}

Theorem \ref{thm:prob interpretation of Bryc polynomials} will be restated as Theorem \ref{thm:prob interpretation of Bryc polynomials again} which will be proved in Section \ref{sec:TwoLayer}. The representation of the Bryc polynomial as a birth-death chain allows us to control the moments of the Bryc polynomial when the $0$-$1$ sequence $\boldsymbol{\mathsf{S}}^{\bs}$ as chosen randomly as in  Definition \ref{def:characterization measure}. 

\begin{proposition}\label{prop:moment bound Bryc MAIN}
Assume $q\in(0,1)$, $\b > \max(\a,1)$, $\a\b>1$, $z>1$, and $k\in\mathbb{N}$. For each $N\in\mathbb{N}$, consider the random path $\bg$ with law $\mathbb{P}^{(0,\a,\b)}$ on $\lrw_N$. Let $\bs(\bg)$ be the composition of $N+1$ associated to $\bg$ as in \eqref{eq:Signature}, and let $W_{\bs(\bg)}(z)$ be the Bryc polynomial indexed by the composition $\bs(\bg)$ as in Definition \ref{def:Bryc poly}.  Then,  
\[
    \sup_{N\in\mathbb{N}}\mathbb{E}\bigl[W_{\bs(\bg)}(z)^k\bigr]<\infty
\] 
is a finite constant depending only on $q$, $\a$, $\b$, $z$, and $k$.
\end{proposition}

Proposition~\ref{prop:moment bound Bryc MAIN} will be restarted as Proposition~\ref{prop:moment bound Bryc} and proved in Section~ \ref{sec:MomentBryc}.  Intuitively, Proposition~\ref{prop:moment bound Bryc MAIN} shows that a typical Bryc polynomial has only mild affect on the stationary measure of the open ASEP. This result provides a general framework in order to transfer results on the stationary measure of the open TASEP to the open ASEP stationary measuret.

\subsection{Approximation of the stationary measure of the open ASEP}

As our main application of the above insights on the two-layer representation of the stationary measure of the open ASEP, we provide sharp approximation results in the high density and the low density phase of the shock region. 
As a measure of distance, we consider for two probability measures $\nu,\nu'$ on a common probability space $\Omega$ the total variation distance
\begin{equation}\label{def:TVdistance}
    \TV{\nu - \nu'} := \frac{1}{2}\sum_{\omega \in \Omega} | \nu(\omega) - \nu'(\omega)| .  
\end{equation}

\begin{theorem}\label{thm:ApproxHighLow}
Consider the open ASEP on the lattice $\llbracket1,N\rrbracket$ with jump rates $\alpha,\beta,\gamma,\delta$ and $q$ satisfying \eqref{eq:conditions open ASEP}. Assume Liggett's condition \eqref{eq:Liggett's condition} holds and that the open ASEP is in the shock region $\a\b>1$. Then there exist constants $C,c>0$ depending only on $q$, $\a$ and $\b$, such that for each $N\geq m\geq1$, the following hold. 
\begin{enumerate}
    \item [(\romannum{1})] In the high density phase $\b > \max(\a,1)$, we have
\begin{equation}
    \TV{ \mu_N^{(q,\a,\b)}\left( (\tau_i)_{i \in\llbracket m,N\rrbracket} \in   \cdot\, \right) - \ber_{\llbracket m,N\rrbracket}\left(\b/(1+\b) \right) }\leq Ce^{-cm}.
\end{equation} 
\item [(\romannum{2})] In the low density phase $\a > \max(\b,1)$, we have      
\begin{equation}
    \TV{ \mu_N^{(q,\a,\b)}\left( (\tau_i)_{i \in\llbracket 1,m\rrbracket} \in   \cdot\, \right) - \ber_{\llbracket 1,m\rrbracket}\left(1/(1+\a) \right) }\leq Ce^{-c(N-m)}.
\end{equation} 
\end{enumerate}
\end{theorem}

Let us mention that the above results are sharp by comparing the densities at a single site \cite{sasamoto2000density}. Moreover, we remark that similar approximation results were achieved in \cite{nestoridi2024approximating} for the three-parameter open ASEP where $\gamma=\delta=0$ in the fan region, and in \cite{yang2025limits} for the stationary measure of the open ASEP near in the boundary. As a further consequence, we obtain the following moderate deviation bounds on the number of particles under the stationary measure of the open ASEP.  

 \begin{proposition}\label{thm:moderate deviation}
Consider the open ASEP on the lattice \(\llbracket1,N\rrbracket\) with jump rates \(\alpha,\beta,\gamma,\delta\) and \(q\) satisfying
\eqref{eq:conditions open ASEP}. Assume that Liggett's condition \eqref{eq:Liggett's condition} holds and that the open ASEP is in the shock region of the high and low density phases, i.e.,  \(\a\b>1\) and $\a\neq\b$, where $\a=(1-\alpha)/\alpha$ and $\b= (1-\beta)/\beta$. We denote its stationary measure by \(\mu_N^{(q,\a,\b)}\). Define $\rho= \b/(1+\b)$ in the high density phase $\b>\max(\a,1)$ and  $\rho= 1/(1+\a)$ in the low density phase $\a>\max(\b,1)$. Then there exist constants $C,c\in(0,\infty)$, depending only on \(q,\a,\b\), such that, for every \(N\in\mathbb N\), every nonempty set $I  \subseteq\llbracket1,N\rrbracket$ and every  \(x\ge0\),   
\begin{equation*} 
\mu_N^{(q,\a,\b)}
\left(
\left|
\sum_{i\in I}\tau_i-\rho|I|
\right|
\ge x\sqrt{|I|}
\right)
\le
C e^{-cx^2 } .
\end{equation*} 
\end{proposition}

Let us point out that the set $I$ in Proposition~\ref{thm:moderate deviation} does not need to be given as a segment.

\subsection{Applications for the  open ASEP with light particles}

For another application of the main approximation results, we consider the open ASEP with a finite number of light particles from \cite{SY:openLight}. In this model, we denote the state space for $r$ light particles on the segment $\llbracket N \rrbracket$ by 
\[
\widehat{\Omega}_{N,r}:=\left\{\tau=(\tau_1,\dots,\tau_N)\, \colon \, \tau_i\in\{0, 1, 2\} \mbox{ for } i\in\llbracket N \rrbracket,\quad \sum_{i=1}^{N} \mathds{1}_{\{ \tau_i = 2 \} } = r\right\}.
\] 
We have the following dynamics with parameters 
\begin{equation}
    \alpha,\beta>0,\quad\gamma,\delta\geq0,\quad0\leq q<1 ,
\end{equation}
which models the evolution of particles on the lattice $\lBr N \rBr=\{1,\dots,N\}$. 
Particles move in the bulk with the following rates:
\begin{equation}
    2 0 \overset{1}{\ral}02,\quad\quad 1 0 \overset{1}{\ral}01,\quad\quad 1 2 \overset{1}{\ral}2 1,
\end{equation}
\begin{equation}
   2 0 \overset{q}{\lal}02,\quad\quad 1 0 \overset{q}{\lal}01,\quad\quad 12\overset{q}{\lal}21. 
\end{equation}
At the open boundaries, only ordinary particles are allowed to enter or exit while light particles of type $2$ are prohibited from entering or exiting. Specifically, at the left boundary, (ordinary) particles enter the system with rate $\alpha$ and exit the system with rate $\gamma$; similarly, at the right boundary, (ordinary) particles enter the system with rate $\delta$ and exit the system with rate $\beta$. 
We prohibit the entry of an (ordinary) particle if the target site is already occupied, whether by an (ordinary) particle or a light particle.
We let $\mu_{N,r}$ denote the unique stationary measure of the system, which is a probability measure on $\widehat{\Omega}_{N,r}$ and let the locations of the $r$ light particles within the lattice be given by 
\[
1\leq\lc_1<\dots<\lc_r\leq N.
\]
In particular, the number of light particles in the system is preserved over time. The following result extends Theorem 1.3 in \cite{SY:openLight} to the shock region.  
\begin{theorem}[]\label{thm:mainHighLowLightIntro} Fix $r\in \mathbb{N}$ and suppose that $\a \b >1$. Moreover, assume that Liggett's condition \eqref{eq:Liggett's condition} holds.
  Suppose $(a_N)_{N=1,2,\dots}$ and $(b_N)_{N=1,2,\dots}$ are two arbitrary sequences of integers such that  $1\leq a_N<b_N \leq N$ and $a_N,N-b_N\rightarrow\infty$. Then in the high density shock region where $\a \b >1$ and $\a\geq 1$, we get 
\begin{equation}
    \lim_{N\rightarrow\infty}\mu_{N,r}(\lc_1,\dots,\lc_r\in \llbracket 1,a_N \rrbracket )=1. 
\end{equation}
Then in the low density shock region where $\a \b >1$ and $\b\geq 1$,
 \begin{equation}
          \lim_{N\rightarrow\infty}\mu_{N,r}( \lc_1,\dots,\lc_r\in \llbracket b_N,N\rrbracket )=1.
 \end{equation}
\end{theorem}
In particular, this resolves a conjecture from \cite{SY:openLight}; see Remark 1.4 therein. 
In \cite{SY:openLight},  the speed of convergence to the stationary measure $\widehat{\mu}_{N,r}$ of the open ASEP with light particles is quantified in terms of total variation mixing times. We define the $\varepsilon$-mixing time for the open ASEP with light particles for all $\varepsilon \in (0,1)$ as
 \begin{equation}\label{def:MixingTime}
     t^N_{\textup{mix}}(\varepsilon) := \inf \left\{ t \geq 0 \, \colon \, \max_{\tilde{\tau} \in \widehat{\Omega}_{N,r}} \TV{ \mathbb{P}(\tau(t) \in \, \cdot \, | \, \tau(0)=\tilde{\tau}) - \widehat{\mu}_{N,r} } \leq \varepsilon \right\},
 \end{equation}
where $\tau(t)=\lb\tau_1(t),\dots,\tau_N(t)\rb$ denotes the occupation variables for the process at time $t\geq0$.  The following result is the analogue of Theorem 1.6 in \cite{SY:openLight} for the shock region. 
\begin{theorem}\label{thm:MainMixingTimeIntro} Fix $r\in \mathbb{N}$ and suppose that $\a \b >1$. Assume that Liggett's condition \eqref{eq:Liggett's condition} holds.
Then there exist $c_0,C_0>0$, depending only on $\a$ and $\b$,  such that for every $\varepsilon\in(0,1)$,
\begin{equation}\label{eq:MixingEquation}
   c_0 \leq  \liminf_{N \rightarrow \infty} \frac{t^N_{\textup{mix}}(\varepsilon)}{N} \leq   \limsup_{N \rightarrow \infty} \frac{t^N_{\textup{mix}}(\varepsilon)}{N} \leq C_0 . 
\end{equation} 
\end{theorem}

Further applications of the above probabilistic characterization of the Bryc polynomials, in particular on the stationary measures of the open KPZ equation and on large deviations for the open ASEP in the shock region, are discussed in Section~\ref{sec:ScopePerspective}.

\subsection{Related work}

The stationary measure of the open ASEP has been extensively studied over the past decades. The phase diagram in Figure~\ref{fig:Phase} of the open ASEP was determined by Liggett in \cite{liggett1975ergodic} under the assumption that the boundary parameters satisfy \eqref{eq:Liggett's condition}. In \cite{derrida1993exact}, Derrida, Evans, Hakim, and Pasquier use the matrix product ansatz to characterize the stationary measure of the open TASEP; see \cite{blythe2007nonequilibrium} for an introduction to the matrix product ansatz technique. For general boundary parameters, Uchiyama, Sasamoto, and Wadati construct in \cite{uchiyama2004asymmetric} a representation for the matrix product ansatz using Askey--Wilson orthogonal polynomials from \cite{askey1985some}. Their connection of the open ASEP with orthogonal polynomials led to a series of results on determining the stationary height fluctuations for the open ASEP through Askey--Wilson processes in the fan region, and through Askey--Wilson signed measures in the shock region  \cite{Bryc2017,bryc2019limit,wang2023askey}. Let us stress that solutions to the matrix product ansatz are in general not unique, and other representations include the Enaud--Derrida construction \cite{enaud2004large}, as well as combinatorial descriptions for example through bi-colored Motzkin paths and staircase tableaux \cite{brak2006combinatorial,corteel2011tableaux}. Moreover, while the matrices in the matrix product ansatz are in general of infinite dimension, special finite-dimensional representations in the shock region  were found by Mallick and Sandow in \cite{mallick1997finite}, closely related to Bernoulli shock measures  \cite{schuetz2022reverse}.

More recently, the stationary measure of the open TASEP was investigated through its two-layer representation in \cite{bryc2024two,bryc2024limitCoex}, following the two-layer Gibbs framework for integrable models on a strip introduced in \cite{barraquand2024stationary}. Similar path representations for the open ASEP appeared earlier, for example in \cite{derrida2002large,brak2006combinatorial}. Of particular relevance to the present work is an extension of the above results by Bryc in \cite{bryc2024stationary} to the open ASEP under Liggett's condition. In \cite{bryc2024stationary}, the recursions on the stationary measure from \cite{liggett1975ergodic,brak2006combinatorial} via basic weight functions are use to identify the polynomial weight encoding the effect of the asymmetry parameter $q$ for the open ASEP.

Complementing the above macroscopic and fluctuation results, it is well-known that the stationary measure of the open ASEP is locally well-approximated by (a mixture of) Bernoulli product measures \cite{liggett1999stochastic}. This approximation of the stationary measure by Bernoulli product measures was extended to larger scales in \cite{nestoridi2024approximating,yang2025limits}. For the three-parameter ASEP, with $\gamma=\delta=0$, \cite{nestoridi2024approximating} establishes in the fan region a Bernoulli product approximation in the bulk, building on a representation of the stationary measure of the open ASEP in terms of reweighted bi-colored Motzkin paths from \cite{brak2006combinatorial}. Similarly, the stationary  measure of the open ASEP near the boundaries is described in \cite{yang2025limits} through Askey--Wilson half-line measures. Finally, let us also mention that the microscopic density profiles and correlations for the stationary measure of the open ASEP were studied for example in \cite{sasamoto2000density,uchiyama2005correlation}.

\subsection{Results in the fan region and the coexistence line}

The present paper is the first in a series of three articles on approximation results for the stationary measure of the open ASEP. For convenience of the reader, let us briefly summaries the main results from the companion works \cite{SchmidYangCoexistence} and \cite{SchmidYangFan}.

In \cite{SchmidYangCoexistence}, we provide sharp approximation results of the stationary measure on the coexistence line of the open ASEP. In this special case $\a=\b>1$, we see the presence of a non-vanishing shock traveling inside the segment. While the overall strategy builds on the two-layer representation from \cite{bryc2024stationary} and the probabilistic characterization of the Bryc polynomials developed in this work, several technical difficulties arise on the way to achieve similarly sharp approximation results as in the high and low density phase. In particular, we require sharp tail estimates for the simultaneous regeneration times of biased random walks, as well as a precise control of the shock location, which we approximate in total variation with respect to the uniform measure on the segment. 

In \cite{SchmidYangFan}, we provide sharp approximation results on the stationary measure of the open ASEP in the fan region. While the analysis and techniques developed in the present paper for the shock region of the open ASEP  are in principle also applicable in the fan region, we take a different route, allowing for general boundary parameters $\alpha,\beta,\gamma,\delta \geq 0$, in particular without assuming  Liggett's condition. In summary, building on the Enaud-Derrida solution to the matrix product ansatz from \cite{enaud2004large}, and its two-layer formulation from \cite{hegde2024large}, we express the height function of a stationary open ASEP using suitable measures on the space of bi-colored Motzkin paths, allowing for sharp approximation results on the stationary measure in terms of certain random walk path measures.

\subsection{Scope and perspective}\label{sec:ScopePerspective}

The results of this paper provide effective approximation results on the
stationary open ASEP in the shock region. Earlier approaches in the shock region were based on finite-dimensional representations of
the matrix product algebra by
Mallick and Sandow~\cite{mallick1997finite}, closely related
to Bernoulli shock measures, and only applied for a very specific subset of boundary parameters; see Section~6 in \cite{nestoridi2024approximating}.  A particular feature of our estimates is that they control joint measures on arbitrary sets of sites under essentially optimal assumptions on the distance from the boundaries. This is particularly useful for fluctuation and moderate
deviation estimates, where the relevant spatial scales vary with
the system size. In this respect, our approach complements also exact
formulas based on Askey--Wilson polynomials and measures \cite{bryc2019limit,wang2023askey}.
Such formulas yield detailed information about stationary
observables, but extracting estimates for a growing collection
of sites from the multiple-integral representations
in~\cite{uchiyama2004asymmetric} requires a delicate analysis. Our methods
instead control the stationary measure directly, at the cost of replacing
exact identities by quantitative approximations.

Our main tool is the probabilistic characterization of the polynomial weights found by Bryc in  \cite{bryc2024stationary}, expressing the polynomial weights arising in the stationary measure of the open ASEP by moments of a birth-death chain. Let us point out that the birth-death chain presentation is a Markovian construction. Hence, the Markov property allows us to compare weights associated with
different path lengths and to localize the effect of the weights. This construction is reminiscent of Liggett's
recursive description of the stationary measure of the open ASEP in \cite{liggett1975ergodic}. However, a precise
relation between the two recursions remains to be
understood.

More broadly, the two-layer representation of the open ASEP fits into the study of stationary
measures through two-layer representations for integrable models
with boundaries. In such representations, summing over an auxiliary
height offset, or zero mode, produces an effective weight on the
remaining path variables. For geometric last passage percolation
and the log-gamma polymer, the resulting stationary measures has an explicit two-layer Gibbs representation
\cite{barraquand2024stationary}. For the open ASEP, the Bryc polynomials are an additional ingredient. The birth-death chain representation makes this
structure useful for asymptotic analysis, but its underlying
algebraic meaning is still to be understood. In particular,
the presence of this polynomial correction factor can not be solely explained by positive
temperature, since for example the log-gamma polymer already provides
a positive-temperature model with an explicit two-layer Gibbs representation without polynomial correction terms as shown in \cite{barraquand2024stationary}. 
Several further questions arise from this perspective. 
The open stochastic six-vertex model from \cite{yang2022stationary} is a
natural setting in which to investigate the appearance of Bryc polynomials and
possible analogues of our approximation results. Moreover, extensions to
higher-spin vertex models and related
harmonic models are of independent interest and left to future work.

Controlling the Bryc polynomials suggests a route for transferring
stationary asymptotic results from the open TASEP to the open ASEP.
For the totally asymmetric case, the two-layer representation takes a particularly simple form, as shown in \cite{bryc2024two}, and can be studied using Denisov's random walk decomposition \cite{bryc2024limitCoex}.
Moreover, it benefits from its connections with
last passage percolation from \cite{rost1981non}, for which additional probabilistic tools
are available; see Section 7 in \cite{nestoridi2024approximating} for further discussions.
Our probabilistic characterization also allows for various natural applications. In
\cite{hegde2024large}, a mild control on the Bryc
polynomials is used for large deviation results on the open ASEP in the shock region. 
Another application concerns the weakly asymmetry regime, where the asymmetry parameter $q$
depends on the system size. The weakly asymmetric regime is investigated
in~\cite{SchmidYangOpenKPZ}, where the probabilistic description of the Bryc polynomials is used to characterize the stationary measures of the open KPZ equation.

Finally, let us mention two questions on the scope and precision of the
present estimates. The first is whether the sharp approximation
results can be extended beyond Liggett's condition. The second
is whether one can identify the optimal exponential decay rates
in place of the bounds proved here. The finer subdivisions of
the phase diagram appearing in the analysis of stationary density
profiles~\cite{sasamoto2000density} suggest that these rates may
contain information not visible in our current estimates.
Determining their dependence on the parameters is promising to connect 
our approximation results with a more refined
description of boundary effects.

\subsection{Overview of the proof ideas}

Our main results are two-folded. 
First, we introduce a general framework in order to compare the stationary measure of the open TASEP to the stationary measure of the open ASEP. To this end, we consider the polynomial factors from \cite{bryc2024stationary} in the Radon-Nikodym derivative with respect to the two-layer representation of the open TASEP, to which we refer as Bryc polynomials, indexed by compositions of $N$ derived from the local times of the underlying random walk paths. We provide a probabilistic characterization of the Bryc polynomials.
More precisely, we show that the Bryc polynomials can be expressed by the moment generating function for a certain inhomogeneous birth-death chain run for $N$ steps. Here, the $q$-difference operators in Bryc's characterization in \cite{bryc2024stationary} give the transition probabilities of our birth-chain, allowing us to conclude this characterization by induction. 
In return, we study this birth-death chain in depth and control its exponential moments, and hence the coefficients in the Bryc polynomials, for typical compositions. 

For the analysis of the Bryc polynomials, we require a fine understanding of the height function of the open TASEP where $\gamma=\delta=q=0$. To this end, we use the Denisov decomposition determined in  \cite{bryc2024limitCoex} together with classical random walk estimates. The major part lies in justifying that the Bryc polynomial is well-behaved. In particular, we require precise lower bounds on the local times for the open TASEP height function, and sharp estimates on regeneration times for biased random walks. These regeneration times allow us to give a effective lower bounds on the steps for the death rate of the birth-death chain after a suitable approximation by infinite biased random walk measures. In particular, this allows us to provide Gaussian tail bounds on the involved birth-death chains for different starting states.

As our main application, we show that the aforementioned characterization of the Bryc polynomials can be used to obtain a sharp quantitative approximation of the stationary measure in terms of homogeneous Bernoulli product measures for the open ASEP in the high and low density phase in the shock region. We prove our results for the two-layer representation of the open TASEP, and then control the Bryc polynomials, allowing us to transfer the results to an approximation of the stationary measure of the open ASEP. 
Moreover, we demonstrate that our approximation results can be used to resolve questions on the number of particles and on the open ASEP with light particles from \cite{SY:openLight}. In these applications, a key tool is a sharp comparison between the stationary measure of the open ASEP and Bernoulli product measures, allowing to achieve the results previously only known in the fan region now also in the shock region. 

\subsection{Outline of the paper}

In Section~\ref{sec:TwoLayer}, we prove the birth-death chain characterization for the Bryc polynomials from Theorem \ref{thm:prob interpretation of Bryc polynomials}. In Section~\ref{sec:Prelims}, we provide preliminaries on the two-layer representation of the open TASEP and introduce suitable replacement measures for the construction of the Bryc polynomials. In Section \ref{sec:MomentBryc}, we show Proposition~\ref{prop:moment bound Bryc MAIN} stating that all finite moments for a typical Bryc polynomial are uniformly bounded. In Section~\ref{sec:Approx}, we show that the stationary measure of the open ASEP is well approximated by homogeneous Bernoulli product measures; see Theorem \ref{thm:ApproxHighLow}. In Section~\ref{sec:FurtherApplications}, we prove moderate deviations estimates on the number of particles under the stationary open ASEP (Theorem \ref{thm:moderate deviation}) as well as approximation and mixing time bounds for the open ASEP with light particles (Theorem \ref{thm:mainHighLowLightIntro} and Theorem \ref{thm:MainMixingTimeIntro}). In Appendix~\ref{sec:Appendix}, we provide sharp tail estimates on regeneration times of biased random walks.

\subsection*{Acknowledgments} 
We thank Ivan Corwin, Nina Gantert and Mehtaab Sawhney for helpful discussions. 
D.S. was partially supported by the Packard Foundation via Amol Aggarwal's Packard Fellowships for Science and Engineering.
Z.Y. was partially supported by the Miller Institute for Basic Research in Science at the University of California, Berkeley, Ivan Corwin's NSF grants DMS-1811143, DMS-2246576, and the Fernholz Foundation's ``Summer Minerva Fellows'' program.
Both authors were supported by Ivan Corwin's Simons Foundation Grant 929852.

\subsection*{AI declaration} 

This project was initiated in fall 2024, and the proofs were written down by the end of July 2026. The authors used AI to assist with editing the manuscript, performing some computations, and checking arguments. All proof ideas in this article were developed by the authors. 
%
\section{A birth-death chain characterization of the Bryc polynomials}  \label{sec:TwoLayer}

The goal of this section is to provide a probabilistic characterization of the Bryc polynomial from \cite[(1.4)]{bryc2024stationary}. 
We begin by introducing some notation.
We call $f(z)$ a rational function (over $\mathbb{R}$) if it can be written as $f(z)=g(z)/h(z)$, where $g(z)$ and $h(z)$ are polynomials in $z$ with coefficients in $\mathbb{R}$, and $h\neq0$. Note that, as a function in $z$, $f(z)$ is defined on $\mathbb{R}\setminus \{p_1,\dots,p_k\}$, where $k\in\mathbb{Z}_{\geq0}$ and $p_1,\dots,p_k\in\mathbb{R}$ are the real roots of $h(z)$.
For $q\in(0,1)$ and a rational function $f(z)$, we will write $\mathcal{D}_qf$ and $\mathcal{E}_qf $ to denote the rational functions 
\[
(\mathcal{D}_qf)(z) =\frac{f(z)-f(qz)}{z(1-q)}\quad\mbox{ and }\quad (\mathcal{E}_qf)(z) =\frac{f(z)-qf(qz)}{ 1-q } .
\]
We will view $\mathcal{D}_q$ and $\mathcal{E}_q$ as operators acting on the space of rational functions, and we mention that $\mathcal{D}_q$ is known by the name of the Jackson derivative or the $q$-differentiation operator; see Jackson \cite{jackson1909xi}, and, for example, monographs \cite[Chapter 1]{koekoek2010hypergeometric} and \cite{gasper2011basic}. Note that the operator $\mathcal{E}_q$ can be seen as $\mathcal{D}_q\circ z$, where $z$ denotes the operator to multiply $z$ to a rational function $f(z)$. Here and below, for $n\in\mathbb{Z}_{\geq0}\cup\{\infty\}$, we use the $q$-Pochhammer symbol $(z;q)_n=\prod_{j=0}^{n-1}\,(1-z q^j)$. 

\begin{definition}\label{def:Bryc poly}
For $q\in[0,1)$ and a composition $\bs=(\sigma_0,\dots,\sigma_r)$ of $N+1$, where $N\in\mathbb{Z}_{\geq0}$, we define a rational function $W_{\bs}(z)$ as follows.
When $q=0$ we set $W_{\bs}(z)=1$, and when $q\in(0,1)$ we set
    \begin{equation}\label{eq:original def of Bryc polynomial}
       W_{\bs}(z)=  \frac{(1-q)^{N+1} }{ (q;q)_{N+1}}  (z;q)_{N+2}  \cdot \left( \underbrace{\mathcal{E}_q\circ\dots \circ\mathcal{E}_q}_{\sigma_0-1}\circ\mathcal{D}_q \right)\circ\dots\circ \left( \underbrace{\mathcal{E}_q\circ\dots \circ\mathcal{E}_q}_{\sigma_r-1}\circ\mathcal{D}_q\right)   \left[\frac{1}{(z;q)_1 }\right].
    \end{equation}  
Here, we apply $\sigma_r+\dots+\sigma_0=N+1$ many operators $\mathcal{D}_q$ or $\mathcal{E}_q$ sequentially on $1/(z;q)_1=1/(1-z)$ and finally multiply with the prefactors $(z;q)_{N+2}$ and $(1-q)^{N+1}/ (q;q)_{N+1}$.
An empty composition of operators is the identity.
\end{definition}
\begin{remark}\label{rmk:prefactor differ}
    We mention that our Definition \ref{def:Bryc poly} of $W_{\bs}(z)$ differs from the $w_{\bs}(z)$ defined in \cite[(1.4)]{bryc2024stationary} by the prefactor $(1-q)^{N+1}/(q;q)_{N+1}$. We choose to include this prefactor for convenience of notation.
\end{remark}
\begin{remark}\label{rmk:Bryc is polynomial}
The rational function $W_{\bs}(z)$ as in Definition \ref{def:Bryc poly} is actually a polynomial in $z$ with non-negative coefficients depending on $q\in[0,1)$, and with constant term $W_{\bs}(0)>0$. This is not obvious from Definition~\ref{def:Bryc poly}, but is shown in \cite[Lemma 2.1]{bryc2024stationary}. We will call $W_{\bs}(z)$ a Bryc polynomial with respect to $\bs$.  We will provide an alternative proof of this fact in Corollary~\ref{cor:Bryc is polynomial}, which follows from our birth-death chain characterization.
\end{remark}

We next prove Theorem \ref{thm:prob interpretation of Bryc polynomials}. For clarity, we restate Theorem \ref{thm:prob interpretation of Bryc polynomials} below.

\begin{theorem}\label{thm:prob interpretation of Bryc polynomials again}
Let $q\in[0,1)$, $N\in\mathbb{Z}_{\geq0}$ and $\bs=(\sigma_0,\dots,\sigma_r)$, $r\in\mathbb{Z}_{\geq0}$ be a composition of $N+1$. We have
    \begin{equation}\label{eq:rep of Bryc polynomial again}
    W_{\bs}(z)=\mathbb{E}\left[ z^{X^{\boldsymbol{\mathsf{S}}^{\bs}}(N+1)}\right].
    \end{equation}
Here, on the left side, $W_{\bs}(z)$ is given by Definition \ref{def:Bryc poly}; and on the right side,
$\boldsymbol{\mathsf{S}}^{\bs}$ is the $0$-$1$ sequence associated to the composition $\bs$ as in Definition \ref{def:associated 01 seq}, and $X^{\boldsymbol{\mathsf{S}}^{\bs}}=\{X^{\boldsymbol{\mathsf{S}^{\bs}}}(i)\}_{i=0}^{N+1}$ is the birth-death Markov chain indexed by this $0$-$1$ sequence $\boldsymbol{\mathsf{S}}^{\bs}$ as in Definition \ref{def:characterization measure}. 
\end{theorem}
\begin{proof} 
When $q=0$, we can observe that $X^{\boldsymbol{\mathsf{S}}^{\bs}}(i)=0$ for $i\in\llbracket0,N+1\rrbracket$, hence \eqref{eq:rep of Bryc polynomial again} holds since both sides are equal to $1$. In the rest of this proof, we assume $q\in(0,1)$.
We claim that \eqref{eq:original def of Bryc polynomial} can be written alternatively as
\begin{equation}\label{eq:etrtr}
    W_{\bs}(z)= \left(\underbrace{\mathsf{E}_q^{(N+1)}\circ\dots \circ\mathsf{E}_q^{(N-\sigma_0+3)}}_{\sigma_0-1}\circ\mathsf{D}_q^{(N-\sigma_0+2)}\right)\circ\dots\circ\left(\underbrace{\mathsf{E}_q^{(\sigma_r)}\circ\dots\circ \mathsf{E}_q^{(2)}}_{\sigma_r-1}\circ\mathsf{D}_q^{(1)}\right)[1],
\end{equation} 
where for each $i\in\mathbb{N}$, we define the operators $\mathsf{D}_q^{(i)}$ and $\mathsf{E}_q^{(i)}$ acting on rational functions by
\begin{equation}\label{eq:D and E alternative}
 \left(\mathsf{D}_q^{(i)}f\right)(z) =\frac{1-q}{1-q^i}(z;q)_{i+1}\cdot\mathcal{D}_q\left[\frac{f(z)}{(z;q)_i}\right]  \quad\mbox{ and }\quad
 \left(\mathsf{E}_q^{(i)}f\right)(z) =\frac{1-q}{1-q^i}(z;q)_{i+1}\cdot\mathcal{E}_q\left[\frac{f(z)}{(z;q)_i}\right] .
\end{equation} 
We explain the proof of \eqref{eq:etrtr} as follows. Define $F_0=1$, and reading the operators in \eqref{eq:original def of Bryc polynomial} sequentially from right to left, let $T_i\in\{\mathcal{D}_q,\mathcal{E}_q\}$ be the $i$-th operator for $i\in\llbracket1,N+1\rrbracket$. Set 
\[
F_i(z)= \frac{(1-q)^{i} }{ (q;q)_{i}}  (z;q)_{i+1}  \cdot T_i\circ\dots\circ T_1 \left[\frac{1}{(z;q)_1 }\right].
\]
One can check from the definition of the operators $\mathcal{D}_q$ and $\mathcal{E}_q$ that  $F_i= \mathsf{D}_q^{(i)}F_{i-1}$ if $T_i=\mathcal{D}_q$ and similarly  $F_i= \mathsf{E}_q^{(i)}F_{i-1}$ if $T_i=\mathcal{E}_q$. Taking $i=N+1$ we conclude \eqref{eq:etrtr}.

To prove \eqref{eq:rep of Bryc polynomial again}, it suffices to show 
\begin{equation}\label{eq:only need to show rep of Bryc}
 \mathbb{E}\left[ z^{X^{\boldsymbol{\mathsf{S}}^{\bs}}(N+1)}\right]=\left(\underbrace{\mathsf{E}_q^{(N+1)}\circ\dots \circ\mathsf{E}_q^{(N-\sigma_0+3)}}_{\sigma_0-1}\circ\mathsf{D}_q^{(N-\sigma_0+2)}\right)\circ\dots\circ\left(\underbrace{\mathsf{E}_q^{(\sigma_r)}\circ\dots\circ \mathsf{E}_q^{(2)}}_{\sigma_r-1}\circ\mathsf{D}_q^{(1)}\right)[1].
\end{equation}
For each $j\in\mathbb{Z}_{\geq0}$, we compute from \eqref{eq:D and E alternative} that
\begin{equation}\label{eq:Jackson 1}
\begin{split}
\mathsf{D}_q^{(i)}\left[z^j\right]&=\frac{1-q}{1-q^i}\mathcal{D}_q\left[\frac{z^j}{(z;q)_i}\right](z;q)_{i+1} 
=\frac{1-q}{1-q^i}\frac{\frac{z^j}{(z;q)_i}-\frac{z^jq^j}{(zq;q)_i}}{z(1-q)}(z;q)_{i+1} \\
&=\frac{z^j(1-zq^i)-z^jq^j(1-z)}{z(1-q^i)}  =\frac{q^j-q^i}{1-q^i}z^j+ \frac{1-q^j}{1-q^i}z^{j-1},
\end{split}
\end{equation} 
and
\begin{equation}\label{eq:Jackson 2}
\begin{split}
\mathsf{E}_q^{(i)}\left[z^j\right]&=\frac{1-q}{1-q^i}\mathcal{E}_q\left[\frac{z^j}{(z;q)_i}\right](z;q)_{i+1} =\frac{1-q}{1-q^i}\frac{\frac{z^j}{(z;q)_i}-\frac{q^{j+1}z^j}{(zq;q)_i}}{1-q}(z;q)_{i+1} \\
&=\frac{z^j(1-zq^i)-q^{j+1}z^j(1-z)}{1-q^i} =\frac{q^{j+1}-q^i}{1-q^i}z^{j+1}+ \frac{1-q^{j+1}}{1-q^i}z^{j}.
\end{split}
\end{equation} 

We recall from Definitions \ref{def:characterization measure} and \ref{def:associated 01 seq} that $X^{\boldsymbol{\mathsf{S}^{\bs}}}(0)=X^{\boldsymbol{\mathsf{S}^{\bs}}}(1)=0$ and for each $i\in\llbracket1,N\rrbracket$ and $j\in\llbracket0,i-1\rrbracket$, 
 if $\mathsf{S}^{\bs}(i)=1$, we have $X^{\boldsymbol{\mathsf{S}^{\bs}}}(i+1)-X^{\boldsymbol{\mathsf{S}^{\bs}}}(i)\in\{-1,0\}$ with 
    \begin{equation}\label{eq:dec}
    \mathbb{P}\left(X^{\boldsymbol{\mathsf{S}^{\bs}}}(i+1)=j-1 \Big{|}X^{\boldsymbol{\mathsf{S}^{\bs}}}(i)=j\right)=\frac{1-q^{j}}{1-q^{i+1}},
   \end{equation}
 and if $\mathsf{S}^{\bs}(i)=0$, we have $X^{\boldsymbol{\mathsf{S}^{\bs}}}(i+1)-X^{\boldsymbol{\mathsf{S}^{\bs}}}(i)\in\{0,1\}$ with 
    \begin{equation}\label{eq:inc}
    \mathbb{P}\left(X^{\boldsymbol{\mathsf{S}^{\bs}}}(i+1)=j+1 \Big{|}X^{\boldsymbol{\mathsf{S}^{\bs}}}(i)=j\right)=\frac{q^{j+1}-q^{i+1}}{1-q^{i+1}},
    \end{equation} 
where $X^{\boldsymbol{\mathsf{S}}^{\bs}}(i)\in\llbracket0,i-1\rrbracket$ for all $i\in\llbracket1,N+1\rrbracket$, and
 $\boldsymbol{\mathsf{S}}^{\bs}=(\mathsf{S}^{\bs}(0),\dots,\mathsf{S}^{\bs}(N))\in\{0,1\}^{N+1}$ is given by
\begin{equation}\label{eq:01seq again}
  \mathsf{S}^{\bs}(i)=\begin{cases}
        1 &\quad\mbox{ if }i\in\{\sigma_r,\sigma_r+\sigma_{r-1},\dots,\sigma_r+\dots+\sigma_1\},\\  
        0 &\quad\mbox{ otherwise}.
    \end{cases} 
\end{equation}

We next prove \eqref{eq:only need to show rep of Bryc} by induction over $N\in\mathbb{Z}_{\geq0}$. When $N=0$, the equation is $\mathbb{E}[ z^{X^{ \boldsymbol{\mathsf{S}^{\bs}}}(1) }]=\mathsf{D}_q^{(1)}[1]$, which holds since both sides are equal to $1$. Assume now that $N\geq1$ and that the equation holds for $N-1$. We split the proof into two cases.\\

\noindent\textbf{Case 1.} Assume $\sigma_0>1$.  Set $\widetilde{\bs}=(\sigma_0-1,\sigma_1,\dots,\sigma_r)$, which is a composition of $N$. From \eqref{eq:01seq again}, we have $\mathsf{S}^{\widetilde{\bs}}(i)=\mathsf{S}^{\bs}(i)$ for $i\in\llbracket0,N-1\rrbracket$, and $\mathsf{S}^{\bs}(N)=0$. Hence the two Markov chains agree in law up to time $N$, and we have
$X^{\boldsymbol{\mathsf{S}}^{\bs}}(N)\stackrel{d}{=}X^{\boldsymbol{\mathsf{S}}^{\widetilde{\bs}}}(N)$. 
Conditioning on $X^{\boldsymbol{\mathsf{S}}^{\bs}}(N)=j$ and using $\mathsf{S}^{\bs}(N)=0$, from \eqref{eq:inc} we get
\[
\mathbb{E}\left[z^{X^{\boldsymbol{\mathsf{S}}^{\bs}}(N+1)}\,\middle|\,X^{\boldsymbol{\mathsf{S}}^{\bs}}(N)=j\right]
=\frac{q^{j+1}-q^{N+1}}{1-q^{N+1}}z^{j+1}+\frac{1-q^{j+1}}{1-q^{N+1}}z^j
=\mathsf{E}_q^{(N+1)}[z^j] \quad \mbox{ for } j\in\llbracket0,N-1\rrbracket,
\]
where the last equality follows from \eqref{eq:Jackson 2}. Therefore, by linearity of the operator $\mathsf{E}_q^{(N+1)}$,
\begin{equation}\label{eq:gergbvet1}
\mathbb{E}\left[z^{X^{\boldsymbol{\mathsf{S}}^{\bs}}(N+1)}\right]
=\mathsf{E}_q^{(N+1)}\left(\mathbb{E}\left[z^{X^{\boldsymbol{\mathsf{S}}^{\bs}}(N)}\right]\right)
=\mathsf{E}_q^{(N+1)}\left(\mathbb{E}\left[z^{X^{\boldsymbol{\mathsf{S}}^{\widetilde{\bs}}}(N)}\right]\right).
\end{equation}
By the induction hypothesis applied to $\widetilde{\bs}$, the last expectation equals the right side of \eqref{eq:only need to show rep of Bryc} with the leftmost operator $\mathsf{E}_q^{(N+1)}$ removed. Hence the right side of \eqref{eq:gergbvet1} equals the right side of \eqref{eq:only need to show rep of Bryc}. \\

\noindent\textbf{Case 2.} Assume $\sigma_0=1$.  Set $\widetilde{\bs}=(\sigma_1,\dots,\sigma_r)$, which is a composition of $N$. From \eqref{eq:01seq again}, we have $\mathsf{S}^{\widetilde{\bs}}(i)=\mathsf{S}^{\bs}(i)$ for $i\in\llbracket0,N-1\rrbracket$, and now $\mathsf{S}^{\bs}(N)=1$. Hence the two Markov chains agree in law up to time $N$, and we have
$X^{\boldsymbol{\mathsf{S}}^{\bs}}(N)\stackrel{d}{=}X^{\boldsymbol{\mathsf{S}}^{\widetilde{\bs}}}(N)$. 
Conditioning on $X^{\boldsymbol{\mathsf{S}}^{\bs}}(N)=j$ and using $\mathsf{S}^{\bs}(N)=1$, from \eqref{eq:dec} we get
\[ 
\mathbb{E}\left[z^{X^{\boldsymbol{\mathsf{S}}^{\bs}}(N+1)}\,\middle|\,X^{\boldsymbol{\mathsf{S}}^{\bs}}(N)=j\right]
=\frac{q^{j}-q^{N+1}}{1-q^{N+1}}z^{j}+\frac{1-q^{j}}{1-q^{N+1}}z^{j-1}
=\mathsf{D}_q^{(N+1)}[z^j] \quad \mbox{ for } j\in\llbracket0,N-1\rrbracket,
\]
where the last equality follows from \eqref{eq:Jackson 1}. Therefore, by linearity of $\mathsf{D}_q^{(N+1)}$,
\begin{equation}\label{eq:gergbvet2}
\mathbb{E}\left[z^{X^{\boldsymbol{\mathsf{S}}^{\bs}}(N+1)}\right]
=\mathsf{D}_q^{(N+1)}\left(\mathbb{E}\left[z^{X^{\boldsymbol{\mathsf{S}}^{\bs}}(N)}\right]\right)
=\mathsf{D}_q^{(N+1)}\left(\mathbb{E}\left[z^{X^{\boldsymbol{\mathsf{S}}^{\widetilde{\bs}}}(N)}\right]\right).
\end{equation}
By the induction hypothesis applied to $\widetilde{\bs}$, the last expectation equals the right side of \eqref{eq:only need to show rep of Bryc} with the leftmost operator $\mathsf{D}_q^{(N+1)}$ removed. Hence the right side of \eqref{eq:gergbvet2} equals the right side of \eqref{eq:only need to show rep of Bryc}. 

We conclude the proof of \eqref{eq:only need to show rep of Bryc} for all $N\in\mathbb{Z}_{\geq0}$ by induction.
\end{proof}

Next, we collect some basic properties of $W_{\bs}(z)$. These were already proven in \cite[Lemma 2.1]{bryc2024stationary} in slightly different notation. Since we rely on this result multiple times, we provide an alternative proof here. This proof follows directly from the birth-death chain characterization of $W_{\bs}(z)$ given in Theorem \ref{thm:prob interpretation of Bryc polynomials again}. Let us mention that $W_{\bs}(0)$ has an explicit formula, i.e.~$W_{\bs}(0)=(1-q)^{N+1}/(q;q)_{N+1}\cdot\prod_{j=0}^r[j+1]_q^{\sigma_j}$; see \cite[Section 1.2]{bryc2024stationary}. 

\begin{corollary}\label{cor:Bryc is polynomial}
    Let $q\in[0,1)$, $N\in\mathbb{Z}_{\geq0}$ and $\bs=(\sigma_0,\dots,\sigma_r)$, $r\in\mathbb{Z}_{\geq0}$ be a composition of $N+1$. The rational function $W_{\bs}(z)$ as in Definition \ref{def:Bryc poly} is a polynomial in $z$ with non-negative coefficients depending on $q$, with $W_{\bs}(0)>0$. 
    When $q\in(0,1)$, the polynomial $W_{\bs}(z)$ has degree $N-r$. When $q=0$, recall $W_{\bs}(z)=1$. 
\end{corollary} 
\begin{proof} 
    Since $X^{\boldsymbol{\mathsf{S}}^{\bs}}(N+1)$ is a random variable taking values in $\llbracket0,N\rrbracket$, it is clear from Theorem \ref{thm:prob interpretation of Bryc polynomials again} (i.e., $W_{\bs}(z)=\mathbb{E} [ z^{X^{\boldsymbol{\mathsf{S}}^{\bs}}(N+1)} ]$) that $W_{\bs}(z)$ is a polynomial in $z$ with non-negative coefficients.  
For $q\in(0,1)$, we observe from \eqref{eq:01seq again} that $\mathsf{S}^{\bs}(i)=0$ for exactly $N-r$ many $i\in\llbracket1,N\rrbracket$. 
Since $X^{\boldsymbol{\mathsf{S}}^{\bs}}(0)=X^{\boldsymbol{\mathsf{S}}^{\bs}}(1)=0$, and $X^{\boldsymbol{\mathsf{S}}^{\bs}}(i+1)=X^{\boldsymbol{\mathsf{S}}^{\bs}}(i)+1$ is possible only when $\mathsf{S}^{\bs}(i)=0$, we have $X^{\boldsymbol{\mathsf{S}}^{\bs}}(N+1)\leq N-r$. 
When $\mathsf{S}^{\bs}(i)=0$, since $X^{\boldsymbol{\mathsf{S}}^{\bs}}(i)\in\llbracket0,i-1\rrbracket$ for $i\geq1$, we have $(q^{X^{\boldsymbol{\mathsf{S}}^{\bs}}(i)+1}-q^{i+1})/(1-q^{i+1})\in(0,1)$. By \eqref{eq:inc}, there is a positive probability such that $X^{\boldsymbol{\mathsf{S}}^{\bs}}(i+1)=X^{\boldsymbol{\mathsf{S}}^{\bs}}(i)+1$ and a positive probability such that $X^{\boldsymbol{\mathsf{S}}^{\bs}}(i+1)=X^{\boldsymbol{\mathsf{S}}^{\bs}}(i)$. When $\mathsf{S}^{\bs}(i)=1$, since $X^{\boldsymbol{\mathsf{S}}^{\bs}}(i)\in\llbracket0,i-1\rrbracket$ for $i\geq1$, we have $(1-q^{X^{\boldsymbol{\mathsf{S}}^{\bs}}(i)})/(1-q^{i+1})<1$. By \eqref{eq:dec}, there is a positive probability such that $X^{\boldsymbol{\mathsf{S}}^{\bs}}(i+1)=X^{\boldsymbol{\mathsf{S}}^{\bs}}(i)$. We conclude that the events $X^{\boldsymbol{\mathsf{S}}^{\bs}}(N+1)=N-r$ and $X^{\boldsymbol{\mathsf{S}}^{\bs}}(N+1)=0$ both have positive probability. Therefore, $\deg W_{\bs}(z)=N-r$ and $W_{\bs}(0)>0$.
\end{proof} 
%
\section{Preliminaries on the two-layer representation of the open TASEP}\label{sec:Prelims}

In this section, we will focus on the two-layer representation of the open TASEP. 
In this case, the Bryc polynomial is simply $W_{\bs}(z)=1$ and the respective two-layer measure was studied in \cite{bryc2024two}, i.e., the measure $\mathbb{P}^{(q,\a,\b)}$ given by \eqref{eq:PNab} becomes 
\begin{equation}\label{eq:PNab TASEP}
\mathbb{P}^{(0,\a,\b)}(\bg)=\frac{1}{Z_N^{(0,\a,\b)}}\frac{2^{H(\bg)}\b^{\gamma_N}}{(\a\b)^{\min(\bg)}} \quad\mbox{ for }\bg\in\lrw_N. 
\end{equation}
The probability measure $\mathbb{P}^{(0,\a,\b)}$ as in \eqref{eq:PNab TASEP} has been characterized in \cite[Section 3]{bryc2024limitCoex}, where the authors established a representation of this measure in terms of certain explicit random variables. This representation, which we recall below as Lemma~\ref{lem:decomposition TASEP}, was inspired by Denisov’s decomposition of random walks \cite{denisov1983random}.

\begin{lemma}[cf. Section 3 in \cite{bryc2024limitCoex}, especially Proposition 3.2 and (3.1), (3.2) therein]\label{lem:decomposition TASEP}
We define the observable $T_N$ on the space $\lrw_N$ by
\begin{equation}\label{def:TN}
    T_N(\bg):=\min\{i\in\llbracket0,N\rrbracket:\gamma_i=\min(\bg)\}\quad\mbox{ for } \bg\in\lrw_N.
\end{equation}
Assume $\a,\b>0$ and consider the probability measure $\mathbb{P}^{(0,\a,\b)}$ on $\lrw_N$ as in \eqref{eq:PNab TASEP}. Then 
\begin{equation}\label{eq: dist of TN}
    \mathbb{P}^{(0,\a,\b)}(T_N = m) = \begin{cases}
\displaystyle\frac{\a}{4C_{N,\a,\b}}w_{\a}^{m-1}w_{\b}^{N-m}p^{(\a)}_{m-1}p^{(\b)}_{N-m}  & \quad\mbox{ if }
m\in\llbracket1,N\rrbracket,\\ 
\displaystyle\frac{1}{C_{N,\a,\b}}w_{\b}^Np^{(\b)}_N  &\quad\mbox{ if } m = 0.
\end{cases}
\end{equation}
Here, we write
\begin{align}
       C_{N,\a,\b} &= w_{\b}^Np^{(\b)}_N+ \sum_{m=1}^{N}\frac{\a}{4} w_{\a}^{m-1}w_{\b}^{N-m}p_{m-1}^{(\a)}p_{N-m}^{(\b)} ,  \label{eq:C N a b}\\
       w_x &= \frac{1}{4}\left(x+x^{-1}+2\right)\quad\mbox{ for }x>0,\label{eq:def of wx} 
\end{align} 
and we denote by $p_{\ell}^{(x)}$ the probability that a $\nu_x$-random walk started at $0$ does not reach level $-1$ during the first $\ell$ steps, where $x>0$ and $\ell\in\mathbb{Z}_{\geq0}$.
In particular, \(\mathbb{P}^{(0,\a,\b)}(T_N=m)>0\) for every $m\in\llbracket0,N\rrbracket$.

Furthermore, for every $m\in\llbracket0,N\rrbracket$, conditionally on \(\{T_N=m\}\), the path $(\gamma_{m+i}-\gamma_m )_{i\in\llbracket0,N-m\rrbracket}$ has the law of a \(\nu_{\b}\)-random walk started at \(0\), conditioned to stay non-negative up to step \(N-m\). 
If \(m\geq1\), then conditionally on \(\{T_N=m\}\) one also has $\gamma_{m-1}-\gamma_m=1$,
and the reversed left path $(\gamma_{m-1-i}-\gamma_{m-1} )_{i\in\llbracket0,m-1\rrbracket}$
has the law of a \(\nu_{\a}\)-random walk started at \(0\), conditioned to stay non-negative up to step \(m-1\).  For \(m\geq1\), this reversed left path and the right path $(\gamma_{m+i}-\gamma_m )_{i\in\llbracket0,N-m\rrbracket}$ are conditionally independent given \(\{T_N=m\}\).  
\end{lemma}
\begin{proof} 
Under the bijection $\omega_i=\gamma_i-\gamma_{i-1}$, $i\in\llbracket1,N\rrbracket$, the law $\mathbb{P}^{(0,\a,\b)}$ on $\lrw_N$ coincides with the law
$P_{\mathrm{RW},N,\a,\b}$ on $\{-1,0,1\}^N$ as in \cite[(2.2)]{bryc2024limitCoex}.
The result follows from \cite[Proposition 3.2]{bryc2024limitCoex}. Specifically, our \eqref{eq: dist of TN}, \eqref{eq:C N a b} are exactly
\cite[(3.1), (3.2)]{bryc2024limitCoex}, and the conditional laws of the left and right paths follow from the construction
$S_N=L_{T_N-1}\odot R_{N-T_N}$, when $T_N\in\llbracket1,N\rrbracket$ and $S_N=R_N$ if $T_N=0$. Here, $\odot$ denotes the concatenation operation, and $T_n$ for $n\in\mathbb{N}$ and $L_m$ and $R_m$ for $m\in\mathbb{Z}_{\geq0}$ are independent processes, while for each $m$, $L_m$ (respectively, $R_m$) is a $\nu_{\a}$ (respectively, $\nu_{\b}$) random walk started from $0$ and conditioned to stay non-negative up to step $m$.
\end{proof}

Next, we prove a standard monotone coupling of $\nu_{\b}$-random walks, which we will use repeatedly in the sequel. 
\begin{lemma}\label{lem:basic hitting monotone coupling}
    Assume $\b>1$, $N\in\mathbb{N}\cup\{\infty\}$ and $h\in\mathbb{Z}_{\geq1}$. Let $\{S_i\}_{i=0}^N$ be the $\nu_{\b}$-random walk started at $S_0=0$. 
    Let 
    \begin{equation*}
        E_{N}^h=\{S_i\geq-h+1\mbox{ for all } i\in\llbracket0,N\rrbracket\} .
    \end{equation*} Then 
    $\mathbb{P}(E_{N}^h)\downarrow\mathbb{P}(E_{\infty}^h)=1-\b^{-2h}$ as $N\to\infty$.
    Denote by $\{S_i'\}_{i=0}^N$ the process $\{S_i\}_{i=0}^N$ conditioned on $E_{N}^{h}$. There exists a coupling such that almost surely $S'_i\geq S_i$ for all $i\in\llbracket0,N\rrbracket$.
    Furthermore, fix $\kappa\in(0,(\b-1)/(\b+1))$, then there exist constants $C,c>0$ depending only on $\b$ and $\kappa$ such that for each $N\geq m\geq1$, we have
    $\mathbb{P}(S'_{m}\leq\kappa m)\leq \mathbb{P}(S_{m}\leq\kappa m)\leq Ce^{-cm}$.
\end{lemma}
\begin{proof}
    The zero steps do not affect the event of ever hitting a level. Hence the hitting probabilities for $\{S_i\}_{i=0}^{\infty}$ is the same as that for the biased nearest-neighbor random walk with upward probability $\b^2/(1+\b^2)$ and downward probability $1/(1+\b^2)$. By \cite[Theorem 4.8.9 (c)]{durrett2019probability} we conclude the proof of the first statement on $E_{N}^h$. For the second statement, note that by Doob's h-transform, for each $x\in\mathbb{Z}_{\geq-h+1}$ and $i\in\llbracket1,N\rrbracket$,  
\begin{align*}
    &\mathbb P(S_i'=x+1\mid S_{i-1}'=x)=\nu_{\b}(1)\cdot\frac{\mathbb{P}_{x+1}(\tau_{-h}>N-i)}{\mathbb{P}_x(\tau_{-h}>N-i+1)},\\
    &\mathbb P(S_i'=x\mid S_{i-1}'=x)=\nu_{\b}(0)\cdot\frac{\mathbb{P}_{x}(\tau_{-h}>N-i)}{\mathbb{P}_x(\tau_{-h}>N-i+1)},\\
    &\mathbb P(S_i'=x-1\mid S_{i-1}'=x)=\nu_{\b}(-1)\cdot\frac{\mathbb{P}_{x-1}(\tau_{-h}>N-i)}{\mathbb{P}_x(\tau_{-h}>N-i+1)},
\end{align*}
where $\mathbb{P}_x$ denotes the law of a $\nu_{\b}$-random walk started at $x$, and $\tau_{-h}$ is the hitting time of $-h$. Note 
\[
\mathbb{P}_{x+1}(\tau_{-h}>N-i)\geq \mathbb{P}_{x}(\tau_{-h}>N-i)\geq \mathbb{P}_{x}(\tau_{-h}>N-i+1),
\]
where the first step follows from the coupling using the same increments.
Hence $\mathbb P(S_i'=x+1\mid S_{i-1}'=x)\geq \nu_{\b}(1)$. Similarly, we have $\mathbb P(S_i'=x-1\mid S_{i-1}'=x)\leq \nu_{\b}(-1)$. The monotone coupling between $\{S_i'\}_{i=0}^N$ and $\{S_i\}_{i=0}^N$ can be constructed using the quantile (monotone) coupling of Bernoulli clocks at each individual step $(i-1)\mapsto i$. The last statement follows from the monotone coupling constructed in the second statement combined with Cramér's theorem, since the mean of the measure $\nu_{\b}$ on $\{-1,0,1\}$ is $(\b-1)/(\b+1)$.
\end{proof}

We have the following control on the location $T_N$ from \eqref{def:TN} in the high density phase of the open TASEP. 

\begin{lemma}\label{lem:TN exponential}
    Assume $\b > \max(\a,1)$ and $\a\b>1$. Consider the probability measure $\mathbb{P}^{(0,\a,\b)}$ on $\lrw_N$ as in \eqref{eq:PNab TASEP}. Then there exist constants $C,c>0$ depending only on $\a$ and $\b$, such that for each $N\geq m\geq1$, we have $\mathbb{P}^{(0,\a,\b)}(T_N\geq m)\leq C e^{-cm}$. Furthermore, we have $\inf_{N\in\mathbb{N}}\mathbb{P}^{(0,\a,\b)}(T_N=0)>0$.
\end{lemma}
\begin{proof}
    By Lemma \ref{lem:basic hitting monotone coupling}, we have $\lim_{\ell\rightarrow\infty}p_{\ell}^{(\b)}=1-1/\b^2$, which is the probability that an infinite $\nu_{\b}$-random walk started at $0$ never reaches level $-1$. Hence $\inf_{\ell\in\mathbb{N}}p_{\ell}^{(\b)}>0$.  We note that for each $N\geq m\geq1$,
    \[
    \mathbb{P}^{(0,\a,\b)}(T_N = m)=\frac{\a}{4C_{N,\a,\b}}w_{\a}^{m-1}w_{\b}^{N-m}p^{(\a)}_{m-1}p^{(\b)}_{N-m}  
    \leq\frac{\a w_{\a}^{m-1}w_{\b}^{N-m}p^{(\a)}_{m-1}p^{(\b)}_{N-m}}{4w_{\b}^Np^{(\b)}_N}
    \leq\frac{\a }{4w_{\a}\inf_{\ell\in\mathbb{N}}p_{\ell}^{(\b)}}\left(\frac{w_{\a}}{w_{\b}}\right)^m,
    \]
    where in the first step we use \eqref{eq: dist of TN}, in the second step we use $C_{N,\a,\b}\geq w_{\b}^Np^{(\b)}_N$ which follows from \eqref{eq:C N a b}, and in the third step we use  $p^{(\a)}_{m-1}, p^{(\b)}_{N-m}\in[0,1]$. Since $\b >  \a$ and $\a\b>1$, we have $w_{\a}/w_{\b}\in(0,1)$. Therefore,
    \[
\mathbb{P}^{(0,\a,\b)}(T_N\ge m)
= \sum_{k=m}^{N}\mathbb{P}^{(0,\a,\b)}(T_N=k)
\le \frac{\a}{4w_{\a}\inf_{\ell\in\mathbb{N}}p_{\ell}^{(\b)}}\sum_{k=m}^{\infty}\left(\frac{w_{\a}}{w_{\b}}\right)^k
= \frac{\a}{4w_{\a}\inf_{\ell\in\mathbb{N}}p_{\ell}^{(\b)}\left(1-\frac{w_{\a}}{w_{\b}}\right)} \left(\frac{w_{\a}}{w_{\b}}\right)^m.
\] We conclude the proof of the first claim in the statement. For the second claim, we have
    \begin{equation}
    \begin{split}
         \mathbb{P}^{(0,\a,\b)}(T_N=0)&=\frac{w_{\b}^Np^{(\b)}_N}{w_{\b}^Np^{(\b)}_N+\sum_{m=1}^{N}\frac{\a}{4} w_{\a}^{m-1}w_{\b}^{N-m}p_{m-1}^{(\a)}p_{N-m}^{(\b)}}
    =\frac{1}{1+\sum_{m=1}^{N}\frac{\a p_{m-1}^{(\a)}p_{N-m}^{(\b)}}{4w_{\a}p^{(\b)}_N} \left(\frac{w_{\a}}{w_{\b}}\right)^m }\\
    &\geq \frac{1}{1+ \frac{\a }{4w_{\a}\inf_{\ell\in\mathbb{N}}p_{\ell}^{(\b)}} \sum_{m=1}^{N}\left(\frac{w_{\a}}{w_{\b}}\right)^m }\geq\frac{1}{1+ \frac{\a }{4w_{\a}\inf_{\ell\in\mathbb{N}}p_{\ell}^{(\b)}}\cdot\frac{w_{\a}/w_{\b}}{1-w_{\a}/w_{\b}} }>0,
    \end{split}
    \end{equation}
     where in the first step we use \eqref{eq: dist of TN} and \eqref{eq:C N a b}, the second step is a reorganization of the terms, in the third step we use $p^{(\a)}_{m-1}, p^{(\b)}_{N-m}\in[0,1]$, and in the fourth and fifth steps we use $w_{\a}/w_{\b}\in(0,1)$ and $\inf_{\ell\in\mathbb{N}}p_{\ell}^{(\b)}>0$. 
\end{proof}

\subsection{A replacement measure for $\mathbb{P}^{(0,\mathsf{a},\mathsf{b})}$} \label{sec:ReplacementP}
In this subsection, we construct a probability measure $\mathrm{P}^{(0,\mathsf{a},\mathsf{b})}$ on $\lrw_N$ that has a simpler structure than that of $\mathbb{P}^{(0,\mathsf{a},\mathsf{b})}$, and that has an exponentially small (in $N$) total variation distance with respect to $\mathbb{P}^{(0,\mathsf{a},\mathsf{b})}$. Therefore, the studies of $\mathbb{P}^{(0,\mathsf{a},\mathsf{b})}$ can be reduced to the studies of $\mathrm{P}^{(0,\mathsf{a},\mathsf{b})}$. 

\begin{definition}\label{def:positive lazy random walk}
Assume $\b>1$ and $N\in\mathbb{N}\cup\{\infty\}$. Consider a $\nu_{\b}$-random walk $\{S_i\}_{i=0}^{N}$ started at $S_0=0$. Let
\begin{equation}\label{eq:E}
E_N=\{S_1=1\text{ and }S_i\geq 1\text{ for all }i\in\llbracket1,N\rrbracket\}.
\end{equation}
Observe that $\mathbb{P}(E_N)>0$. 
We call the process $\{S_i\}_{i=0}^{N}$ conditioned on $E_N$ the \emph{positive $\nu_{\b}$-random walk with $N$ steps}. If $N=\infty$, we call the conditioned process the \emph{positive $\nu_{\b}$-random walk with infinitely many steps}.
\end{definition}
Instead of the observable $T_N$, it will be convenient to consider the following closely related quantity $Q_N$, denoting the last location a lazy random walk path obtains its minimum.
\begin{definition}\label{def:approx TASEP measure}
We define the observable $Q_N$ on the space $\lrw_N$ by
\begin{equation}\label{eq:btrert}
Q_N(\bg)=\max\{i\in\llbracket0,N\rrbracket:\gamma_i=\min(\bg)\}\quad\mbox{ for } \bg\in\lrw_N.
\end{equation}
We have $0\leq T_N(\bg)\leq Q_N(\bg)\leq N$ for each $\bg\in\lrw_N$.

Assume $\b > \max(\a,1)$ and $\a\b>1$. Consider the measure $\mathbb{P}^{(0,\a,\b)}$ on $\lrw_N$ as in \eqref{eq:PNab TASEP}. 
By Lemma~\ref{lem:decomposition TASEP}, conditionally on $T_N$, the law of $(\gamma_{T_N+i}-\gamma_{T_N})_{i\in\llbracket 0,N-T_N\rrbracket}$ is that of a $\nu_{\b}$-random walk started at $0$ and conditioned to stay non-negative up to step $N-T_N$.
Hence in view of \eqref{eq:btrert}, conditionally on $Q_N$, the law of $(\gamma_{Q_N+i}-\gamma_{Q_N})_{i\in\llbracket 0,N-Q_N\rrbracket}$ is that of a positive $\nu_{\b}$-random walk with $N-Q_N$ steps as in Definition \ref{def:positive lazy random walk}.

We define the probability measure $\mathrm{P}^{(0,\mathsf{a},\mathsf{b})}$ by sampling $Q_N$ and $(\gamma_i)_{i\in\llbracket0,Q_N\rrbracket}$ exactly as under the probability measure $\mathbb{P}^{(0,\mathsf{a},\mathsf{b})}$, and then replacing the conditional law of $(\gamma_{Q_N+i}-\gamma_{Q_N})_{i\in\llbracket 0,N-Q_N\rrbracket}$ by the law of the positive $\nu_{\b}$-random walk with infinitely many steps, as in Definition \ref{def:positive lazy random walk}, restricted to the first $N-Q_N$ steps.
\end{definition}

The next result shows that $\mathrm{P}^{(0,\a,\b)}$ is a good approximation of $\mathbb{P}^{(0,\a,\b)}$ using the total-variation distance from \eqref{def:TVdistance}.
\begin{lemma}\label{lem:TV distance replaced measure}
    Assume $\b > \max(\a,1)$ and $\a\b>1$. Consider the probability measures $\mathbb{P}^{(0,\a,\b)}$ and $\mathrm{P}^{(0,\mathsf{a},\mathsf{b})}$ on $\lrw_N$ as in \eqref{eq:PNab TASEP} and Definition \ref{def:positive lazy random walk}, respectively.   
    Then there exist constants $C,c>0$ depending only on $\a$ and $\b$, such that for each $N\in\mathbb{N}$,
    \begin{equation}
        \TV{\mathbb{P}^{(0,\a,\b)}-\mathrm{P}^{(0,\a,\b)}}\leq Ce^{-cN}.
    \end{equation}
\end{lemma}

Before we prove Lemma \ref{lem:TV distance replaced measure}, we upper bound the total variation distance between the law of the positive $\nu_{\b}$-random walk with $N$ steps and that of the positive $\nu_{\b}$-random walk infinitely many steps.

\begin{lemma}\label{lem:TV difference positive lazy random walks}
    Assume $\b>1$. Let $\{Y_i\}_{i=0}^{N}$ be the positive $\nu_{\b}$-random walk with $N$ steps and $\{Y_i'\}_{i=0}^{\infty}$ be the positive $\nu_{\b}$-random walk with infinitely many steps, as in Definition \ref{def:positive lazy random walk}. Then there exist constants $C,c>0$ depending only on $\b$, such that for each $N\in\mathbb{N}$, 
    \begin{equation}
        \TV{\mathbb{P}\left((Y_i)_{i\in\llbracket0,N\rrbracket}\in\cdot\,\right)-\mathbb{P}\left((Y_i')_{i\in\llbracket0,N\rrbracket}\in\cdot\,\right)}\leq Ce^{-cN}.
    \end{equation} 
\end{lemma}
\begin{proof}
Let $\{S_i\}_{i=0}^{\infty}$ be an ordinary $\nu_{\b}$-random walk started at $S_0=0$. 
    We denote
    \[
    E_N=\{S_1=1\text{ and }S_i\geq 1\text{ for all }i\in\llbracket1,N\rrbracket\},\quad 
    E_{\infty}=\{S_1=1\text{ and }S_i\geq 1\text{ for all }i\in\mathbb{N}\}.
    \]
    We have $E_{\infty}\subset E_{N}$.
    By Definition \ref{def:positive lazy random walk}, we have
    \begin{equation}\label{eq:nrrtrt1}
        \begin{split}
            &\TV{\mathbb{P}\left((Y_i)_{i\in\llbracket0,N\rrbracket}\in\cdot\,\right)-\mathbb{P}\left((Y_i')_{i\in\llbracket0,N\rrbracket}\in\cdot\,\right)}
             =\TV{\mathbb{P}\left(\,\cdot\,|E_N\right)-\mathbb{P}\left(\,\cdot\,|E_{\infty}\right)}\\
            &=\sup_{A\subset\lrw_N}\left|\mathbb{P}\left(A|E_N\right)-\mathbb{P}\left(A|E_{\infty}\right)\right| 
            =\sup_{A\subset\lrw_N}\left|\frac{\mathbb{P}\left(A \cap E_N\right)}{\mathbb{P}(E_N)}-\frac{\mathbb{P}\left(A \cap E_{\infty}\right)}{\mathbb{P}(E_{\infty})}\right|\\
            &\leq\sup_{A\subset\lrw_N}\frac{\mathbb{P}(E_{\infty})\left(\mathbb{P}\left(A \cap E_N\right)-\mathbb{P}\left(A \cap  E_{\infty}\right)\right)+\mathbb{P}\left(A   \cap E_{\infty}\right)\left(\mathbb{P}(E_N)-\mathbb{P}(E_{\infty})\right)}{\mathbb{P}(E_N)\mathbb{P}(E_{\infty})} \leq\frac{2\mathbb{P}(E_N\setminus E_{\infty})}{\mathbb{P}(E_{\infty})}.
        \end{split}
    \end{equation}
 Let $\kappa=(\b-1)/(2(\b+1))$. We have
\begin{equation}\label{eq:nrrtrt2}
\begin{split}
    \mathbb{P}(E_N\setminus E_{\infty})&=\sum_{h=1}^{\infty}\mathbb{P}(E_N\cap\{S_N=h\})\cdot\mathbb{P}_{h}(\tau_0<\infty)\\
    &\leq\mathbb{P}(S_N\leq\kappa N)+\sum_{h>\kappa N} \mathbb{P}(E_N\cap\{S_N=h\})\cdot\b^{-2h}
    \leq C_1e^{-c_1N}+\b^{-2\kappa N},
\end{split}
\end{equation}
for some $C_1,c_1>0$ depending only on $\b$. Here, we mention that $\mathbb{P}_{h}(\tau_0<\infty)=\b^{-2h}$ is the probability that an ordinary infinite $\nu_{\b}$-random walk started at $S_N=h$ hits $0$ in finite time, in view of Lemma \ref{lem:basic hitting monotone coupling}; and in the last step we use Lemma \ref{lem:basic hitting monotone coupling} again. Combining \eqref{eq:nrrtrt1} and \eqref{eq:nrrtrt2} we conclude the proof.
 \end{proof}

We are now ready to prove Lemma \ref{lem:TV distance replaced measure}.
\begin{proof}[Proof of Lemma \ref{lem:TV distance replaced measure}]
    We first show that there exist $C_1,c_1>0$ depending only on $\a$ and $\b$ such that 
    \begin{equation}\label{eq:rgregrt}
        \mathbb{P}^{(0,\a,\b)}(Q_N\geq m)\leq C_1e^{-c_1m}\quad\mbox{ for each }N\geq m\geq1,
    \end{equation}
    where $Q_N$ is as in Definition \ref{def:approx TASEP measure}. Note that conditionally on $T_N$, we have
    $Q_N-T_N\overset{d}{=}\max\{k\in\llbracket0,N-T_N\rrbracket: Y_k=0\}$, where $\{Y_i\}_{i=0}^{N-T_N}$ is a $\nu_{\b}$-random walk started at $Y_0=0$ and conditioned on staying non-negative up to $N-T_N$ steps. 
    By Lemma \ref{lem:basic hitting monotone coupling}, there exist $C_2,c_2>0$ depending only on $\b$ such that for each $N\geq m\geq1$, we have $\mathbb{P}(Q_N-T_N\geq m)\leq C_2e^{-c_2m}$. 
    On the other hand, by Lemma \ref{lem:TN exponential}, there exist $C_3,c_3>0$ depending on $\a$ and $\b$ such that for each $N\geq m\geq1$, we have $\mathbb{P}(T_N\geq m)\leq C_3e^{-c_3m}$. We conclude the proof of \eqref{eq:rgregrt}. 
    
    Note that the law of $Q_N$ is the same under $\mathbb{P}^{(0,\a,\b)}$ and $\mathrm{P}^{(0,\a,\b)}$ by construction.  
    For each $A\subset\lrw_N$, we have
    \begin{equation*}
    \begin{split}
        \left|\mathbb{P}^{(0,\a,\b)}(A)-\mathrm{P}^{(0,\a,\b)}(A)\right|&\leq \mathbb{P}(Q_N>N/2)+\sum_{m=0}^{\lfloor N/2\rfloor}\mathbb{P} (Q_N=m)\cdot
    \TV{\mathbb{P}^{(0,\a,\b)}(\,\cdot\,| Q_N=m)-\mathrm{P}^{(0,\a,\b)}(\,\cdot\,| Q_N=m)}\\
    &\leq C_1e^{-c_1N/2}+\sup_{0\leq m\leq \lfloor N/2\rfloor } C_4e^{-c_4(N-m)}\leq Ce^{-cN},
    \end{split}
    \end{equation*}
for some constants $C,c>0$ depending only on $\a$ and $\b$, where in the first step we condition on the value of $Q_N$; and in the second step we use \eqref{eq:rgregrt} to bound the first summand and Definition \ref{def:approx TASEP measure} combined with Lemma \ref{lem:TV difference positive lazy random walks} (with constants $C_4,c_4>0$ depending only on $\b$) to bound the second summand. 
\end{proof}

\section{Moment bounds of the Bryc polynomials}\label{sec:MomentBryc}

The goal of this subsection is to provide effective bounds on the moments of the Bryc polynomial when the underlying $0$-$1$-sequence is chosen according to the two-layer representation of the open TASEP. 

\begin{proposition}\label{prop:moment bound Bryc}
Assume $q\in(0,1)$, $\b > \max(\a,1)$, $\a\b>1$, $z>1$, and $k\in\mathbb{N}$. For each $N\in\mathbb{N}$, consider the random path $\bg$ with law $\mathbb{P}^{(0,\a,\b)}$ on $\lrw_N$. Let $\bs(\bg)$ be the composition of $N+1$ associated to $\bg$ as in \eqref{eq:Signature}, and let $W_{\bs(\bg)}(z)$ be the Bryc polynomial indexed by the composition $\bs(\bg)$ as in Definition \ref{def:Bryc poly}.  Then,  
\[
    \sup_{N\in\mathbb{N}}\mathbb{E}\bigl[W_{\bs(\bg)}(z)^k\bigr]<\infty
\] 
is a finite constant depending only on $q$, $\a$, $\b$, $z$, and $k$.
\end{proposition}

Let us give a brief outline of the proof of Proposition~\ref{prop:moment bound Bryc}.
As a first step, we define an auxiliary birth-death chain $\mathscr{X}^{\boldsymbol{\mathsf{S}}}$ in Definition~\ref{def:simpler chain}, where the transition probabilities take a simpler form, and which will be easier to analyze. Moreover, we will argue that it suffices to show Proposition~\ref{prop:moment bound Bryc} with respect to the replacement measure $\mathrm{P}^{(0,\a,\b)}$ from Section~\ref{sec:ReplacementP}. The latter will be achieved in Section~\ref{sec:MomentBDReplace} by a comparison to a suitable homogeneous birth-death chain. A main technical input are exponential tail bounds on the regeneration times of the random walk paths under $\mathrm{P}^{(0,\a,\b)}$, shown in Lemma~\ref{lem:regeneration indices MAIN}, which are of independent interest.

\subsection{A simplified birth-death chain}

Assume $q\in(0,1)$, $N\in\mathbb{Z}_{\geq0}$, and $\boldsymbol{\mathsf{S}}=(\mathsf{S}(0),\dots,\mathsf{S}(N))\in\{0,1\}^{N+1}$ is a $0$-$1$ sequence.
We recall the inhomogeneous birth-death Markov chain $X^{\boldsymbol{\mathsf{S}}}=\{X^{\boldsymbol{\mathsf{S}}}(i)\}_{i=0}^{N+1}$ from Definition~\ref{def:characterization measure}. We have
$X^{\boldsymbol{\mathsf{S}}}(0)=X^{\boldsymbol{\mathsf{S}}}(1)=0$ and that for each $i\in\llbracket1,N\rrbracket$ and $j\in\llbracket0,i-1\rrbracket$,
\begin{enumerate}
    \item [(\romannum{1})] if $\mathsf{S}(i)=1$, we let $X^{\boldsymbol{\mathsf{S}}}(i+1)-X^{\boldsymbol{\mathsf{S}}}(i)\in\{-1,0\}$ with 
    \begin{equation}\label{eq:etgetq1}
    \mathbb{P}\left(X^{\boldsymbol{\mathsf{S}}}(i+1)=j-1 \Big{|}X^{\boldsymbol{\mathsf{S}}}(i)=j\right)=\frac{1-q^{j}}{1-q^{i+1}},
    \end{equation}
    \item [(\romannum{2})] if $\mathsf{S}(i)=0$, we let $X^{\boldsymbol{\mathsf{S}}}(i+1)-X^{\boldsymbol{\mathsf{S}}}(i)\in\{0,1\}$ with 
    \begin{equation}\label{eq:etgetq2}
    \mathbb{P}\left(X^{\boldsymbol{\mathsf{S}}}(i+1)=j+1 \Big{|}X^{\boldsymbol{\mathsf{S}}}(i)=j\right)=\frac{q^{j+1}-q^{i+1}}{1-q^{i+1}}.
    \end{equation}
\end{enumerate} 
Our goal is to understand $X^{\boldsymbol{\mathsf{S}}}(N)$ for a typical $\mathsf{S}$.
To this end, we define a simpler auxiliary Markov chain as follows.
\begin{definition}\label{def:simpler chain}
    Assume $q\in(0,1)$, $N,x\in\mathbb{Z}_{\geq0}$, and $\boldsymbol{\mathsf{S}}=(\mathsf{S}(0),\dots,\mathsf{S}(N))\in\{0,1\}^{N+1}$ is a $0$-$1$ sequence. Define a Markov chain $\mathscr{X}^{\boldsymbol{\mathsf{S}},x}=\{\mathscr{X}^{\boldsymbol{\mathsf{S}},x}(i)\}_{i=0}^{N+1}$ by $\mathscr{X}^{\boldsymbol{\mathsf{S}},x}(0)=x$, and for each $i\in\llbracket0,N\rrbracket$ and $j\in\mathbb{Z}_{\geq0}$,
    \begin{enumerate}
    \item [(\romannum{1})] if $\mathsf{S}(i)=1$, we let $\mathscr{X}^{\boldsymbol{\mathsf{S}},x}(i+1)-\mathscr{X}^{\boldsymbol{\mathsf{S}},x}(i)\in\{-1,0\}$ with 
    \begin{equation}\label{eq:etgets1}
    \mathbb{P}\left(\mathscr{X}^{\boldsymbol{\mathsf{S}},x}(i+1)=j-1 \Big{|}\mathscr{X}^{\boldsymbol{\mathsf{S}},x}(i)=j\right)=(1-q)\cdot {\bf1}_{j\geq1},
    \end{equation}
    \item [(\romannum{2})] if $\mathsf{S}(i)=0$, we let $\mathscr{X}^{\boldsymbol{\mathsf{S}},x}(i+1)-\mathscr{X}^{\boldsymbol{\mathsf{S}},x}(i)\in\{0,1\}$ with 
    \begin{equation}\label{eq:etgets2}
    \mathbb{P}\left(\mathscr{X}^{\boldsymbol{\mathsf{S}},x}(i+1)=j+1 \Big{|}\mathscr{X}^{\boldsymbol{\mathsf{S}},x}(i)=j\right)=q^{j+1}.
    \end{equation}
\end{enumerate} 
When $x=0$, we denote the Markov chain $\mathscr{X}^{\boldsymbol{\mathsf{S}},0}$ simply as $\mathscr{X}^{\boldsymbol{\mathsf{S}}}$.
\end{definition}

For two random variables $X$ and $X'$, we write $X \succeq X'$ whenever there exists a coupling $\mathbf{P}$ between $X$ and $X'$ such that $\mathbf{P}(X \succeq X')=1$. 

\begin{lemma}\label{lem:coupling birth death}
    Assume $q\in(0,1)$ and $N\in\mathbb{Z}_{\geq0}$. Then the following statements hold.
    \begin{enumerate}
        \item [(\romannum{1})] Let $\boldsymbol{\mathsf{S}},\boldsymbol{\mathsf{S}}'\in\{0,1\}^{N+1}$ satisfy $\mathsf{S}(i)\geq\mathsf{S}'(i)$ for each $i\in\llbracket0,N\rrbracket$. Then there exists a coupling between Markov chains $X^{\boldsymbol{\mathsf{S}}}$ and $X^{\boldsymbol{\mathsf{S}}'}$ such that almost surely
        \[
        X^{\boldsymbol{\mathsf{S}}}(i)\leq X^{\boldsymbol{\mathsf{S}}'}(i)\quad\mbox{ for all }i\in\llbracket0,N+1\rrbracket.
        \]
        \item [(\romannum{2})] Let $\boldsymbol{\mathsf{S}}\in\{0,1\}^{N+1}$. Then there exists a coupling between Markov chains $X^{\boldsymbol{\mathsf{S}}}$ and $\mathscr{X}^{\boldsymbol{\mathsf{S}} }$ such that almost surely
        \[
        X^{\boldsymbol{\mathsf{S}}}(i)\leq \mathscr{X}^{\boldsymbol{\mathsf{S}} }(i)\quad\mbox{ for all }i\in\llbracket0,N+1\rrbracket.
        \]
        \item [(\romannum{3})] Let $x,x'\in\mathbb{Z}_{\geq0}$ satisfy $x\leq x'$ and $\boldsymbol{\mathsf{S}},\boldsymbol{\mathsf{S}}'\in\{0,1\}^{N+1}$ satisfy $\mathsf{S}(i)\geq\mathsf{S}'(i)$ for each $i\in\llbracket0,N\rrbracket$. Then there exists a coupling between Markov chains $\mathscr{X}^{\boldsymbol{\mathsf{S}},x}$ and $\mathscr{X}^{\boldsymbol{\mathsf{S}}',x'}$ such that  almost surely
        \[
        \mathscr{X}^{\boldsymbol{\mathsf{S}},x}(i)\leq \mathscr{X}^{\boldsymbol{\mathsf{S}}',x'}(i)\quad\mbox{ for all }i\in\llbracket0,N+1\rrbracket.
        \]
        \item [(\romannum{4})] Let $x,y\in\mathbb{Z}_{\geq0}$ and $\boldsymbol{\mathsf{S}}\in\{0,1\}^{N+1}$. Then there exists a coupling between Markov chains $\mathscr{X}^{\boldsymbol{\mathsf{S}},x}$ and $\mathscr{X}^{\boldsymbol{\mathsf{S}},x+y}$ such that 
        \[
        \mathscr{X}^{\boldsymbol{\mathsf{S}},x+y}(i)-y\leq \mathscr{X}^{\boldsymbol{\mathsf{S}},x}(i)\quad\mbox{ for all }i\in\llbracket0,N+1\rrbracket.
        \]
        \item [(\romannum{5})] Let $M\in\llbracket0,N-1\rrbracket$ and $\boldsymbol{\mathsf{S}}'\in\{0,1\}^{N-M}$ be given by $\mathsf{S}'(i)=\mathsf{S}(i+M+1)$ for each $i\in\llbracket0,N-M-1\rrbracket$.
        Then we have
        \[
        \mathscr{X}^{\boldsymbol{\mathsf{S}}} (N+1)-\mathscr{X}^{\boldsymbol{\mathsf{S}}} (M+1)\preceq\mathscr{X}^{\boldsymbol{\mathsf{S}}'} (N -M).
        \]
    \end{enumerate}
\end{lemma}
\begin{proof}
    We first introduce deterministic update maps. Assume $s\in\{0,1\}$ and $u\in(0,1)$. We let
    \begin{equation}\label{eq:phipsi}
        \begin{split}
            \Phi_i^s(j,u)=j-{\bf1}_{s=1}{\bf1}_{\left\{u\leq \frac{1-q^j}{1-q^{i+1}}\right\}}+{\bf1}_{s=0}{\bf1}_{\left\{u\leq \frac{q^{j+1}-q^{i+1}}{1-q^{i+1}}\right\}},\quad &\mbox{ for }i\in\llbracket1,N\rrbracket\mbox{ and }j\in\llbracket0,i-1\rrbracket,\\
            \Psi^s(j,u)=j-{\bf1}_{s=1}{\bf1}_{\{j\geq1,u\leq 1-q\}}+{\bf1}_{s=0}{\bf1}_{\{u\leq q^{j+1}\}},\quad &\mbox{ for }j\in\mathbb{Z}_{\geq0}.
        \end{split}
    \end{equation}
Note that if $U\sim\mathrm{Unif}(0,1)$ and $s=\mathsf{S}(i)\in\{0,1\}$, then $\Phi_i^s(j,U)$ and $\Psi^s(j,U)$ have exactly the transition laws of the Markov chains $X^{\boldsymbol{\mathsf{S}}}$ and $\mathscr{X}^{\boldsymbol{\mathsf{S}},x}$ at step $i\mapsto i+1$.
We now state the following four deterministic inequalities regarding the maps $\Phi$ and $\Psi$.
For all $i\in\llbracket1,N\rrbracket$, $u\in(0,1)$ and $s,s'\in\{0,1\}$ satisfying $s\geq s'$, 
\begin{align}
   \label{eq:betbretgsertb1} \Phi_i^s(j,u)\leq \Phi_i^{s'}(k,u)\quad&\mbox{ for } 0\leq j\leq k\leq i-1,\\
   \label{eq:betbretgsertb2} \Phi_i^s(j,u)\leq \Psi^s(k,u)\quad&\mbox{ for } j\in\llbracket0,i-1\rrbracket\mbox{ and }k\in\mathbb{Z}_{\geq0}\mbox{ satisfying }j\leq k,\\
   \label{eq:betbretgsertb3} \Psi^s(j,u)\leq \Psi^{s'}(k,u)\quad&\mbox{ for } j,k\in\mathbb{Z}_{\geq0}\mbox{ satisfying }j\leq k,\\
    \label{eq:betbretgsertb4}\Psi^s(k,u)-y\leq \Psi^s(j,u)\quad&\mbox{ for } y\in\mathbb{Z}_{\geq0}\mbox{ and }j,k\in\mathbb{Z}_{\geq0}\mbox{ satisfying }k-y\leq j.
\end{align}
These inequalities can be directly checked from the definition \eqref{eq:phipsi} of the maps $\Phi$ and $\Psi$.

We now construct the four required couplings. Let $U_0,\dots,U_N$ be i.i.d. $\mathrm{Unif}(0,1)$ random variables.
\begin{enumerate}
    \item [(\romannum{1})] We define $Y$ and $Y'$ by $Y(0)=Y'(0)=Y(1)=Y'(1)=0$ and 
    \[
    Y(i+1)=\Phi_i^{\mathsf{S}(i)}(Y(i),U_i),
        \quad
        Y'(i+1)=\Phi_i^{\mathsf{S}'(i)}(Y'(i),U_i)\quad \mbox{ for }i\in\llbracket1,N\rrbracket.
    \]
Then $Y\overset{d}{=}X^{\boldsymbol{\mathsf{S}}}$ and $Y'\overset{d}{=}X^{\boldsymbol{\mathsf{S}}'}$. Statement (\romannum{1}) follows from an induction through inequality \eqref{eq:betbretgsertb1}.
    

\item [(\romannum{2})] We define $Y$ and $Z$ by $Y(0)=Y(1)=0$, $Z(0)=0$, $Z(1)=\Psi^{\mathsf{S}(0)}(Z(0),U_0)$, and 
    \[
    Y(i+1)=\Phi_i^{\mathsf{S}(i)}(Y(i),U_i),
        \quad
        Z(i+1)=\Psi^{\mathsf{S}(i)}(Z(i),U_i)\quad \mbox{ for }i\in\llbracket1,N\rrbracket.
    \]
Then $Y\overset{d}{=}X^{\boldsymbol{\mathsf{S}}}$ and $Z\overset{d}{=}\mathscr{X}^{\boldsymbol{\mathsf{S}} }$. Statement (\romannum{2}) follows from an induction through inequality \eqref{eq:betbretgsertb2}.    

\item [(\romannum{3})] We define $Y$ and $Y'$ by $Y(0)=x$, $Y'(0)=x'$, and
\[
Y(i+1)=\Psi^{\mathsf{S}(i)}(Y(i),U_i),
        \quad
        Y'(i+1)=\Psi^{\mathsf{S}'(i)}(Y'(i),U_i)\quad \mbox{ for }i\in\llbracket0,N\rrbracket.
\]
Then $Y\overset{d}{=}\mathscr{X}^{\boldsymbol{\mathsf{S}},x}$ and $Y'\overset{d}{=}\mathscr{X}^{\boldsymbol{\mathsf{S}}',x'}$. Statement (\romannum{3}) follows from an induction through inequality \eqref{eq:betbretgsertb3}.

\item [(\romannum{4})] We define $Y$ and $Z$ by $Y(0)=x$, $Z(0)=x+y$, and 
\[
        Y(i+1)=\Psi^{\mathsf{S}(i)}(Y(i),U_i),
        \quad
        Z(i+1)=\Psi^{\mathsf{S}(i)}(Z(i),U_i)\quad \mbox{ for }i\in\llbracket0,N\rrbracket.
\]
Then $Y\overset{d}{=}\mathscr{X}^{\boldsymbol{\mathsf{S}},x}$ and $Z\overset{d}{=}\mathscr{X}^{\boldsymbol{\mathsf{S}},x+y}$. Statement (\romannum{4}) follows from an induction through inequality \eqref{eq:betbretgsertb4}.  
\end{enumerate}

For the last statement, note that by the Markov property, conditionally on $\mathscr{X}^{\boldsymbol{\mathsf{S}}}(M+1)$,  
\[
\mathscr{X}^{\boldsymbol{\mathsf{S}}}(N+1)-\mathscr{X}^{\boldsymbol{\mathsf{S}}}(M+1)
\overset{d}{=}
\mathscr{X}^{\boldsymbol{\mathsf{S}}',\mathscr{X}^{\boldsymbol{\mathsf{S}}}(M+1)}(N-M)-\mathscr{X}^{\boldsymbol{\mathsf{S}}}(M+1) 
\preceq
\mathscr{X}^{\boldsymbol{\mathsf{S}}'}(N-M),
\]
where the last step uses Statement (\romannum{4}) for the sequence $\boldsymbol{\mathsf{S}}'\in\{0,1\}^{N-M}$ with $x=0$ and $y=\mathscr{X}^{\boldsymbol{\mathsf{S}}}(M+1)\in\mathbb{Z}_{\geq0}$.
Averaging over the value of $\mathscr{X}^{\boldsymbol{\mathsf{S}}}(M+1)$ yields Statement (\romannum{5}). 
\end{proof}

We now prove an apriori bound that holds for every $0$-$1$ sequence $\boldsymbol{\mathsf{S}}$.
\begin{lemma}\label{lem:apriori bound birth death chain}
    Assume $q\in(0,1)$ and $z>1$. Consider the Markov chain $\mathscr{X}^{\boldsymbol{\mathsf{S}} }$ as in Definition \ref{def:simpler chain}. There exists a constant $C\in\mathbb{R}_{>0}$ depending only on $q$ and $z$ such that for all $N\in\mathbb{N}$, $m\in\llbracket0,N+1\rrbracket$,  $\boldsymbol{\mathsf{S}}\in\{0,1\}^{N+1}$,
\begin{equation}\label{eq:ynrtyn}
\mathbb{E}\left[z^{\mathscr{X}^{\boldsymbol{\mathsf{S}} }(m)}\right]
    \leq
    Cz^{m^{3/4}} .
\end{equation} 
\end{lemma}
\begin{proof} When $m=0$, the left side of \eqref{eq:ynrtyn} is $1$. We next assume $m\geq1$. By restricting the sequence $\boldsymbol{\mathsf{S}}$ to the first $m$ numbers, and in view of Lemma \ref{lem:coupling birth death} (\romannum{3}), it suffices to show that for each $N\in\mathbb{Z}_{\geq0}$,
    \begin{equation*}
        \mathbb{E}\left[z^{\mathscr{X}^{\boldsymbol{\mathsf{S}}_0 }(N+1)}\right]
    \leq
    Cz^{N^{3/4}},\quad\mbox{ where }\boldsymbol{\mathsf{S}}_0=(0,0,\dots,0)\in\{0,1\}^{N+1}.
    \end{equation*}
    In the rest of the proof, we denote $Y=\mathscr{X}^{\boldsymbol{\mathsf{S}}_0 }$.
    By Definition \ref{def:simpler chain}, we have $Y(0)=0$, and for each $i\in\llbracket0,N\rrbracket$,
    we have $Y(i+1)-Y(i)\in\{0,1\}$ with probabilities
    \[ 
    \mathbb{P}\left(Y(i+1)=j+1 \Big{|}Y(i)=j\right)=q^{j+1}\quad\mbox{ and }\quad
    \mathbb{P}\left(Y(i+1)=j \Big{|}Y(i)=j\right)=1-q^{j+1}.
    \]
    For each deterministic path $(h_0,\dots,h_{N+1})$ satisfying $h_0=0$ and $h_{i+1}-h_i\in\{0,1\}$ for $i\in\llbracket0,N\rrbracket$, we have
    \[
    \mathbb{P}\left(Y(i)=h_i\mbox{ for } i\in\llbracket0,N+1\rrbracket\right)=\prod_{i=0}^{N}\left(q^{h_i+1}{\bf1}_{h_{i+1}=h_i+1}+\left(1-q^{h_i+1}\right){\bf1}_{h_{i+1}=h_i}\right)\leq q^{\sum_{i=0}^N(h_i+1){\bf1}_{h_{i+1}=h_i+1}}.
    \]
Note that there are totally $2^{N+1}$ many Bernoulli paths $(h_0,\dots,h_{N+1})$. For each Bernoulli path, we always have $h_{N+1}\in\llbracket0,N+1\rrbracket$, and if $h_{N+1}\geq \lceil N^{3/4}\rceil$, then we have $\sum_{i=0}^N(h_i+1){\bf1}_{h_{i+1}=h_i+1}\geq\lceil N^{3/4}\rceil^2/2\geq N^{3/2}/2$. 
By summing up these probabilities, we have 
$\mathbb{P}(Y(N+1)\geq \lceil N^{3/4}\rceil)\leq 2^{N+1} \cdot q^{N^{3/2}/2}$. Therefore,
\[
\mathbb{E}\left[z^{Y(N+1)}\right]\leq z^{\lceil N^{3/4}\rceil}\mathbb{P}\left(Y(N+1)\leq \lceil N^{3/4}\rceil\right)+
z^{N+1}\mathbb{P}\left(Y(N+1)\geq \lceil N^{3/4}\rceil\right)\leq z^{\lceil N^{3/4}\rceil}+z^{N+1}  2^{N+1}  q^{N^{3/2}/2}.
\]
  The right side above can be bounded by $Cz^{N^{3/4}}$ for a finite constant $C$ depending only on $q$ and $z$.   
\end{proof}

\begin{proposition} \label{lem:moment bound simpler chain} 
Assume $q\in(0,1)$, $\b > \max(\a,1)$, $\a\b>1$, and $z>1$. For each $N\in\mathbb{N}$, sample $\bg\sim\mathrm{P}^{(0,\a,\b)}$ as in Definition \ref{def:approx TASEP measure}, let $\bs=\bs(\bg)$ be the associated composition of $N+1$ as in \eqref{eq:Signature}, let $\boldsymbol{\mathsf{S}}^{\bs(\bg)}\in\{0,1\}^{N+1}$ be the associated $0$-$1$ sequence as in Definition \ref{def:associated 01 seq}, and let $\mathscr{X}^{\boldsymbol{\mathsf{S}}^{\bs(\bg)} }$ be the Markov chain as in Definition \ref{def:simpler chain}, sampled conditionally on $\boldsymbol{\mathsf{S}}^{\bs(\bg)}$ using independent transition randomness. Then 
\begin{equation}\label{eq:ryjutuymyu}
\sup_{N\in\mathbb{N}}\sup_{i\in\llbracket0,N+1\rrbracket}
\mathbb{E}\left[
z^{\mathscr{X}^{\boldsymbol{\mathsf{S}}^{\bs(\bg)}}(i)}
\right] <\infty
\end{equation}
is a finite constant depending only on $q$, $\a$, $\b$, and $z$. 
\end{proposition}

The proof of Proposition~\ref{lem:moment bound simpler chain} will be provided at the end of this subsection. Let us stress that the main difference between Proposition~\ref{prop:moment bound Bryc} and Proposition~\ref{lem:moment bound simpler chain} is that the path $\bg$ is chosen according to the replacement measure  $\mathbb{P}^{(0,\a,\b)}$ instead of $\mathrm{P}^{(0,\a,\b)}$. We now prove Proposition~\ref{prop:moment bound Bryc} modulo Proposition~\ref{lem:moment bound simpler chain} whose proof we defer to the end of this subsection.

\begin{proof}[Proof of Proposition \ref{prop:moment bound Bryc} assuming Proposition \ref{lem:moment bound simpler chain}]
Let $z_*=z^k>1$. Recall that $\bg\sim\mathbb{P}^{(0,\a,\b)}$ on $\lrw_N$. After enlarging the probability space if necessary, we sample the Markov chain $X^{\boldsymbol{\mathsf{S}}^{\bs(\bg)}}$ as in Definition \ref{def:characterization measure} conditionally on $\bg$ (equivalently on $\boldsymbol{\mathsf{S}}^{\bs(\bg)}$) using independent transition randomness. Then we have
\begin{equation}\label{eq:ttyujt1}
    \begin{split}
        \mathbb{E}\left[W_{\bs(\bg)}(z)^k\right]&=\mathbb{E}\left[\left(\mathbb{E}\left[z^{X^{\boldsymbol{\mathsf{S}}^{\bs(\bg)}}(N+1)}\Big{|}\bs\right]\right)^k\right]
        \leq\mathbb{E}\left[ \mathbb{E}\left[z^{kX^{\boldsymbol{\mathsf{S}}^{\bs(\bg)}}(N+1)}\Big{|}\bs\right] \right]\\
        &=\mathbb{E}\left[z_*^{X^{\boldsymbol{\mathsf{S}}^{\bs(\bg)}}(N+1)}\right]\leq \mathbb{E}\left[z_*^{\mathscr{X}^{\boldsymbol{\mathsf{S}}^{\bs(\bg)} }(N+1)}\right],
    \end{split}
\end{equation}
where in the first step we use the characterization of the Bryc polynomial from Theorem \ref{thm:prob interpretation of Bryc polynomials}, in the second step we use conditional Jensen's inequality, in the third step we use $z_*=z^k$, and in the last step we use the coupling constructed in Lemma \ref{lem:coupling birth death} (\romannum{2}).

By Lemma \ref{lem:TV distance replaced measure}, there exists a coupling of $\bg\sim\mathbb{P}^{(0,\a,\b)}$ and $\bg'\sim\mathrm{P}^{(0,\a,\b)}$ such that 
\begin{equation}\label{eq:nthtyjyjn5ej}
\mathbb{P}(\bg\neq\bg')\leq C_1e^{-c_1N},
\end{equation}
where the constants $C_1,c_1>0$ depend only on $\a$ and $\b$. Therefore,
\begin{equation}\label{eq:ttyujt}
    \begin{split}
        \mathbb{E}\left[z_*^{\mathscr{X}^{\boldsymbol{\mathsf{S}}^{\bs(\bg)} }(N+1)}\right]
        &=\mathbb{E}\left[z_*^{\mathscr{X}^{\boldsymbol{\mathsf{S}}^{\bs(\bg)} }(N+1)}\cdot{\bf1}_{\left\{\bg=\bg'\right\}}\right]
        +\mathbb{E}\left[z_*^{\mathscr{X}^{\boldsymbol{\mathsf{S}}^{\bs(\bg)} }(N+1)}\cdot{\bf1}_{\left\{\bg\neq\bg'\right\}}\right]\\
        &\leq \mathbb{E}\left[z_*^{\mathscr{X}^{\boldsymbol{\mathsf{S}}^{\bs(\bg')} }(N+1)} \right]
        +\mathbb{E}\left[z_*^{2\mathscr{X}^{\boldsymbol{\mathsf{S}}^{\bs(\bg)} }(N+1)} \right]^{1/2}\mathbb{P}\left(\bg\neq\bg'\right)^{1/2}\\
        &\leq C_2+\left(C_3 z_*^{2N^{3/4}}\right)^{1/2}\left(C_1e^{-c_1N}\right)^{1/2},
    \end{split}
\end{equation}
where in the last step we use Lemma \ref{lem:moment bound simpler chain} (with constant $C_2$ depending only on $q$, $\a$, $\b$, and $z_*$) to bound the first summand and Lemma \ref{lem:apriori bound birth death chain}
(with $z_*^2$ instead of $z_*$ and with constant $C_3$ depending only on $q$ and $z_*$)
combined with \eqref{eq:nthtyjyjn5ej} to bound the second summand.
The right side of \eqref{eq:ttyujt} can be uniformly bounded by a constant $C$ depending only on $q$, $\a$, $\b$, $z$, and $k$. 
Combining with \eqref{eq:ttyujt1} we conclude the proof.
\end{proof}

\subsection{Moment bounds on the birth-death chain under the replacement measure}\label{sec:MomentBDReplace}

In the following lemma, we bound the moments of the birth-death chain under replacement measure, allowing us to treat the process as a time-homogeneous birth-death process.

\begin{lemma}\label{lem:upward jumps U chain}
Assume $q\in(0,1)$.
For each $y\in\mathbb{Z}_{\geq0}$, define a Markov chain
$Q^y=\{Q^y(i)\}_{i=0}^\infty$ on $\mathbb{Z}_{\geq0}$ by $Q^y(0)=0$ and, for each
$i,j\in\mathbb{Z}_{\geq0}$,
\[
\mathbb{P}\left(Q^y(i+1)=j+1\mid Q^y(i)=j\right)=q^{y+j+1},
\qquad
\mathbb{P}\left(Q^y(i+1)=j\mid Q^y(i)=j\right)=1-q^{y+j+1}.
\]
Then, for each $z>1$, there exists a constant $C\in\mathbb{N}$  depending only on $q$ and $z$, such that 
\begin{equation}\label{eq:Q-poly-bound}
    \mathbb{E}\left[z^{Q^y(h)}\right]\leq 1+Cq^yh^{C} \quad\mbox{ for all }  y,h\in\mathbb{Z}_{\geq0}.
\end{equation} 
\end{lemma}

\begin{proof}
Let $G_m^y\sim\mathrm{Geo}(q^{y+m})$, $m\in\mathbb{N}$, be independent Geometric random variables, i.e.
\begin{equation}
    \mathbb P(G_m^y = x) = q^{y+m}(1-q^{y+m})^{x-1}
\end{equation} for all $x\in \mathbb{N}$.  By the strong Markov property,
\[
Q^y(h)=\sup\left\{m\in\mathbb{Z}_{\geq0}:G_1^y+\cdots+G_m^y\leq h\right\}.
\]
Therefore, for $m\in\mathbb{N}$ we have
\[
\mathbb{P}(Q^y(h)\geq m)
\leq  \prod_{r=1}^{m}\mathbb{P}(G_r^y\leq h) 
\leq \min\left(hq^{y+1},\prod_{r=1}^{m} \left(hq^{y+r}\right)\right)
\leq q^y \min\left(h, h^mq^{m^2/2}\right).
\]
Hence
\[
\mathbb{E}\left[z^{Q^y(h)}\right]
=1+\sum_{m=1}^{\infty}\left(z^m-z^{m-1}\right)\mathbb{P}(Q^y(h)\geq m)
\leq 1+q^y\sum_{m=1}^{\infty}  z^m \min\left(h, h^mq^{m^2/2}\right).
\]
It suffices to find $C\in\mathbb{N}$ depending only on $q$ and $z$, such that 
\[
\sum_{m=1}^{\infty}  z^m \min\left(h, h^mq^{m^2/2}\right)\leq Ch^C\quad\mbox{ for all } h\in\mathbb{Z}_{\geq0}.
\]
When $h=0$, both sides equal $0$. When $h\in\mathbb{N}$, we let $M_h= \lceil -4\log(zh)/(\log q) \rceil\in\mathbb{N}$ and notice that, when $m\geq M_h$, we have
$z^mh^mq^{m^2/2}
=\exp\left(m\log(zh)+m^2(\log q)/2\right)
\leq \exp\left( m^2(\log q)/4\right)$. Hence 
\begin{equation*}
    \begin{split}
        &\sum_{m=1}^{\infty}  z^m \min\left(h, h^mq^{m^2/2}\right)\leq h\sum_{m=1}^{M_h-1}z^m+\sum_{m=M_h}^{\infty}z^mh^mq^{m^2/2}\\
        &\leq  C_1hz^{M_h}+\sum_{m=1}^{\infty}e^{m^2(\log q)/4}\leq C_2 h^{1-4(\log z)/(\log q)}+C_3\leq Ch^{C},
    \end{split}
\end{equation*} 
for some $C\in\mathbb{N}$ depending only on $q$ and $z$ and for all $h\in\mathbb{N}$. 
\end{proof}

\begin{lemma}\label{lem:block moment bound simpler chain arbitrary start}
 Assume $q\in(0,1)$. Let $x,\ell\in\mathbb{Z}_{\geq0}$ and $h_0,\dots,h_{\ell-1}\in\mathbb{Z}_{\geq0}$. Define $Z(0)=x$ and 
\[
Z(i+1)= Z(i) +Q_i^{Z(i)}(h_i)-B_i\,{\bf1}_{\left\{Z(i)+Q_i^{Z(i)}(h_i)>0\right\}}\quad \mbox{ for }i\in\llbracket0,\ell-1\rrbracket,
\]
where, conditionally on $Z(i)$, $Q_i^{Z(i)}$ is an independent copy of the Markov chain from Lemma \ref{lem:upward jumps U chain}, $B_i$, $i\in\mathbb{Z}_{\geq0}$, are i.i.d. $\ber(p)$ random variables independent of all the $Q_i$'s, and $p\in[1-q,1]$. 
Assume $z>1$. There exist constants $C>0$  and $\rho\in(0,1)$ depending only on $q$ and $z$, such that 
\begin{equation}\label{eq:block moment bound simpler chain}
\mathbb{E}\left[z^{Z(\ell)}\right]
\leq C\left(1+\rho^\ell z^x+\sum_{i=0}^{\ell-1}\rho^{\ell-i} h_{i} ^C\right).
\end{equation} 
\end{lemma}

\begin{proof}
Let $C_0\in\mathbb N$ be the constant from Lemma
\ref{lem:upward jumps U chain}, so that
\begin{equation}\label{eq:Q-moment-used-arbitrary-start}
\mathbb E\left[z^{Q^y(h)}\right]
\leq 1+C_0q^yh^{C_0}
\qquad
\mbox{ for all }y,h\in\mathbb Z_{\geq0}.
\end{equation}
Let
$r_z=q+(1-q)/z\in(0,1)$,
and choose $\rho\in(r_z,1)$. For fixed $i\in\llbracket0,\ell-1\rrbracket$, write $S=y+Q_i^y(h_i)$, and we have 
\[
\mathbb E\left[z^{Z(i+1)}\big{|} Z(i)=y,\ Q_i^y(h_i)\right]={\bf1}_{\{S=0\}}
  +{\bf1}_{\{S>0\}}\left((1-p)z^S+pz^{S-1}\right)\leq (1-r_z)+r_z z^S ,
\]
where we use $1-p+p/z \leq q+(1-q)/z=r_z$ since $p\geq1-q$ and $z>1$. 
Taking the expectation over $Q_i^y(h_i)$ and using
\eqref{eq:Q-moment-used-arbitrary-start}, we obtain
\begin{equation}\label{eq:one-step-arbitrary-start-before-absorb}
\mathbb E\left[z^{Z(i+1)}\big{|} Z(i)=y\right]
\leq
(1-r_z)+r_zz^y+r_zC_0z^yq^yh_i^{C_0}.
\end{equation}
Note that $\rho-r_z>0$. We claim that there exists a constant $C_1\in\mathbb N$ depending only on $q$ and $z$, such that 
\begin{equation}\label{eq:yujm}
 r_z C_0 z^yq^yh^{C_0} \le (\rho-r_z)z^y+C_1\left(1+h^{C_1}\right) \quad \mbox{ for all } y,h\in\mathbb Z_{\ge0}. 
 \end{equation}
Indeed, when $h=0$ the left side of \eqref{eq:yujm} is $0$. We next assume $h\geq1$. If $r_zC_0q^yh^{C_0}\le\rho-r_z$, then the term $(\rho-r_z) z^y$ on the right side of \eqref{eq:yujm} dominates $ r_z C_0 z^yq^yh^{C_0}$. 
If $r_zC_0q^yh^{C_0}>\rho-r_z$, then, since $\log q<0$, 
\[
y< \left(C_0\log h+\log\left((r_zC_0)/(\rho-r_z)\right)\right)/ \left(-\log q\right),
\]
and hence  $z^y\le C'\left(1+h^{C_0\log z/(-\log q)}\right)$ for a constant $C'$ depending only on $q,z$. Since $q^y\le 1$, in this case
\[
 r_z C_0 z^yq^yh^{C_0}
 \le r_zC_0C'\left(h^{C_0}+h^{C_0+C_0\log z/(-\log q)}\right).
\]
After increasing $C_1\in\mathbb N$ if necessary (which still depends only on $q$ and $z$), the term $C_1\left(1+h^{C_1}\right)$ on the right side of  \eqref{eq:yujm} dominates $ r_z C_0 z^yq^yh^{C_0}$. We conclude the proof of \eqref{eq:yujm}. 

In view of \eqref{eq:one-step-arbitrary-start-before-absorb} and \eqref{eq:yujm}, there exists a constant $C_2\in\mathbb N$ depending only on $q$ and $z$, such that \[ \mathbb E\left[z^{Z(i+1)}\Big{|} Z(i)\right] \le \rho z^{Z(i)}+C_2\left(1+h_i^{C_2}\right). \] Hence \[ \mathbb E\left[z^{Z(i+1)}\right]\le \rho \mathbb E\left[z^{Z(i)}\right]+C_2\left(1+h_i^{C_2}\right) \quad\mbox{ for } i\in\llbracket0,\ell-1\rrbracket . \] Iterating, we get
\[ \mathbb E\left[z^{Z(\ell)}\right] \le \rho^\ell z^x +C_2\sum_{i=0}^{\ell-1}\rho^{\ell-1-i}\left(1+h_i^{C_2}\right) \le  \rho^\ell z^x+\frac{C_2}{1-\rho} +\frac{C_2}{\rho}\sum_{i=0}^{\ell-1}\rho^{\ell-i}h_i^{C_2} \le C\left(1+\rho^\ell z^x+\sum_{i=0}^{\ell-1}\rho^{\ell-i}h_i^C\right), \] for some constant $C\in\mathbb{N}$ depending only on $q$ and $z$. 
\end{proof}

As a final ingredient for Proposition \ref{lem:moment bound simpler chain}, we require 
the following result on regeneration times of the underlying random walk paths. 
We defer its proof to the appendix.  

\begin{lemma}\label{lem:regeneration indices MAIN} 
    Assume $\b > \max(\a,1)$ and $\a\b>1$. Consider a random path $\bg\sim\mathrm{P}^{(0,\mathsf{a},\mathsf{b})}$, where $\mathrm{P}^{(0,\mathsf{a},\mathsf{b})}$ is the probability measure on $\lrw_N$ as in Definition \ref{def:approx TASEP measure}. 
    Consider the composition $\bs(\bg)$ associated to $\bg$ as in \eqref{eq:Signature} and the $0$-$1$ sequence $\boldsymbol{\mathsf{S}}^{\bs(\bg)}\in\{0,1\}^{N+1}$ associated to the composition $\bs(\bg)$ as in Definition \ref{def:associated 01 seq}. 
    Then, after enlarging the probability space if necessary, there exist a random integer $K\in\mathbb{Z}_{\geq0}$, $\mathbb{N}$-valued random variables $R_i$ for $i\in\mathbb{N}$, and, on the event $\{K\geq1\}$, $\mathbb{N}$-valued random variables
    \begin{equation*}
        1\leq \mathrm{n}_1<\dots<\mathrm{n}_K\leq N,
    \end{equation*}
    such that 
\[
K=\max\{j\in\mathbb Z_{\geq0}:R_1+\cdots+R_j\leq N\},
\]
and satisfying the following properties, with constants $C_1,c_1>0$ depending only on $\a$ and $\b$:
    \begin{enumerate}
        \item [(\romannum{1})] On the event $\{K\geq 1\}$, we have
\begin{align*}
    &\mathsf{S}^{\bs(\bg)}(\mathrm{n}_i)=1\quad\mbox{ for all }i\in\llbracket1,K\rrbracket,\\
    &\mathrm{n}_i-\mathrm{n}_{i-1}=R_{K+2-i}\quad\mbox{ for all }i\in\llbracket2,K\rrbracket,\quad
        \mbox{and}\quad
        N-\mathrm{n}_K=R_1-1.
\end{align*}
        \item [(\romannum{2})] We have $R_i$ for $i\in\mathbb{N}$ are independent and $R_i$ for $i\geq2$ are i.i.d. Moreover, for all $y\in\mathbb{N}$ we have
        \begin{align*}
            & \mathbb{P}(K\geq 1, \mathrm{n}_1\geq y)\leq C_1e^{-c_1y},\quad \mathbb{P}(K\geq 1,  N-\mathrm{n}_K\geq y)\leq C_1e^{-c_1y},\\
            & \mathbb{P}(K=0)\leq C_1e^{-c_1N},\quad \mathbb{P}(R_i\geq y)\leq C_1e^{-c_1y}  \mbox{    for } i\in\mathbb{N},\quad\mbox{ and } \mathbb{P}(R_2=1)=\nu_{\b}(1).
        \end{align*} 
\end{enumerate} 
\end{lemma}

\begin{proof} [Proof of Proposition \ref{lem:moment bound simpler chain}]
When $i=0$, the left side of \eqref{eq:ryjutuymyu} equals $1$. In the rest of the proof, we assume $i\geq1$ and the notation from  Lemma \ref{lem:regeneration indices MAIN}. 
For $i\in\llbracket1,N+1\rrbracket$, we let $A=\max\left\{j\in\mathbb{Z}_{\geq0}: R_1+\dots+R_j\leq N+1-i\right\}\leq K$ and $J=K-A\in\mathbb Z_{\geq0}$. Note that the random variables $A$ and $J$ depend on $i$, and that on the event $\{K\geq1\}$, the indices among $\mathrm n_1,\dots,\mathrm n_K$ which lie in $\llbracket0,i-1\rrbracket$ are exactly $\mathrm n_1,\dots,\mathrm n_J$. 
For clarity, we split the rest of the proof into two steps.\\

\noindent\textbf{Step 1.} In this step, we conclude the proof of the lemma assuming the following claim: for each $C_0\in\mathbb N$ and $\rho\in(0,1)$, there exists a finite
constant $C_2$ depending only on $\a,\b,C_0,\rho$, such that for $N\in\mathbb N$,
$i\in\llbracket1,N+1\rrbracket$,
\begin{equation}\label{eq:assumestep1bound}
    \mathbb E\left[{\bf1}_{\{J\geq1\}}R_{A+1}^{C_0}\right]\leq C_2,\quad \mathbb E\left[
{\bf1}_{\{J\geq1\}}\sum_{j=2}^{J}\rho^{j-1}R_{A+j}^{C_0}
\right]\leq C_2.
\end{equation} 

We begin by defining a random $0$-$1$ sequence  $\bar{\boldsymbol{\mathsf{S}}}=(\bar{ \mathsf S}(j))_{j \in \llbracket 1,N\rrbracket}\in\{0,1\}^{N+1}$. 
On the event $\{K=0\}$, we set $\bar{\mathsf{S}}(j)=0$ for $j\in\llbracket0,N\rrbracket$. On the event $\{K\geq1\}$, we set
\[
\bar{\mathsf{S}}(j)={\bf1}_{\{j\in\{\mathrm n_1,\dots,\mathrm n_K\}\}}\quad\mbox{ for }j\in\llbracket0,N\rrbracket.
\]
By the first part of Lemma~\ref{lem:regeneration indices MAIN}, we have $\mathsf{S}^{\bs(\bg)}(j)\geq \bar{\mathsf{S}}(j)$ for $j\in\llbracket0,N\rrbracket$. By Lemma \ref{lem:coupling birth death} (\romannum{3}), \begin{equation}\label{eq:ytujmndt}\mathscr{X}^{\boldsymbol{\mathsf{S}}^{\bs(\bg)} }(i)\preceq \mathscr{X}^{\bar{\boldsymbol{\mathsf{S}}} }(i).\end{equation}
We define random variables $Y_1,Y_2,Y_3$. On $\{J=0\}$, set $Y_1=Y_2=Y_3=0$. On $\{J\geq1\}$, set $\bar{\boldsymbol{\mathsf S}}^{(1)}\in\{0,1\}^{\mathrm n_1+1}$ 
\[
\bar{\mathsf S}^{(1)}(j)={\bf1}_{\{j=\mathrm n_1\}}\quad\mbox{ for }j\in\llbracket0,\mathrm n_1\rrbracket. 
\]
If $J=1$, set $\bar{\boldsymbol{\mathsf S}}^{(2)}=\emptyset$; if $J\geq2$, define $\bar{\boldsymbol{\mathsf S}}^{(2)}\in\{0,1\}^{\mathrm n_J-\mathrm n_1}$ by
\[
\bar{\mathsf S}^{(2)}(j)={\bf1}_{\{j\in\{\mathrm n_2-\mathrm n_1-1,\dots,\mathrm n_J-\mathrm n_1-1\}\}} 
\quad\mbox{ for } j\in\llbracket0,\mathrm n_J-\mathrm n_1-1\rrbracket.
\]
If $i-\mathrm n_J-1=0$, set $\bar{\boldsymbol{\mathsf S}}^{(3)}=\emptyset$; if $i-\mathrm n_J-1\geq1$, define $\bar{\boldsymbol{\mathsf S}}^{(3)}\in\{0,1\}^{i-\mathrm n_J-1}$ to be the all-zero sequence. We use the convention that $\mathscr X^{\emptyset}(0)=0$ for the empty sequence. Conditionally on the regeneration data
\begin{equation*}
    \mathcal{G}=\sigma(K,\mathrm{n}_1,\dots,\mathrm{n}_K) , 
\end{equation*}
and on the event $\{J\geq1\}$, let $Y_1,Y_2,Y_3$ have marginal laws
\[
Y_1\overset{d}{=}\mathscr{X}^{\bar{\boldsymbol{\mathsf S}}^{(1)}}(\mathrm n_1+1),\quad
Y_2\overset{d}{=}\mathscr{X}^{\bar{\boldsymbol{\mathsf S}}^{(2)}}(\mathrm n_J-\mathrm n_1),\quad
Y_3\overset{d}{=}\mathscr{X}^{\bar{\boldsymbol{\mathsf S}}^{(3)}}(i-\mathrm n_J-1),
\]
and couple these variables with $\mathscr{X}^{\bar{\boldsymbol{\mathsf S}}}(i)$ using the corollary of Lemma \ref{lem:coupling birth death} (\romannum{4}) so that
\begin{equation}\label{eq:ntynhymjtu}
\mathscr{X}^{\bar{\boldsymbol{\mathsf S}}}(i)
\leq
Y_1+Y_2+Y_3\quad\mbox{ on }\{J\geq1\}.
\end{equation}
Therefore,
\begin{equation}\label{eq:betbsetesbhe}
    \begin{split}
        &\mathbb{E}\left[z^{\mathscr{X}^{\boldsymbol{\mathsf{S}}^{\bs(\bg)} }(i)}\right] \leq
        \mathbb{E}\left[z^{\mathscr{X}^{\bar{\boldsymbol{\mathsf{S}}} }(i)}\right]
        =\mathbb{E}\left[z^{\mathscr{X}^{\bar{\boldsymbol{\mathsf{S}}} }(i)}{\bf1}_{\{J=0\}}\right]+
        \mathbb{E}\left[z^{\mathscr{X}^{\bar{\boldsymbol{\mathsf{S}}} }(i)}{\bf1}_{\{J\geq1\}}\right] \\
        &\leq \mathbb{E}\left[z^{2\mathscr{X}^{\bar{\boldsymbol{\mathsf{S}}} }(i)} \right]^{1/2}\mathbb{P}(J=0)^{1/2}  
     +\mathbb{E}\left[z^{3Y_1}{\bf1}_{\{J\geq 1\}}\right]^{1/3}
         \mathbb{E}\left[z^{3Y_2}{\bf1}_{\{J\geq 1\}}\right]^{1/3}
         \mathbb{E}\left[z^{3Y_3}{\bf1}_{\{J\geq 1\}}\right]^{1/3}.
    \end{split}
\end{equation}
where the first step uses \eqref{eq:ytujmndt}, the second step is a splitting, the third step uses H\"older's inequality and \eqref{eq:ntynhymjtu}. 

It remains to bound the four factors on the right side of \eqref{eq:betbsetesbhe}. For the first factor, 
\begin{equation}\label{eq:nyjdrynse5y0}
\mathbb E\left[
z^{2\mathscr{X}^{\bar{\boldsymbol{\mathsf S}}}(i)}
\right]
\mathbb P(J=0) \leq \mathbb E\left[
z^{2\mathscr{X}^{\bar{\boldsymbol{\mathsf S}}}(i)}
\right]\left(\mathbb P(K=0)+\mathbb P(K\geq1,\mathrm n_1\geq i)\right) \leq
C_3 z^{2i^{3/4}} \left(C_1e^{-c_1N}+C_1e^{-c_1i}\right)\leq C_4,
\end{equation}
where $C_4$ depends only on $q$, $\a$, $\b$, and $z$. Here, in the first step we use $\{K\geq1,J=0\}\subset\{\mathrm n_1\geq i\}$, and in the second step we use Lemma \ref{lem:apriori bound birth death chain}   with $z^2$ in place of $z$, combined with the second part of Lemma~\ref{lem:regeneration indices MAIN}.

The next two block factors on the right side of \eqref{eq:betbsetesbhe} (involving $Y_1$ and $Y_2$, respectively) will be bounded using Lemma \ref{lem:block moment bound simpler chain arbitrary start}. Consider a deterministic block made of \(h_0\) zeros followed by a one, then \(h_1\) zeros followed by a one, and so on, ending after \(h_{\ell-1}\) zeros followed by a one. If the current state at the beginning of a zero run is \(y\), the evolution through \(h\) zeros has the same law as \(y+Q^y(h)\). Therefore the endpoint of such a block has the same law as \(Z(\ell)\) in Lemma \ref{lem:block moment bound simpler chain arbitrary start} with \(p=1-q\) and with the corresponding initial state. 

Next, we will use Lemma \ref{lem:block moment bound simpler chain arbitrary start} for $x=0$, $p=1-q$, with $z^3$ in place of $z$, and with constants $C_5\in\mathbb{N}$ and $\rho\in(0,1)$. For the first block on \(\{J\geq1\}\), there are \(\mathrm n_1\) zeros followed by a terminal one. We get that
\[
\mathbb E\left[
z^{3\mathscr{X}^{\bar{\boldsymbol{\mathsf S}}^{(1)}}(\mathrm n_1+1)}
\,\Big{|}\, \mathcal G
\right]
\leq C_5\left(1+ \mathrm{n}_1^{C_5}\right)
\quad\mbox{ on }\{J\geq1\}.
\]
Combining with \(\{J\ge1\}\subset \{K\ge1\}\) and that $\mathrm{n}_1$ has exponential tails from Lemma~\ref{lem:regeneration indices MAIN}, we obtain
\begin{equation}\label{eq:nyjdrynse5y1}
   \mathbb{E}\left[z^{3Y_1}{\bf1}_{\{J\geq 1\}}\right] \leq C_6,
\end{equation}
for a finite constant $C_6$ depending only on $q$, $\a$, $\b$, and $z$.

For the second block, the endpoint is \(0\) when \(J=1\).  When \(J\geq2\),
the successive zero-run lengths before the ones in
\(\bar{\boldsymbol{\mathsf S}}^{(2)}\) are $\mathrm n_2-\mathrm n_1-1,\mathrm n_3-\mathrm n_2-1,\dots,\mathrm n_J-\mathrm n_{J-1}-1$. 
For \(j\in\llbracket2,J\rrbracket\),  
$\mathrm n_j-\mathrm n_{j-1}=R_{K+2-j}$.
Thus,  
\begin{equation*} 
\mathbb E\left[
z^{3\mathscr{X}^{\bar{\boldsymbol{\mathsf S}}^{(2)}}(\mathrm n_J-\mathrm n_1)}
\Bigm| \mathcal G
\right] \leq
C_5\left(
1+\sum_{j=2}^{J}\rho^{J-j+1}R_{K+2-j}^{C_5}
\right)
=
C_5\left(
1+\sum_{j=2}^{J}\rho^{j-1}R_{A+j}^{C_5}
\right) \quad\mbox{ on }\{J\geq1\} ,
\end{equation*}
where the first step uses Lemma \ref{lem:block moment bound simpler chain arbitrary start} and the second step uses \(J=K-A\) and change of variables \(j\mapsto J+2-j\). Using the claim \eqref{eq:assumestep1bound}, we obtain 
\begin{equation}\label{eq:nyjdrynse5y2}
\mathbb E\left[
z^{3Y_2} {\bf1}_{\{J\geq1\}}
\right]\leq C_7,
\end{equation}
for a finite constant $C_7$ depending only on $q$, $\a$, $\b$, and $z$.  

For the third block, the sequence is all zeros.
Observe that $\mathscr{X}^{\bar{\boldsymbol{\mathsf S}}^{(3)}}(i-\mathrm n_J-1)
\overset{d}{=}Q^0(i-\mathrm n_J-1)$ conditionally on \(\mathcal G\) on the event \(\{J\geq1\}\), where \(Q^0\) is the chain from Lemma
\ref{lem:upward jumps U chain}. We have
\begin{equation}
\mathbb E\left[
z^{3\mathscr{X}^{\bar{\boldsymbol{\mathsf S}}^{(3)}}(i-\mathrm n_J-1)}
\Bigm| \mathcal G
\right]
\leq
1+C_8(i-\mathrm n_J-1)^{C_8}
\leq
1+C_8R_{A+1}^{C_8} 
\quad\mbox{ on }\{J\geq1\},
\end{equation} 
where the first step uses Lemma
\ref{lem:upward jumps U chain} with \(z^3\) in place of \(z\) and with constant $C_8\in\mathbb{N}$, and the second step uses the fact that $i-\mathrm n_J-1\leq R_{A+1}$, which follows from the definition of $A$ and $J$.
Using the claim \eqref{eq:assumestep1bound}, 
 \begin{equation}\label{eq:nyjdrynse5y3}
\mathbb E\left[
z^{3Y_3} {\bf1}_{\{J\geq1\}}
\right]\leq C_9,
\end{equation}
for a finite constant $C_9$ depending only on $q$, $\a$, $\b$, and $z$.   

Plugging \eqref{eq:nyjdrynse5y0}, \eqref{eq:nyjdrynse5y1}, \eqref{eq:nyjdrynse5y2}, and \eqref{eq:nyjdrynse5y3} into \eqref{eq:betbsetesbhe}, we obtain
\[
\mathbb{E} \left[z^{\mathscr{X}^{\boldsymbol{\mathsf{S}}^{\bs(\bg)} }(i)} \right]
\leq C_4^{1/2}+C_6^{1/3}C_7^{1/3}C_9^{1/3}\quad\mbox{ for all }N\in\mathbb{N},\ i\in\llbracket1,N+1\rrbracket,
\]
where the right side is a constant depending only on $q$, $\a$, $\b$, and $z$. Taking the supremum, we conclude the first step.\\

\noindent\textbf{Step 2.} In this step, we prove claim \eqref{eq:assumestep1bound}, which we restate as follows:  for each $C_0\in\mathbb N$ and $\rho\in(0,1)$, there exists a finite
constant $C_2$ depending only on $\a,\b,C_0,\rho$, such that for all $N\in\mathbb N$ and
$i\in\llbracket1,N+1\rrbracket$,
\begin{equation}\label{eq:assumestep1bound again}
    \mathbb E\left[{\bf1}_{\{J\geq1\}}R_{A+1}^{C_0}\right]\leq C_2,\quad \mathbb E\left[
{\bf1}_{\{J\geq1\}}\sum_{j=2}^{J}\rho^{j-1}R_{A+j}^{C_0}
\right]\leq C_2,
\end{equation} 
where we recall $A=\max\left\{j\in\mathbb{Z}_{\geq0}: R_1+\dots+R_j\leq N+1-i\right\}$ and $J=K-A$.
In the rest of the proof, we denote $d=N+1-i$ and $m_j=R_1+\dots+R_j$ for $j\in\mathbb{Z}_{\geq0}$. Then we have $0=m_0<m_1<\dots$ and $A=\max\left\{j\in\mathbb{Z}_{\geq0}: m_j\leq d\right\}$. We begin by proving the first inequality in \eqref{eq:assumestep1bound again}. For each $y\in\mathbb{N}$, we have
\begin{equation*}
    \begin{split}
        &\mathbb P(J\geq1,\ R_{A+1}\geq y)
        \leq\sum_{a=0}^{\infty}\mathbb P(m_a\leq d<m_{a+1},\ R_{a+1}\geq y)  \\
& \leq C_1e^{-c_1y}+\sum_{a=1}^{\infty}\sum_{\ell=0}^{\infty}
\mathbb P\left(m_a=d-\ell,\ R_{a+1}\geq \max(y,\ell+1)\right)\\
&\leq C_1e^{-c_1y}+
C_1\sum_{\ell=0}^{\infty}
e^{-c_1\max(y,\ell+1)}
\sum_{a=1}^{\infty}\mathbb P(m_a=d-\ell) \\
& \leq C_1e^{-c_1y}+
C_1\sum_{\ell=0}^{\infty}
e^{-c_1\max(y,\ell+1)}  \\
 &\leq C_{2}e^{-c_{2}y},
    \end{split}
\end{equation*}
for some constants $C_{2},c_{2}>0$ depending only on $\a,\b$. Here, the first step uses $A+1=\min\{j\in\mathbb N:m_j>d\}$; 
the second step uses the tail bounds from Lemma~\ref{lem:regeneration indices MAIN}; the third step uses change of variables $\ell=d-m_a$; the fourth step uses that $m_a$ is $\sigma(R_1,\ldots,R_a)$-measurable,  
$R_{a+1}$ is independent of $\sigma(R_1,\ldots,R_a)$, and again Lemma~\ref{lem:regeneration indices MAIN}; the fifth step uses the fact that $(m_a)_{a\in\mathbb{Z}_{\geq0}}$ is strictly increasing; and the sixth step is a change of the identity form. Therefore,
\[ 
\mathbb E\left[{\bf1}_{\{J\geq1\}}R_{A+1}^{C_0}\right]
=
\sum_{y=1}^{\infty}y^{C_0}\mathbb P(J\geq1,\ R_{A+1}=y)\\
\leq
\sum_{y=1}^{\infty}y^{C_0}\mathbb P(J\geq1,\ R_{A+1}\geq y)\\
\leq
C_{2}\sum_{y=1}^{\infty}y^{C_0} e^{-c_{2}y}<\infty ,
\]
where the right side is finite and depend only on $C_0$, $\a$ and $\b$. We conclude the first inequality of \eqref{eq:assumestep1bound again}.

We next prove the second inequality of \eqref{eq:assumestep1bound again}.
We have
\begin{equation}
    \begin{split}
       & \mathbb E\left[
{\bf1}_{\{J\geq1\}}\sum_{j=2}^{J}\rho^{j-1}R_{A+j}^{C_0}
\right]
\leq
\sum_{j=2}^{\infty}\rho^{j-1}
\mathbb E\left[{\bf1}_{\{J\geq1\}}R_{A+j}^{C_0}\right]
=\sum_{j=2}^{\infty}\rho^{j-1}
\sum_{a=0}^{\infty}
\mathbb E\left[{\bf1}_{\{J\geq1\}\cap\{A=a\}}R_{a+j}^{C_0}\right]\\
&=\sum_{j=2}^{\infty}\rho^{j-1}
\sum_{a=0}^{\infty}\mathbb P(J\geq1,\ A=a)\,\mathbb E\left[R_2^{C_0}\right]
=\sum_{j=2}^{\infty}\rho^{j-1}\mathbb P(J\geq1)\,\mathbb E\left[R_2^{C_0}\right]\\
&\leq \frac{\rho}{1-\rho}\cdot \sum_{y=1}^{\infty}y^{C_0}\mathbb P(R_2= y)\leq \frac{\rho}{1-\rho}\cdot \sum_{y=1}^{\infty}y^{C_0} C_1e^{-c_1y}<\infty,
    \end{split}
\end{equation}
where the right side is finite and depend only on $C_0$, $\rho$, $\a$ and $\b$. Here, in the first step we use linearity of expectation; in the second step we decompose according to the value of $A$; in the third step we note that the event $\{J\geq1\}\cap\{A=a\}$ belongs to
$\sigma(R_1,\ldots,R_{a+1})$, while $R_{a+j}$ is independent of this $\sigma$-field and has the same law as $R_2$; in the  fourth step we sum over $a$; in the fifth step we decompose according to the value of $R_2$, and in the sixth step we use Lemma~\ref{lem:regeneration indices MAIN}. We conclude the second inequality of \eqref{eq:assumestep1bound again}.
\end{proof}

\section{Approximation of the stationary distribution of the open ASEP}\label{sec:Approx}

In this section, we prove Theorem \ref{thm:ApproxHighLow}, which is our main approximation result of the stationary measure of open ASEP in the high and low density phases of the shock region. 
For clarity, we restate Theorem \ref{thm:ApproxHighLow} below. Recall that we write $\mu_N^{(q,\a,\b)}$ for the stationary distribution of the open ASEP. 
\begin{theorem}\label{thm:ApproxHighLow again}
Consider the open ASEP on the lattice $\llbracket1,N\rrbracket$ with jump rates $\alpha,\beta,\gamma,\delta$ and $q$ satisfying \eqref{eq:conditions open ASEP}. Assume Liggett's condition \eqref{eq:Liggett's condition} holds and that the open ASEP is in the shock region $\a\b>1$, where $\a=(1-\alpha)/\alpha$ and $\b=(1-\beta)/\beta$. Then there exist constants $C,c>0$ depending only on $q$, $\a$ and $\b$, such that for each $N\geq m\geq1$, the following hold. 
\begin{enumerate}
    \item [(\romannum{1})] In the high density phase $\b > \max(\a,1)$, we have
\begin{equation}\label{eq:approx HD}
    \TV{ \mu_N^{(q,\a,\b)}\left( (\tau_i)_{i \in\llbracket m,N\rrbracket} \in \cdot \,\right) - \ber_{\llbracket m,N\rrbracket}\left(\b/(1+\b) \right) }\leq Ce^{-cm}.
\end{equation} 
\item [(\romannum{2})] In the low density phase $\a > \max(\b,1)$, we have      
\begin{equation}\label{eq:approx LD}
    \TV{ \mu_N^{(q,\a,\b)}\left( (\tau_i)_{i \in\llbracket 1,m\rrbracket} \in  \cdot \,\right) - \ber_{\llbracket 1,m\rrbracket}\left(1/(1+\a) \right) }\leq Ce^{-c(N-m)}.
\end{equation} 
\end{enumerate}
\end{theorem}

We show that Theorem \ref{thm:ApproxHighLow again} can be reduced to the following approximation result of the path measure $\mathbb{P}^{(q,\a,\b)}$ on $\lrw_N$ as in \eqref{eq:PNab}. Recall that we write $\nu_{\b}$ for $\b>0$ to denote the probability measure on $\{-1,0,1\}$ given by
    \[
    \nu_{\b}(1)=\frac{\b}{\b+\b^{-1}+2},\quad\nu_{\b}(0)=\frac{2}{\b+\b^{-1}+2},\quad \nu_{\b}(-1)=\frac{\b^{-1}}{\b+\b^{-1}+2},
    \]
and $\nu_{\b}^{\otimes k}$ for $k\in\mathbb{N}$ to denote the product measure (of $k$ independent copies of $\nu_{\b}$) on $\{-1,0,1\}^k$.

\begin{proposition}\label{prop:suffices to prove}
    Assume $q\in[0,1)$, $\b > \max(\a,1)$ and $\a\b>1$. Consider the probability measure 
    \begin{equation}\label{eq:PNab again}
\mathbb{P}^{(q,\a,\b)}(\bg)=\frac{1}{Z_N^{(q,\a,\b)}}\frac{2^{H(\bg)}\b^{\gamma_N}}{(\a\b)^{\min(\bg)}} W_{\bs(\bg)}(\a\b)\quad\mbox{ for }\bg\in\lrw_N. 
\end{equation}
    Then there exist constants $C,c>0$ depending only on $q$, $\a$ and $\b$, such that for each $N\geq m\geq1$, 
\begin{equation}\label{eq:approx of path measure HD}
     \TV{ \mathbb{P}^{(q,\a,\b)}\left( (\gamma_{i}-\gamma_{i-1})_{i \in\llbracket m+1,N\rrbracket} \in \cdot \,\right) - \nu_{\b}^{\otimes (N-m)}    }\leq Ce^{-cm}.
\end{equation} 
\end{proposition}
\begin{proof}[Proof of Theorem \ref{thm:ApproxHighLow again} assuming Proposition \ref{prop:suffices to prove}]
By Theorem \ref{thm:characterization stationary measure ASEP}, the stationary measure $\mu_N^{(q,\a,\b)}$ of open ASEP on $\{0,1\}^N$ is a pushforward of the measure $\mathbb{P}^{(q,\a,\b)}$ on $\lrw_N$. Specifically, at each site $i\in\llbracket1,N\rrbracket$, we set $\tau_i=1$ if the path $\bg$ takes a step $+1$ at time $i$; $\tau_i=0$ if $\bg$ takes a step $-1$ at time $i$; and $\tau_i$ to be a Bernoulli random variable with mean $1/2$ if $\bg$ takes a step $0$ at time $i$. Since the pushforward of $\nu_{\b}$ equals $\ber \left(\b/(1+\b) \right)$, we conclude \eqref{eq:approx HD}, hence Theorem \ref{thm:ApproxHighLow again} holds in the high density phase of the shock region.  
By the particle-hole duality of the open ASEP; see for example \cite[Lemma 2.7]{yang2025limits} and \cite[Section 4.2]{wang2023askey}, the stationary measures $\mu_N^{(q,\a,\b)}$ and $\mu_N^{(q,\b,\a)}$ are related by a transformation of occupation variables
\[
\widetilde{\tau}_i=1-\tau_{N+1-i}\quad\mbox{ for } i\in\llbracket1,N\rrbracket.
\] 
This transformation swaps the high density phase with the low density phase, and swaps $\ber \left(\b/(1+\b) \right)$ with $\ber \left(1/(1+\a)\right)$. We conclude \eqref{eq:approx LD}, and Theorem \ref{thm:ApproxHighLow again} holds in the low density phase of the shock region.  
\end{proof}

Hence, in the remainder of this section, we will focus on proving Proposition \ref{prop:suffices to prove}. 

\subsection{Path measure approximation for the open TASEP}

This subsection is devoted to establishing the following approximation result on the path measure for the open ASEP.  
\begin{lemma}\label{lem:approx of paths HD TASEP}
    Assume $\b > \max(\a,1)$ and $\a\b>1$. Consider the probability measure $\mathbb{P}^{(0,\a,\b)}$ on $\lrw_N$ as in \eqref{eq:PNab TASEP}. Then there exist constants $C,c>0$ depending only on $\a$ and $\b$, such that for each $N\geq m\geq1$, 
    \begin{equation}\label{eq:approx of path measure HD TASEP}
     \TV{ \mathbb{P}^{(0,\a,\b)}\left( (\gamma_{i}-\gamma_{i-1})_{i \in\llbracket m+1,N\rrbracket} \in  \cdot\, \right) - \nu_{\b}^{\otimes (N-m)}    }\leq Ce^{-cm}.
\end{equation}
\end{lemma}

Lemma \ref{lem:approx of paths HD TASEP} will serve as a main input for the proof of Proposition \ref{prop:suffices to prove} in Section~\ref{sec:PropProof}, where we combine the approximation of the open TASEP path measure with suitable tail bounds on the endpoint of the underlying birth-death chain in the probabilistic characterization of the Bryc polynomial. As a preparation for the proof of Lemma \ref{lem:approx of paths HD TASEP}, we start with the following approximation result for the $\nu_{\b}$-random walk with $N$ steps.

\begin{lemma}\label{lem:appro conditioned random walk}
Assume $\b>1$. Let $\{Y_i\}_{i=0}^N$ be a $\nu_{\b}$-random walk with $N$ steps started at $Y_0=0$, conditioned to stay non-negative. 
Then there exist constants $C,c>0$ depending only on $\b$, such that for each $N\geq m\geq1$, 
    \begin{equation}
        \TV{ \mathbb{P}((Y_{i}-Y_{i-1})_{i\in\llbracket m+1,N\rrbracket} \in \cdot\, ) -\nu_{\b}^{\otimes (N-m)}} \leq Ce^{-cm}.
    \end{equation} 
\end{lemma}
\begin{proof}
Let $\kappa=(\b-1)/(2(\b+1))$.
We write $\{S_i\}_{i=0}^N$ to denote an unconditioned $\nu_{\b}$-random walk started at $S_0=0$. By Lemma \ref{lem:basic hitting monotone coupling}, there exist $C_1,c_1>0$ depending only on $\b$
such that for each $N\ge m\ge1$, we have 
$\mathbb{P}(Y_m<\kappa m)\leq\mathbb{P}\left(S_m<\kappa m \right) 
\le C_1e^{-c_1m}$.
Therefore, for each $A\subset\{-1,0,1\}^{N-m}$, we have
\begin{equation}\label{eq:brbtgrb1}
    \begin{split}
        &\mathbb{P}\left((Y_{i}-Y_{i-1})_{i\in\llbracket m+1,N\rrbracket} \in A\right)-\nu_{\b}^{\otimes (N-m)}(A)\\
        &=\sum_{h\geq\lceil\kappa m\rceil}\mathbb{P}(Y_m=h)\cdot\left(\mathbb{P}\left((Y_{i}-Y_{i-1})_{i\in\llbracket m+1,N\rrbracket} \in A\big{|}Y_m=h\right)-\nu_{\b}^{\otimes (N-m)}(A)\right)\\
        &\quad+\mathbb{P}\left(\{Y_m <\kappa m\}\cap\{(Y_{i}-Y_{i-1})_{i\in\llbracket m+1,N\rrbracket} \in A\}\right)-\mathbb{P}( Y_m< \kappa m )\cdot\nu_{\b}^{\otimes (N-m)}(A) \\
&\leq\sup_{h\geq\lceil\kappa m\rceil}
        \TV{\mathbb{P}\left((Y_{i}-Y_{i-1})_{i\in\llbracket m+1,N\rrbracket}\in\cdot\,\big{|}Y_m=h\right)-\nu_{\b}^{\otimes (N-m)}}+C_1e^{-c_1m}.
    \end{split}
\end{equation}
For $h\in\mathbb{N}$, we note that the law of $(Y_{i}-Y_m)_{i\in\llbracket m,N\rrbracket}$ conditioned on $Y_m=h$ is that of a $\nu_{\b}$-random walk started at 0 and conditioned to stay at or above level $-h$ up to step $N-m$. Denote the conditioned event by $E_h$, then we have 
$\mathbb{P}(E_h)\geq1-\b^{-2h-2}$ in view of Lemma \ref{lem:basic hitting monotone coupling}. Therefore, 
\begin{equation}\label{eq:brbtgrb2}
    \begin{split}
        &\TV{\mathbb{P}\left((Y_{i}-Y_{i-1})_{i\in\llbracket m+1,N\rrbracket}\in\cdot\,\big{|}Y_m=h\right)-\nu_{\b}^{\otimes (N-m)}}\\
        &=\TV{\nu_{\b}^{\otimes (N-m)}(\,\cdot\,|E_h)-\nu_{\b}^{\otimes (N-m)}}
        \le \frac{\nu_{\b}^{\otimes (N-m)} (E_h^{\complement} )}{\nu_{\b}^{\otimes (N-m)}(E_h)}
        \le \frac{\b^{-2h-2}}{1-\b^{-2h-2}}
        \le \frac{1}{(\b^2-1)\b^{2h}}.
    \end{split}
\end{equation}
Combining \eqref{eq:brbtgrb1} and \eqref{eq:brbtgrb2}, we conclude the proof.
\end{proof}

We are now ready to prove Lemma \ref{lem:approx of paths HD TASEP}.
\begin{proof}[Proof of Lemma \ref{lem:approx of paths HD TASEP}]
    Let $\bg\sim\mathbb{P}^{(0,\a,\b)}$. Recall from Lemma \ref{lem:decomposition TASEP} that, conditionally on $T_N=k$, the law of $(\gamma_{k+i}-\gamma_{k})_{i\in\llbracket 0,N-k\rrbracket}$ is that of a $\nu_{\b}$-random walk started at $0$ and conditioned to stay non-negative up to step $N-k$. Therefore, for each $A\subset\{-1,0,1\}^{N-m}$, we have
    \begin{equation*}
        \begin{split}
            &\mathbb{P} \left((\gamma_{i}-\gamma_{i-1})_{i\in\llbracket m+1,N\rrbracket} \in A\right)-\nu_{\b}^{\otimes (N-m)}(A)\\
&=\sum_{k\leq\lfloor m/2\rfloor}\mathbb{P}(T_N=k)\cdot\left(\mathbb{P}\left((\gamma_{i}-\gamma_{i-1})_{i\in\llbracket m+1,N\rrbracket} \in A\big{|}T_N=k\right)-\nu_{\b}^{\otimes (N-m)}(A)\right)\\
        &\quad+\mathbb{P}\left(\{T_N > m/2\}\cap\{(\gamma_{i}-\gamma_{i-1})_{i\in\llbracket m+1,N\rrbracket} \in A\}\right)-\mathbb{P}( T_N>m/2 )\cdot\nu_{\b}^{\otimes (N-m)}(A) \\ 
        &\leq\max_{0\leq k\leq\lfloor m/2\rfloor}
        \TV{\mathbb{P}\left(\left( Y^{(N-k)}_i-Y^{(N-k)}_{i-1} \right)_{i\in\llbracket m+1-k,N-k\rrbracket}\in\cdot\,\right)- \nu_{\b}^{\otimes (N-m)}}
        +\mathbb{P}\left(T_N > m/2\right)\\
        &\leq \max_{0\leq k\leq\lfloor m/2\rfloor}C_1e^{-c_1(m-k)}+C_2e^{-c_2m/2}\leq Ce^{-cm}, 
        \end{split}
    \end{equation*}
    for some $C,c>0$ depending only on $\a$ and $\b$. 
    Here, we write $Y_i^{(N-k)}$, $i\in\llbracket0,N-k\rrbracket$ to denote a $\nu_{\b}$-random walk with $N-k$ steps started at $Y^{(N-k)}_0=0$, conditioned to stay non-negative. The first step follows from the definition of conditional probability; the second step uses the fact that conditionally on $T_N=k$, the law of $(\gamma_{k+i}-\gamma_{k})_{i\in\llbracket 0,N-k\rrbracket}$ is that of $(Y_i^{(N-k)})_{i\in\llbracket 0,N-k\rrbracket}$, and hence the law of $(\gamma_{i}-\gamma_{i-1})_{i\in\llbracket m+1,N\rrbracket}$ is that of $( Y^{(N-k)}_i-Y^{(N-k)}_{i-1} )_{i\in\llbracket m+1-k,N-k\rrbracket}$; the third step uses Lemma \ref{lem:appro conditioned random walk} (with $C_1,c_1>0$ depending on $\b$) to bound the first summand and Lemma \ref{lem:TN exponential} (with $C_2,c_2>0$ depending on $\a,\b$) to bound the second summand.  
\end{proof}

\subsection{Endpoints under the birth-death chain characterization: Proof of Proposition~\ref{prop:suffices to prove}}\label{sec:PropProof}

In the following, we utilize the probabilistic characterization of the Bryc polynomial in order conclude the proof of Proposition~\ref{prop:suffices to prove}. The next result on the change of endpoints of the birth-death chain in the probabilistic characterization of the Bryc polynomial is our main input.

\begin{lemma}\label{lem:conditioning does not change much endpoint}
Assume \(q\in(0,1)\), \(\b>\max(\a,1)\), and \(\a\b>1\). Then there exist
constants \(C,c,\delta>0\), depending only on \(q,\a,\b\), such that the
following holds. For every \(N,m\in\mathbb N\) with \(N\ge m\), and every event
\[
A\in\sigma\left((\gamma_i-\gamma_{i-1})_{i\in\llbracket m+1,N\rrbracket}\right)
\quad\mbox{ satisfying }\quad
\mathbb P^{(0,\a,\b)}(A)\ge e^{-\delta m},
\]
we have
\[
\TV{
\mathbb{P}^{(0,\a,\b)}\left(
X^{\boldsymbol{\mathsf{S}}^{\bs(\bg)}}(N+1)\in\cdot\,\Big|\,A
\right)
-
\mathbb{P}^{(0,\a,\b)}\left(
X^{\boldsymbol{\mathsf{S}}^{\bs(\bg)}}(N+1)\in\cdot
\right)
}
\le Ce^{-cm}.
\]
Here, in both probabilities, one first samples
\(\bg\in\lrw_N\) under the indicated law, forms the composition
\(\bs(\bg)\) as in \eqref{eq:Signature}, forms the associated \(0\)-\(1\)
sequence \(\boldsymbol{\mathsf S}^{\bs(\bg)}\) as in Definition
\ref{def:associated 01 seq}, and then samples the Markov chain
\(X^{\boldsymbol{\mathsf S}^{\bs(\bg)}}\) from Definition
\ref{def:characterization measure} using transition randomness independent of
\(\bg\).
\end{lemma} 

The proof of Lemma~\ref{lem:conditioning does not change much endpoint} will be provided in Section~\ref{sec:LemmaSupSup}. We first argue that assuming Lemma \ref{lem:conditioning does not change much endpoint}, we can conclude the proof of Proposition \ref{prop:suffices to prove}. 

\begin{proof}[Proof of Proposition \ref{prop:suffices to prove} assuming Lemma \ref{lem:conditioning does not change much endpoint}] 
For any $q\in[0,1)$, $\b > \max(\a,1)$ and $\a\b>1$, we need to show that there exist constants $C,c>0$ depending only on $q$, $\a$ and $\b$, such that for each $N\geq m\geq1$, 
\begin{equation}\label{eq:approx of path measure HD again}
     \TV{ \mathbb{P}^{(q,\a,\b)}\left( (\gamma_{i}-\gamma_{i-1})_{i \in\llbracket m+1,N\rrbracket} \in \cdot \,\right) - \nu_{\b}^{\otimes (N-m)}    }\leq Ce^{-cm},
\end{equation} 
where we recall that
    \begin{equation}\label{eq:PNab again again}
\mathbb{P}^{(q,\a,\b)}(\bg)=\frac{1}{Z_N^{(q,\a,\b)}}\frac{2^{H(\bg)}\b^{\gamma_N}}{(\a\b)^{\min(\bg)}} W_{\bs(\bg)}(\a\b)\quad\mbox{ for }\bg\in\lrw_N. 
\end{equation}
When $q=0$, using Lemma \ref{lem:approx of paths HD TASEP}, we have 
    \begin{equation}\label{eq:approx of path measure HD TASEP again}
     \TV{ \mathbb{P}^{(0,\a,\b)}\left( (\gamma_{i}-\gamma_{i-1})_{i \in\llbracket m+1,N\rrbracket} \in  \cdot\, \right) - \nu_{\b}^{\otimes (N-m)}    }\leq C_1e^{-c_1m},
\end{equation}
for some constants $C_1,c_1>0$ depending only on $\a$ and $\b$. Therefore \eqref{eq:approx of path measure HD again} holds for $q=0$. It suffices to establish \eqref{eq:approx of path measure HD again} for $q\in(0,1)$. For clarity, we split the rest of the proof into two steps. 

In the rest of the proof, we denote $\mathbb{P}^{(q,\a,\b)}=\mathbb{P}^{q}$ and $\mathbb{E}^{(q,\a,\b)}=\mathbb{E}^{q}$.
\\

\noindent\textbf{Step 1.} In this step, we conclude the proof of \eqref{eq:approx of path measure HD again} assuming the following claim: there exist constants $C_2,c_2>0$ depending only on $q$, $\a$ and $\b$, such that for each $N\geq m\geq1$ and each $A\in\sigma\left((\gamma_i-\gamma_{i-1})_{i\in\llbracket m+1,N\rrbracket}\right)$ with $\mathbb{P}^0(A)>0$, we have
\begin{equation}\label{eq:btrwbtwarvwg}
    \mathbb{P}^0(A)\cdot\left|\mathbb{E}^0\left[W_{\bs(\bg)}(\a\b)|A\right]-\mathbb{E}^0\left[W_{\bs(\bg)}(\a\b)\right] \right|\leq C_2e^{-c_2m}.
\end{equation}
Indeed, in view of \eqref{eq:PNab again again}, we have
\begin{equation}\label{eq:btmio}
\mathbb{P}^q(\bg)=\frac{W_{\bs(\bg)}(\a\b)}{\mathbb{E}^0\left[W_{\bs(\bg)}(\a\b)\right]}\mathbb{P}^0(\bg) \quad\mbox{ for }\bg\in\lrw_N. 
\end{equation}
Let $A\in\sigma\left((\gamma_i-\gamma_{i-1})_{i\in\llbracket m+1,N\rrbracket}\right)$.
If $\mathbb{P}^0(A)=0$, then $\mathbb{P}^q(A)=0$ by \eqref{eq:btmio}, and hence 
\begin{equation}\label{eq:grwtrt1}
\left|\mathbb{P}^q(A)-\mathbb{P}^0(A)\right|=0.
\end{equation} 
If $\mathbb{P}^0(A)>0$, then 
\begin{equation}\label{eq:grwtrt2}
    \begin{split}
        \left|\mathbb{P}^q(A)-\mathbb{P}^0(A)\right|
        &=\frac{\mathbb{P}^0(A)}{\mathbb{E}^0\left[W_{\bs(\bg)}(\a\b)\right]}\left|\mathbb{E}^0\left[W_{\bs(\bg)}(\a\b)\big{|}A\right]-\mathbb{E}^0\left[W_{\bs(\bg)}(\a\b)\right] \right|\\
        &\leq \mathbb{P}^0(A)\left|\mathbb{E}^0\left[W_{\bs(\bg)}(\a\b)\big{|}A\right]-\mathbb{E}^0\left[W_{\bs(\bg)}(\a\b)\right] \right|\leq C_2e^{-c_2m},
    \end{split}
\end{equation} 
where in the first step we use \eqref{eq:btmio}, in the second step we use $\a\b>1$ and the fact that 
\[
W_{\bs}(z)=\mathbb{E}\left[ z^{X^{\boldsymbol{\mathsf{S}}^{\bs}}(N+1)}\right]\geq1\quad\mbox{ for any }z\geq1\mbox{ and any composition }\bs\mbox{ of }N+1,
\]
and in the third step we use \eqref{eq:btrwbtwarvwg}. Combining \eqref{eq:grwtrt1}, \eqref{eq:grwtrt2} and take the supremum over all events in $\sigma\left((\gamma_i-\gamma_{i-1})_{i\in\llbracket m+1,N\rrbracket}\right)$, we conclude that
\[ 
     \TV{ \mathbb{P}^q\left( (\gamma_{i}-\gamma_{i-1})_{i \in\llbracket m+1,N\rrbracket} \in  \cdot\, \right) - \mathbb{P}^0\left( (\gamma_{i}-\gamma_{i-1})_{i \in\llbracket m+1,N\rrbracket} \in  \cdot\, \right)   }\leq C_2e^{-c_2m}.
\]
Combining the above bound with \eqref{eq:approx of path measure HD TASEP again}, we conclude the proof of \eqref{eq:approx of path measure HD again}.\\

\noindent\textbf{Step 2.} In this step, we prove that there exist constants $C_2,c_2>0$ depending only on $q$, $\a$ and $\b$, such that for each $N\geq m\geq1$ and each $A\in\sigma\left((\gamma_i-\gamma_{i-1})_{i\in\llbracket m+1,N\rrbracket}\right)$ with $\mathbb{P}^0(A)>0$, we have
\begin{equation}\label{eq:btrwbtwarvwg again}
    \mathbb{P}^0(A)\cdot\left|\mathbb{E}^0\left[W_{\bs(\bg)}(\a\b)|A\right]-\mathbb{E}^0\left[W_{\bs(\bg)}(\a\b)\right] \right|\leq C_2e^{-c_2m}.
\end{equation}
Using Proposition \ref{prop:moment bound Bryc} with $z=\a\b>1$, $k=2$ and $z=(\a\b)^2>1$, $k=1$, there exist $M\in(0,\infty)$ depending only on $q$, $\a$ and $\b$,  such that
\begin{equation}\label{eq:ntneyjetyjn}
\sup_{N\in\mathbb{N}}\mathbb{E}^0\left[W_{\bs(\bg)}(\a\b)^2\right]\leq M \quad\mbox{ and }\quad
\sup_{N\in\mathbb{N}}\mathbb{E}^0\left[W_{\bs(\bg)}\left((\a\b)^2\right) \right]\leq M.
\end{equation}
Choose constants $C_3,c_3,\delta>0$ depending on $q$, $\a$ and $\b$ from Lemma \ref{lem:conditioning does not change much endpoint}. 
Let $A\in\sigma\left((\gamma_i-\gamma_{i-1})_{i\in\llbracket m+1,N\rrbracket}\right)$. When $\mathbb{P}^0(A)\geq e^{-\delta m}$, by Lemma \ref{lem:conditioning does not change much endpoint}, there exists a coupling $\mathbb{P}$ between random variables 
\begin{equation}\label{eq:srthth}
X_A=X^{\boldsymbol{\mathsf{S}}^{\bs}}(N+1),\quad \bg\sim\mathbb{P}^0(\cdot\,|A)\quad \mbox{ and }\quad 
X=X^{\boldsymbol{\mathsf{S}}^{\bs}}(N+1),\quad \bg\sim\mathbb{P}^0,
\end{equation}
satisfying 
\begin{equation}\label{eq:rthrth}
    \mathbb{P}(X_A\neq X)\leq C_3e^{-c_3m}.
\end{equation}
Note also that, by the second inequality in \eqref{eq:ntneyjetyjn}, we have
\begin{equation}\label{eq:ntyyysteh}
    \mathbb{E}\left[(\a\b)^{2X }\right]=\mathbb{E}^0\left[W_{\bs(\bg)}\left((\a\b)^2\right) \right]\leq M, \quad\mbox{ and } \quad 
    \mathbb{E}\left[(\a\b)^{2X_A }\right]=\mathbb{E}^0\left[W_{\bs(\bg)}\left((\a\b)^2\right)\big{|}A \right]\leq \frac{M}{\mathbb{P}^0(A)}.
\end{equation}
Therefore,
\begin{equation}\label{eq:ukjt1}
\begin{split} 
     &\mathbb{P}^0(A)\cdot\left|\mathbb{E}^0\left[W_{\bs(\bg)}(\a\b)|A\right]-\mathbb{E}^0\left[W_{\bs(\bg)}(\a\b)\right] \right| 
      =\mathbb{P}^0(A)\cdot\left|\mathbb{E} \left[(\a\b)^{X_A}\right]-\mathbb{E} \left[ (\a\b)^X\right] \right|\\
     &\leq\mathbb{P}^0(A)\cdot \mathbb{E} \left[\left|(\a\b)^{X_A}- (\a\b)^X\right|\cdot{\bf1}_{\{X_A\neq X\}}\right]  
      \leq\mathbb{P}^0(A)\cdot \mathbb{E} \left[\left|(\a\b)^{X_A}- (\a\b)^X\right|^2\right]^{1/2}\cdot\mathbb{P}\left(X_A\neq X\right)^{1/2}   \\
     &\leq\mathbb{P}^0(A)\cdot \mathbb{E} \left[ 2(\a\b)^{2X_A}+ 2(\a\b)^{2X} \right]^{1/2}\cdot\mathbb{P}\left(X_A\neq X\right)^{1/2}   \\
     &\leq 2M^{1/2}  \cdot  \mathbb{P}^0(A)^{1/2}\cdot\mathbb{P}\left(X_A\neq X\right)^{1/2}
     \leq 2M^{1/2} C_3^{1/2}e^{-c_3m/2},
\end{split}
\end{equation}
where the first step uses Theorem \ref{thm:prob interpretation of Bryc polynomials} and definitions \eqref{eq:srthth} of $X,X_A$, the second step uses triangle inequality, the third and fourth step use the Cauchy-Schwarz inequality, the fifth step uses \eqref{eq:ntyyysteh}, and the last step uses \eqref{eq:rthrth}.
When $\mathbb{P}^0(A)\in(0,e^{-\delta m})$, 
\begin{equation}\label{eq:ukjt2}
\begin{split} 
     &\mathbb{P}^0(A)\cdot\left|\mathbb{E}^0\left[W_{\bs(\bg)}(\a\b)|A\right]-\mathbb{E}^0\left[W_{\bs(\bg)}(\a\b)\right] \right| 
      \leq \mathbb{E}^0\left[W_{\bs(\bg)}(\a\b)\cdot {\bf 1}_A\right]+\mathbb{P}^0(A)\cdot\mathbb{E}^0\left[W_{\bs(\bg)}(\a\b)\right]\\
     &\leq\mathbb{E}^0\left[W_{\bs(\bg)}(\a\b)^2\right]^{1/2}\cdot\mathbb{P}^0(A)^{1/2}
     +\mathbb{P}^0(A)\cdot\mathbb{E}^0\left[W_{\bs(\bg)}(\a\b)^2\right]^{1/2}\leq 2M^{1/2}e^{-\delta m/2},
\end{split}
\end{equation}
where the first step uses the triangle inequality, the second step uses the Cauchy-Schwarz inequality, and the last step uses  the first inequality in \eqref{eq:ntneyjetyjn} together with our assumption that $\mathbb{P}^0(A) <e^{-\delta m}$.
Combining \eqref{eq:ukjt1} and \eqref{eq:ukjt2}, we conclude the proof of \eqref{eq:btrwbtwarvwg again}.
\end{proof}

\subsection{Stability of the lower level compositions}

In remainder of this section, we work towards establishing Lemma \ref{lem:conditioning does not change much endpoint}. We have the following total-variation bound on the law of local times under the open TASEP path measures when conditioning on events $A$ of non-negligible probability.

\begin{lemma}\label{lem:conditioning does not affect first compositions}
     Assume $\b > \max(\a,1)$ and $\a\b>1$. Then there exist constants $C,c,\delta,\kappa>0$ depending only on $\a$ and $\b$, such that for each $N\geq m\geq1$ and each event 
    \begin{equation}\label{eq:gergrgegr}
    A\in\sigma\left((\gamma_i-\gamma_{i-1})_{i\in\llbracket m+1,N\rrbracket}\right)\quad\mbox{ satisfying } \quad \mathbb{P}^{(0,\a,\b)}(A)\geq e^{-\delta m},
    \end{equation}
    we have 
    \begin{equation}\label{eq:jnyt5r}
        \TV{\mathbb{P}^{(0,\a,\b)}\left(\left(\sigma_0,\dots,\sigma_{\lfloor\kappa m\rfloor}\right)\in\cdot\,|A\right)-
        \mathbb{P}^{(0,\a,\b)}\left(\left(\sigma_0,\dots,\sigma_{\lfloor\kappa m\rfloor}\right)\in\cdot\,\right)   }\leq Ce^{-cm},
    \end{equation}
    where $\bs(\bg)=(\sigma_0,\dots,\sigma_r)$ is the composition associated to $\bg$ given by \eqref{eq:Signature}, and we let $\sigma_i=0$ for all $i>r$.
\end{lemma}
\begin{proof}
We begin by introducing some notation.
For $\bg=(\gamma_0,\dots,\gamma_N)\in\lrw_N$ and $m\in\llbracket0,N\rrbracket$, we consider the truncated path $(\gamma_0,\dots,\gamma_m)\in\lrw_m$ and the associated composition
\[
\bs\left((\gamma_i)_{i\in\llbracket0,m\rrbracket}\right)=\left(\sigma_0^{(m)},\dots,\sigma_{r_m}^{(m)}\right),
\]
which is a composition of $m+1$. We extend it by letting $\sigma_i^{(m)}=0$ for all $i>r_m$.
By \eqref{eq:Signature},  we have 
\begin{equation}\label{eq:rhtrbpoi}
\sigma_i=\#\left\{j\in\llbracket0,N\rrbracket:\gamma_j=\min_{k\in\llbracket0,N\rrbracket}\gamma_k+i\right\},  \quad 
\sigma_i^{(m)}=\#\left\{j\in\llbracket0,m\rrbracket:\gamma_j=\min_{k\in\llbracket0,m\rrbracket}\gamma_k+i\right\}\quad\mbox{ for all  }i\in\mathbb{Z}_{\geq0}.
\end{equation}
We will prove the lemma for $\kappa=(\b-1)/(3(\b+1))>0$, which will be fixed throughout the proof.
Denote 
\begin{equation}\label{eq:def of Dm}
D_m=\left\{\left(\sigma_0,\dots,\sigma_{\lfloor\kappa m\rfloor}\right)\neq\left(\sigma^{(m)}_0,\dots,\sigma^{(m)}_{\lfloor\kappa m\rfloor}\right)\right\}.
\end{equation}

We will also need a probability measure $\mathsf{P}_{N,m}^{(0,\a,\b)}$ on $\lrw_N$ for $m\in\llbracket0,N\rrbracket$ defined as follows: first sample the prefix $(\gamma_i)_{i\in\llbracket0,m\rrbracket}$ with its law under $\mathbb P^{(0,\a,\b)}$, and then, conditionally on this prefix, sample the increments $(\gamma_{m+i}-\gamma_m)_{i\in\llbracket0,N-m\rrbracket}$ as an independent $\nu_{\b}$-random walk started from $0$.

For clarity, we split the rest of the proof into three steps. In the first step, we conclude the proof of the result assuming the exponential bounds of the probability of event $D_m$ and that of the total variation distance between $\mathbb{P}^{(0,\a,\b)}$ and $\mathsf{P}_{N,m}^{(0,\a,\b)}$. These two bounds will then be proven in Steps 2 and 3, respectively. 

In the rest of the proof, we will use the notation $\mathbb{P}^{0}=\mathbb{P}^{(0,\a,\b)}$ and $\mathsf{P}_{N,m}^0=\mathsf{P}_{N,m}^{(0,\a,\b)}$.
\\ 

\noindent\textbf{Step 1.} 
In this step, we conclude the proof of the result assuming the following claim:
there exist constants $C_1,c_1,C_2,c_2>0$ depending only on $\a$ and $\b$, such that for each $N\geq m\geq1$, we have
\begin{equation}\label{eq:kty7uityi}
\mathbb{P}^0(D_m)\leq C_1e^{-c_1m}\quad\mbox{ and }\quad \TV{\mathbb{P}^0-\mathsf{P}_{N,m}^0}\leq C_2e^{-c_2m}.
\end{equation}
Indeed, let $B\subset\mathbb{Z}_{\geq0}^{\lfloor\kappa m\rfloor+1}$ be an arbitrary subset, and consider the events 
\[
E=\left\{\left(\sigma_0,\dots,\sigma_{\lfloor\kappa m\rfloor}\right)\in B\right\}\quad\mbox{ and }\quad
E^{(m)}=\left\{\left(\sigma_0^{(m)},\dots,\sigma^{(m)}_{\lfloor\kappa m\rfloor}\right)\in B\right\}.
\]
In view of \eqref{eq:def of Dm}, we have $E\setminus D_m=E^{(m)}\setminus D_m$, and hence
\begin{equation}\label{eq:sdfafawef}
\left|\mathbb{P}^0(E|A)-\mathbb{P}^0(E)\right|\leq \left|\mathbb{P}^0\left(E^{(m)}\big{|}A\right)-\mathbb{P}^0\left(E^{(m)}\right)\right|+\mathbb{P}^0(D_m|A)+\mathbb{P}^0(D_m).
\end{equation}
Combining now \eqref{eq:gergrgegr} and \eqref{eq:kty7uityi}, we have
\begin{equation}\label{eq:sdfafawef2}
    \mathbb{P}^0(D_m|A)+\mathbb{P}^0(D_m)\leq 2\mathbb{P}^0(D_m)/\mathbb{P}^0(A)\leq 2  C_1e^{-c_1m}e^{\delta m}.
\end{equation}
By the definition of the probability measure $\mathsf{P}_{N,m}^0$ and the fact that
\[
E^{(m)}\in\sigma\left((\gamma_i)_{i\in\llbracket1,m\rrbracket}\right),\quad A\in\sigma\left((\gamma_i-\gamma_{i-1})_{i\in\llbracket m+1,N\rrbracket}\right),
\]
we know that the events $E^{(m)}$ and $A$ are independent under $\mathsf{P}_{N,m}^0$, and that $\mathsf{P}_{N,m}^0(E^{(m)})=\mathbb{P}^0(E^{(m)})$. 
Hence $\mathsf{P}_{N,m}^0(E^{(m)}\cap A)=\mathbb{P}^0(E^{(m)})\mathsf{P}_{N,m}^0(A)$, and consequently
\begin{equation}\label{eq:sdfafawef3}
    \begin{split}
        &\left|\mathbb{P}^0\left(E^{(m)}\big{|}A\right)-\mathbb{P}^0\left(E^{(m)}\right)\right| 
         =\left|\frac{\mathbb{P}^0\left(E^{(m)}\cap A\right)}{\mathbb{P}^0(A)}-\mathbb{P}^0\left(E^{(m)}\right)\right|\\
        &\leq \frac{\left| \mathbb{P}^0\left(E^{(m)}\cap A\right) - \mathsf{P}_{N,m}^0\left(E^{(m)}\cap A\right) \right|}{\mathbb{P}^0(A)}
        +\mathbb{P}^0\left(E^{(m)}\right)\frac{\left|\mathsf{P}_{N,m}^0(A)-\mathbb{P}^0(A)\right|}{\mathbb{P}^0(A)}\\
        &\leq 2 \TV{\mathbb{P}^0-\mathsf{P}_{N,m}^0}\cdot \mathbb{P}^0(A)^{-1}\leq 2C_2e^{-c_2m} e^{\delta m},
    \end{split}
\end{equation}
where the last step uses again \eqref{eq:gergrgegr} and \eqref{eq:kty7uityi}.

Combining \eqref{eq:sdfafawef}, \eqref{eq:sdfafawef2}, and \eqref{eq:sdfafawef3}, we conclude that 
$\left|\mathbb{P}^0(E|A)-\mathbb{P}^0(E)\right|\leq2  C_1e^{-c_1m}e^{\delta m}+2C_2e^{-c_2m} e^{\delta m}$. Taking the maximum over all $B\subset\mathbb{Z}_{\geq0}^{\lfloor\kappa m\rfloor+1}$ and let $C=2C_1+2C_2$, $c=\delta=\min(c_1,c_2)/2$, we conclude  \eqref{eq:jnyt5r}.
\\ 

\noindent\textbf{Step 2.} In this step, we prove the first claim, i.e.,  we show that there exist constants $C_1,c_1>0$ depending only on $\a$ and $\b$, such that 
\begin{equation}\label{eq:kty7uityit34t}
\mathbb{P}^0(D_m)\leq C_1e^{-c_1m}\quad\mbox{ for each }N\geq m\geq1.
\end{equation}
In view of \eqref{eq:rhtrbpoi} and \eqref{eq:def of Dm}, we have
\begin{equation}\label{eq:mfyumu}
D_m=\left\{\sigma_i\neq\sigma_i^{(m)}\mbox{ for some }i\in\llbracket0,\lfloor\kappa m\rfloor\rrbracket\right\}
\subset\left\{\gamma_j\leq \min_{k\in\llbracket0,N\rrbracket}\gamma_k+\lfloor\kappa m\rfloor\mbox{ for some } j\in\llbracket m+1,N\rrbracket\right\}.
\end{equation}
Intuitively, we bound the right-hand side using the above random walk decomposition. More precisely, 
recall the observable $T_N$ defined in Lemma \ref{lem:decomposition TASEP}.
By Lemma~\ref{lem:TN exponential},  
\begin{equation}\label{eq:tnjydy}
\mathbb{P}^0(T_N> \lfloor m/2\rfloor)\leq C_3e^{-c_3m},
\end{equation}   
where $C_3,c_3>0$ depend only on $\a$ and $\b$.
Conditioned on $T_N=x$, by Lemma \ref{lem:decomposition TASEP}, $(\gamma_{x+i}-\gamma_{x})_{i\in\llbracket 0,N-x\rrbracket}$ is a $\nu_{\b}$-random walk started at $0$ and conditioned to stay non-negative up to step $N-x$. By Lemma \ref{lem:basic hitting monotone coupling},
\begin{equation}\label{eq:tmum,kuio}
    \mathbb{P}^0\left(\gamma_{j}-\gamma_{x}\leq (j-x)\cdot 3/4\cdot(\b-1)/(\b+1)\,|\, T_N=x\right)\leq C_4e^{-c_4(j-x)}\quad\mbox{ for each }j\in\llbracket x+1,N\rrbracket,
\end{equation} 
where $C_4,c_4>0$ depend only on $\b$.  For $x\in\llbracket0,\lfloor m/2\rfloor\rrbracket$ and
$j\in\llbracket m,N\rrbracket$, our choice $\kappa=(\b-1)/(3(\b+1))$ gives that the event
$\{\gamma_j-\gamma_x\leq \lfloor\kappa m\rfloor\}$ is contained in the
event appearing on the left side of \eqref{eq:tmum,kuio}. Hence
\begin{equation}\label{eq:ykuk}
    \mathbb{P}^0\left(\gamma_{j}-\gamma_{T_N}\leq \lfloor\kappa m\rfloor\,|\,T_N\leq \lfloor m/2\rfloor\right)\leq C_4e^{-c_4(j-\lfloor m/2\rfloor)}\quad\mbox{ for each }j\in\llbracket m,N\rrbracket.
\end{equation}
Summing \eqref{eq:ykuk} over $j\in\llbracket m+1,N\rrbracket$ and combining with \eqref{eq:tnjydy}, in view of \eqref{eq:mfyumu}, we conclude \eqref{eq:kty7uityit34t}.  \\

\noindent\textbf{Step 3.} In this step, we show the second part of the claim, i.e.~that there exist constants $C_2,c_2>0$ depending only on $\a$ and $\b$, such that 
\begin{equation}\label{eq:kty7uityit34t2}
     \TV{\mathbb{P}^0-\mathsf{P}_{N,m}^0}\leq C_2e^{-c_2m}\quad\mbox{ for each }N\geq m\geq1.
\end{equation}
When $N=m$, we have $\mathsf{P}_{N,N}^0=\mathbb{P}^0$, and \eqref{eq:kty7uityit34t2} holds. We next assume $N>m$. By the definition of $\mathsf P_{N,m}^0$, 
\begin{equation}\label{eq:disintegration full TV prefix tail}
\begin{split}
   & \TV{\mathbb P^0-\mathsf P_{N,m}^0}
\le
\mathbb E^0\left[
\TV{
\mathbb P^0\left(\Delta_{m,N}\in\cdot\,\middle|\,\mathcal F_m\right)
-\nu_\b^{\otimes(N-m)}
}
\right]\\
&\leq \mathbb P^0\left(T_N>\lfloor m/2\rfloor\right)
+
\mathbb E^0\left[
{\bf 1}_{\{T_N\le \lfloor m/2\rfloor\}}
\TV{
\mathbb P^0\left(\Delta_{m,N}\in\cdot\,\middle|\,\mathcal F_m,T_N\right)
-\nu_\b^{\otimes(N-m)}
}
\right],
\end{split}    
\end{equation}
where we write $\mathcal F_m=\sigma\left((\gamma_i)_{i\in\llbracket0,m\rrbracket}\right)$ and $\Delta_{m,N}=(\gamma_j-\gamma_{j-1})_{j\in\llbracket m+1,N\rrbracket}$. Conditionally on $T_N=x$, Lemma~\ref{lem:decomposition TASEP} identifies
$(\gamma_{x+i}-\gamma_x)_{i\in\llbracket0,N-x\rrbracket}$ with a
$\nu_\b$-random walk conditioned to stay non-negative up to time $N-x$.
Therefore, denoting $h=\gamma_m-\gamma_x$ and $E_h=\left\{S_i\ge -h\mbox{ for }i\in\llbracket0,N-m\rrbracket\right\}$, we have
\begin{equation}
    \mathbb P^0\left(\Delta_{m,N}\in\cdot\,\middle|\,\mathcal F_m,T_N=x\right)
=
\nu_\b^{\otimes(N-m)}\left(\,\cdot\,\middle|\,E_h\right).
\end{equation}
By Lemma~\ref{lem:basic hitting monotone coupling}, we have $\nu_\b^{\otimes(N-m)}(E_h)\ge 1-\b^{-2h-2}$, hence 
\begin{equation}\label{eq:conditioned tail tv by height}
\TV{
\nu_\b^{\otimes(N-m)}\left(\,\cdot\,\middle|\,E_h\right)
-\nu_\b^{\otimes(N-m)}
}
\le
\ \b^{-2h-2}/(1-\b^{-2h-2})
\le  \b^{-2h}/(\b^2-1).
\end{equation}
By \eqref{eq:ykuk} with $j=m$, and since
$\gamma_m-\gamma_{T_N}\geq0$ on $\{T_N\le\lfloor m/2\rfloor\}$,
\begin{equation}\label{eq:height negative moment exponential} 
 \mathbb E^0\left[
{\bf 1}_{\{T_N\le\lfloor m/2\rfloor\}}
\b^{-2 (\gamma_m-\gamma_{T_N} )}
\right] 
 \le
\mathbb P^0\left(
\gamma_m-\gamma_{T_N}\leq \lfloor\kappa m\rfloor,\,
T_N\le\Big\lfloor \frac{m}{2}\Big\rfloor
\right)
+\b^{-2\lfloor\kappa m\rfloor} 
 \le C_4e^{-c_4\left(m-\lfloor \frac{m}{2}\rfloor\right)}
+\b^{-2\lfloor\kappa m\rfloor} . 
\end{equation} 
Combining
\eqref{eq:disintegration full TV prefix tail},  \eqref{eq:conditioned tail tv by height},
\eqref{eq:height negative moment exponential}, and \eqref{eq:tnjydy}, we
conclude \eqref{eq:kty7uityit34t2}. 
\end{proof}

\subsection{Stability of the tail of the birth-death chain}

As a next step towards Lemma \ref{lem:conditioning does not change much endpoint}, we control the tails of the birth-death chain at location $N$ when revealing the information about the trajectory only up to location $N-m$.

\begin{definition}\label{def:chain tail}
Let $q\in(0,1)$, $N\in\mathbb{N}$, and $\boldsymbol{\mathsf{S}}=(\mathsf{S}(0),\dots,\mathsf{S}(N))\in\{0,1\}^{N+1}$ be a $0$-$1$ sequence. 
For $m\in\llbracket0,N-1\rrbracket$, denote $\boldsymbol{\mathsf{S}}_{N,m}=(\mathsf{S}(N-m),\dots,\mathsf{S}(N))\in\{0,1\}^{m+1}$. 
Assume $x\in\llbracket0, N-m-1 \rrbracket$. Consider the Markov chain $\{X^{\boldsymbol{\mathsf{S}},x}_{N,m}(i)\}_{i=N-m}^{N+1}$ defined by 
$X^{\boldsymbol{\mathsf{S}},x}_{N,m}(N-m)=x$ and for $i\in\llbracket N-m,N\rrbracket$ and $j\in\llbracket0,i-1\rrbracket$,
\begin{enumerate}
    \item [(\romannum{1})] if $\mathsf{S}(i)=1$, we let $X^{\boldsymbol{\mathsf{S}},x}_{N,m}(i+1)-X^{\boldsymbol{\mathsf{S}},x}_{N,m}(i)\in\{-1,0\}$ with 
    \begin{equation}\label{eq:etget11}
    \mathbb{P}\left(X^{\boldsymbol{\mathsf{S}},x}_{N,m}(i+1)=j-1 \Big{|}X^{\boldsymbol{\mathsf{S}},x}_{N,m}(i)=j\right)=\frac{1-q^{j}}{1-q^{i+1}},
    \end{equation}
    \item [(\romannum{2})] if $\mathsf{S}(i)=0$, we let $X^{\boldsymbol{\mathsf{S}},x}_{N,m}(i+1)-X^{\boldsymbol{\mathsf{S}},x}_{N,m}(i)\in\{0,1\}$ with 
    \begin{equation}\label{eq:etget21}
    \mathbb{P}\left(X^{\boldsymbol{\mathsf{S}},x}_{N,m}(i+1)=j+1 \Big{|}X^{\boldsymbol{\mathsf{S}},x}_{N,m}(i)=j\right)=\frac{q^{j+1}-q^{i+1}}{1-q^{i+1}}.
    \end{equation}
\end{enumerate} 
The right sides of \eqref{eq:etget11} and \eqref{eq:etget21} lie in $[0,1]$ since $j\in\llbracket0,i-1\rrbracket$ and $q\in[0,1)$. 
By an induction argument, we have $X^{\boldsymbol{\mathsf{S}},x}_{N,m}(i)\in\llbracket x,x+i-1\rrbracket$  for all $i\in\llbracket N-m,N+1\rrbracket$, and hence this Markov chain is well-defined.  We denote by $K_{N,m}(\boldsymbol{\mathsf{S}};x,\cdot)$ the law of $X^{\boldsymbol{\mathsf{S}},x}_{N,m}(N+1)$ started at $X^{\boldsymbol{\mathsf{S}},x}_{N,m}(N-m)=x$.
\end{definition}

We have the following control on the above birth-chain chain.

\begin{proposition}\label{lem:expectation of TV distance}
Assume \(q\in(0,1)\), \(\b>\max(\a,1)\), and \(\a\b>1\).
For \(\bg\in\lrw_N\), let \(\bs(\bg)\) be the composition associated to
\(\bg\) as in \eqref{eq:Signature}, let
\(\boldsymbol{\mathsf S}^{\bs(\bg)}\in\{0,1\}^{N+1}\) be the associated
\(0\)-\(1\) sequence as in Definition~\ref{def:associated 01 seq}, and let
\(K_{N,m}(\boldsymbol{\mathsf S}^{\bs(\bg)};x,\cdot)\) be the transition law
from Definition \ref{def:chain tail}. Then there exist constants
\(\varepsilon,C,c>0\), depending only on \(q,\a,\b\), such that, for every
\(N\in\mathbb N\) and every \(m\in\llbracket0,N-1\rrbracket\),
\[
\mathbb E^{(0,\a,\b)}\left[
\sup_{x\in\llbracket0, \min (\lfloor\varepsilon m\rfloor, N-m-1 )\rrbracket}
\TV{
K_{N,m}\left(\boldsymbol{\mathsf{S}}^{\bs(\bg)};x,\cdot\right)
-
K_{N,m}\left(\boldsymbol{\mathsf{S}}^{\bs(\bg)};0,\cdot\right)
}
\right]
\leq Ce^{-cm}.
\]
\end{proposition}

As a first step for the proof of Proposition~\ref{lem:expectation of TV distance}, we relate $K_{N,m}\left(\boldsymbol{\mathsf{S}}^{\bs(\bg)};x,\cdot\right)$ and $K_{N,m}\left(\boldsymbol{\mathsf{S}}^{\bs(\bg)};0,\cdot\right)$ to the chain $\mathscr{X}^{\boldsymbol{\mathsf{S}}_{N,m},x}$ from Definition \ref{def:simpler chain}.

\begin{lemma}\label{lem:bound the TV distance by hitting}
Let $q\in(0,1)$, $N\in\mathbb{N}$, and $m\in\llbracket0,N-1\rrbracket$. Let $\boldsymbol{\mathsf{S}}\in\{0,1\}^{N+1}$ be a $0$-$1$ sequence. Let  $\boldsymbol{\mathsf{S}}_{N,m}\in\{0,1\}^{m+1}$ and $K_{N,m}(\boldsymbol{\mathsf{S}};x,\cdot)$ for $x\in\llbracket0, N-m-1 \rrbracket$ be as in Definition \ref{def:chain tail}. Let the Markov chain $\mathscr{X}^{\boldsymbol{\mathsf{S}}_{N,m},x}=\{\mathscr{X}^{\boldsymbol{\mathsf{S}}_{N,m},x}(i)\}_{i=0}^{m+1}$ be as in Definition \ref{def:simpler chain}. Then for each $h\in\llbracket0,N-m-1\rrbracket$,
\begin{equation*}
    \sup_{x\in\llbracket0, h\rrbracket}
\TV{
K_{N,m}\left(\boldsymbol{\mathsf{S}} ;x,\cdot\right)
-
K_{N,m}\left(\boldsymbol{\mathsf{S}} ;0,\cdot\right)
} 
\leq
\mathbb{P}\left( \mathscr{X}^{\boldsymbol{\mathsf{S}}_{N,m},h}(i)>0\mbox{ for all }i\in\llbracket0,m+1\rrbracket \right).
\end{equation*} 
\end{lemma} 
\begin{proof}
Fix any \(x\in\llbracket0,h\rrbracket\). We first prove the following claim: there exists a coupling between three Markov chains given by 
$\{Y^0(i)\}_{i=0}^{m+1}\overset{d}{=}\{X_{N,m}^{\boldsymbol{\mathsf{S}},0}(N-m+i)\}_{i=0}^{m+1}$,
$\{Y^x(i)\}_{i=0}^{m+1}\overset{d}{=}\{X_{N,m}^{\boldsymbol{\mathsf{S}},x}(N-m+i)\}_{i=0}^{m+1}$,
and
$\{Z(i)\}_{i=0}^{m+1}\overset{d}{=}\{\mathscr{X}^{\boldsymbol{\mathsf{S}}_{N,m},h}(i)\}_{i=0}^{m+1}$, 
satisfying 
\begin{equation}\label{eq:ntejnuyjtemu}
Y^0(i)\leq Y^x(i)\leq Z(i)
\quad\mbox{for all } i\in\llbracket0,m+1\rrbracket .
\end{equation}
Indeed, the construction is very similar to Lemma \ref{lem:coupling birth death}. Assume $s\in\{0,1\}$ and $u\in(0,1)$. As in \eqref{eq:phipsi}, we let
    \begin{equation*} 
        \begin{split}
            \Phi_i^s(j,u)=j-{\bf1}_{s=1}{\bf1}_{\left\{u\leq \frac{1-q^j}{1-q^{i+1}}\right\}}+{\bf1}_{s=0}{\bf1}_{\left\{u\leq \frac{q^{j+1}-q^{i+1}}{1-q^{i+1}}\right\}},\quad &\mbox{ for }i\in\llbracket N-m,N\rrbracket\mbox{ and }j\in\llbracket0,i-1\rrbracket,\\
            \Psi^s(j,u)=j-{\bf1}_{s=1}{\bf1}_{\{j\geq1,u\leq 1-q\}}+{\bf1}_{s=0}{\bf1}_{\{u\leq q^{j+1}\}},\quad &\mbox{ for }j\in\mathbb{Z}_{\geq0}.
        \end{split}
    \end{equation*}
As in \eqref{eq:betbretgsertb1} and \eqref{eq:betbretgsertb2}, for all \(i\in\llbracket N-m,N\rrbracket\), \(u\in(0,1)\), and \(s\in\{0,1\}\),  
\begin{equation} \label{eq:hnbytrn1} 
\begin{aligned} 
   \Phi_i^s(j,u)\leq \Phi_i^s(k,u)\quad&\mbox{ for }0\leq j\leq k\leq i-1, \\
    \Phi_i^s(j,u)\leq \Psi^s(k,u)\quad&\mbox{ for } j\in\llbracket0,i-1\rrbracket\mbox{ and }k\in\mathbb{Z}_{\geq0}\mbox{ satisfying }j\leq k.
\end{aligned}\end{equation}
Let $U_0,\dots,U_m$ be i.i.d. $\mathrm{Unif}(0,1)$ random variables. We define the processes $Y^0$, $Y^x$, and $Z$ by 
$Y^0(0)=0$, $Y^x(0)=x$, and $Z(0)=h$, and for each $i\in\llbracket0,m\rrbracket$, we let
\[
Y^0(i+1)=\Phi_{N-m+i}^{\mathsf{S}(N-m+i)}\left(Y^0(i),U_i\right),\quad
Y^x(i+1)=\Phi_{N-m+i}^{\mathsf{S}(N-m+i)}\left(Y^x(i),U_i\right),\quad
Z(i+1)=\Psi^{\mathsf{S}(N-m+i)}\left(Z(i),U_i\right).
\]
Then the three equalities in distribution in the claim follow from the above construction, and the monotonicity \eqref{eq:ntejnuyjtemu} follows from an induction through the deterministic inequalities \eqref{eq:hnbytrn1}.

We return to the proof of the lemma. If \(Z(i_0)=0\) for some \(i_0\in\llbracket0,m+1\rrbracket\), then \eqref{eq:ntejnuyjtemu} and non-negativity imply \(Y^0(i_0)=Y^x(i_0)=0\). From time \(i_0\) onward, the processes \(Y^0\) and \(Y^x\) are updated by the same maps \(\Phi\) with the same uniforms, so equality at time \(i_0\) forces equality at all later times. Hence
\[
\left\{Y^x(m+1)\neq Y^0(m+1)\right\}
\subset
\left\{Z(i)>0\mbox{ for all }i\in\llbracket0,m+1\rrbracket\right\}.
\]
For every \(A\subset\mathbb Z_{\geq0}\),
\[
\begin{split}
&\left|
K_{N,m}\left(\boldsymbol{\mathsf{S}};x,A\right)
-
K_{N,m}\left(\boldsymbol{\mathsf{S}};0,A\right)
\right|  
  =
\left|
\mathbb{P}\left(Y^x(m+1)\in A\right)
-
\mathbb{P}\left(Y^0(m+1)\in A\right)
\right|\\
&\leq
\mathbb{P}\left(Y^x(m+1)\neq Y^0(m+1)\right) 
  \leq
\mathbb{P}\left(
Z(i)>0\mbox{ for }i\in\llbracket0,m+1\rrbracket
\right) =
\mathbb{P}\left(
\mathscr{X}^{\boldsymbol{\mathsf{S}}_{N,m},h}(i)>0
\mbox{ for }i\in\llbracket0,m+1\rrbracket
\right).
\end{split}
\]
Taking the supremum over \(A\subset\mathbb Z_{\geq0}\) and then over
\(x\in\llbracket0,h\rrbracket\) concludes the proof.
\end{proof}

\begin{lemma}\label{lem:simpler chain hitting 0}
 Let \(q\in(0,1)\), $C_0,c_0>0$ and \(H\) be a \(\mathbb N\)-valued random
variable, satisfying $\mathbb{P}(H\geq y)\leq C_0e^{-c_0y}$ for all $y\in\mathbb{R}_{\geq0}$.
Let \(H_1,H_2\ldots\) be i.i.d. copies of \(H\). For $\ell\in\mathbb{N}$, consider the $0$-$1$ sequence $\boldsymbol{\mathsf S}^{H,\ell}\in\{0,1\}^{H_1+\dots+H_{\ell}}$ given by 
\[
{\mathsf S}^{H,\ell}(i)=1\quad\mbox{if and only if } i=H_1+\dots+H_j-1\mbox{ for some }j\in\llbracket1,\ell\rrbracket,
\]
Consider the Markov chain $\mathscr{X}^{\boldsymbol{\mathsf S}^{H,\ell},x}=\{\mathscr{X}^{\boldsymbol{\mathsf S}^{H,\ell},x}(i)\}_{i=0}^{H_1+\dots+H_{\ell}}$ started at $x\in\mathbb{Z}_{\geq0}$ as in Definition \ref{def:simpler chain}, sampled using transition randomness independent of  \(H_1,H_2\ldots\). Then there exist constants $C,c>0$ and $v>1$
depending only on \(q\) and on the law of \(H\), such that
\begin{equation}\label{eq:nryrbsbtn4y}
\mathbb{P}\left(\mathscr{X}^{\boldsymbol{\mathsf S}^{H,\ell},x}(i)>0\mbox{ for all }i\in\llbracket0,H_1+\dots+H_{\ell}\rrbracket\right)
\leq C e^{-c\ell}v^x\quad \mbox{for all }\ell\in\mathbb N,\ x\in\mathbb Z_{\ge0}.
\end{equation}
\end{lemma}

\begin{proof}
    We work on the probability space with variables \(H_1,H_2,\ldots\), an array of i.i.d. \(\mathrm{Unif}(0,1)\) variables $U_{r,s}$ for $r,s\in\mathbb N$, and i.i.d. variables $B_n\sim\ber(1-q)$ for $n\in\mathbb N$, all independent of each other. In the \(n\)-th block, the variables \(U_{n,1},\ldots,U_{n,H_n-1}\) are used for the \(H_n-1\) zero-updates, and \(B_n\) is used for the terminal \(1\)-update. A zero-update from state \(y\in\mathbb Z_{\geq0}\) sends \(y\) to \(y+1\) if and only if the
corresponding uniform variable is at most \(q^{y+1}\), while the terminal \(1\)-update sends \(y\) to \(y-B_n{\bf1}_{\{y>0\}}\).  
Let 
\[
Y_0=x\quad\mbox{and}\quad Y_n=\mathscr{X}^{\boldsymbol{\mathsf S}^{H,n},x}(H_1+\cdots+H_n)\quad\mbox{for }n\in\mathbb N.
\]
Note that $Y_n$ is the state immediately after the terminal \(1\)-update of the \(n\)-th block. 
Conditionally on the block lengths \(H_1,H_2,\ldots\), we have the transition rule
\[
Y_{n+1}=Y_n+Q_n^{Y_n}(H_{n+1}-1)-B_{n+1}{\bf1}_{\left\{Y_n+Q_n^{Y_n}(H_{n+1}-1)>0\right\}}
\quad\mbox{for } n\in\mathbb Z_{\geq0},
\]
where \(Q_n^y\) for $n\in\mathbb Z_{\geq0}$ are independent copies of the chain from Lemma \ref{lem:upward jumps U chain}.

Let $\phi=(1-q)\mathbb E [(1-q^2)^{H-1} ]$, then $\phi\in(0, 1-q]$. Choose $v>1/\phi$.  
Let \(C_1\in\mathbb N\) and \(\rho\in(0,1)\) be the constants from Lemma \ref{lem:block moment bound simpler chain arbitrary start} with \(z=v\) and \(p=1-q\). Let
\[
        C_2=C_1\left(1+\mathbb E\left[(H-1)^{C_1}\right]\frac{\rho}{1-\rho}\right),
\]
which is finite since $H$ is subexponential. Let $\mathcal F_m=\sigma (H_n,B_n,U_{n,s}:n\in\llbracket1,m\rrbracket,\ s\in\mathbb N )$
be the sigma field generated by the block lengths and transition randomness in the first \(m\) blocks. 
We claim that 
\begin{equation}\label{eq:simpler-hit-cond-moment-noatom}
\mathbb E\left[v^{Y_{m+n}}{\bf1}_A\right]
\leq C_2\left( \mathbb P(A)+\rho^n\mathbb E\left[v^{Y_m}{\bf1}_A\right] \right)\quad\mbox{for all }m,n\in\mathbb Z_{\geq0},\,A\in\mathcal F_m.
\end{equation}
Indeed, Lemma \ref{lem:block moment bound simpler chain arbitrary start} gives
\[
\mathbb E\left[v^{Y_{m+n}}\,\middle|\,\mathcal F_m,H_{m+1},\ldots,H_{m+n}\right]
\leq C_1\left(1+\rho^n v^{Y_m} + \sum_{i=0}^{n-1}\rho^{n-i}(H_{m+i+1}-1)^{C_1}\right).
\]
Averaging over $H_{m+1},\ldots,H_{m+n}$, which are independent of \(\mathcal F_m\), gives
\[
\mathbb E\left[v^{Y_{m+n}}\,\middle|\,\mathcal F_m\right]
\leq C_1\left( 1+\rho^n v^{Y_m}+\mathbb E\left[(H-1)^{C_1}\right]\sum_{i=0}^{n-1}\rho^{n-i}\right)\leq C_2\left(1+\rho^n v^{Y_m}\right).
\]
Multiplying by \({\bf1}_A\) and taking expectation proves \eqref{eq:simpler-hit-cond-moment-noatom}.

Next, we choose \(M_0\in\mathbb N\) and \(m_0\in\mathbb N\) such that
\begin{equation}\label{eq:simpler-hit-parameter-choice-noatom}
    C_2v^{-M_0}\leq \phi^{M_0}/4=\delta_0,
    \quad
    C_2\rho^{m_0}v^{-M_0}\leq \delta_0^2/C_2,
    \quad
    C_2\rho^{m_0+M_0}\leq \delta_0/2 .
\end{equation}
Indeed, since \(v\phi>1\), the first inequality holds for all sufficiently large \(M_0\); after \(M_0\) is fixed, the last two inequalities hold for all sufficiently large \(m_0\) because \(\rho\in(0,1)\). Let $T=m_0+M_0$. For \(k\in\mathbb Z_{\geq0}\), let
\[E_k=\bigcap_{n=kT+m_0+1}^{kT+m_0+M_0}\{B_n=1,\, U_{n,s}>q^2\mbox{ for } s\in\llbracket1, H_n-1\rrbracket  \}.
\]
Then \(E_k\) is independent of \(\mathcal F_{kT+m_0}\) and we have $\mathbb P(E_k)=\phi^{M_0}=4\delta_0$. On the event $E_k$, if the $n$-th block starts from a state $y\geq1$, then it sends this state to $y-1$ for $n\in\llbracket kT+m_0+1,kT+m_0+M_0\rrbracket$, since every zero-update in this block has threshold at most \(q^2\), hence leaves the state unchanged, and the terminal \(1\)-update is successful. 
Therefore, denote $\tau_0=\inf\{n\in\mathbb Z_{\geq0}:Y_n=0\}$ with the convention \(\inf\emptyset=\infty\),  we have
\begin{equation}\label{eq:simpler-hit-set-inclusion-noatom}
\{\tau_0>(k+1)T\}
\subseteq
\{\tau_0>kT\}
\cap
\left(
\{Y_{kT+m_0}>M_0\}\cup E_k^{\complement}
\right),
\end{equation}
Indeed, if \(Y_{kT+m_0}=0\), then the chain has already hit \(0\). If \(1\leq Y_{kT+m_0}\leq M_0\) and \(E_k\) occurs, then, as long as the block-starting state is positive, each of the blocks in $\llbracket kT+m_0+1,kT+m_0+M_0\rrbracket$ decreases the state by one. Hence the chain hits \(0\) by time \((k+1)T\).  
Let  
\[
        P_k=\mathbb P(\tau_0>kT) 
        \quad\mbox{and}\quad
        V_k=\mathbb E\left[v^{Y_{kT}}{\bf1}_{\{\tau_0>kT\}}\right].
\]
Then we have
\begin{equation*}
\begin{split}
P_{k+1}
&\leq
\mathbb E\left[
{\bf1}_{\left\{Y_{kT+m_0}>M_0\right\}}
{\bf1}_{\left\{\tau_0>kT\right\}}
\right]
+
\mathbb P\left(E_k^{\complement}\cap\left\{\tau_0>kT\right\}\right)                               \leq
v^{-M_0}
\mathbb E\left[
v^{Y_{kT+m_0}}{\bf1}_{\left\{\tau_0>kT\right\}}
\right]
+
\left(1-\phi^{M_0}\right)P_k                                                              \\
&\leq
\left(1-\phi^{M_0}+C_2v^{-M_0}\right)P_k
+
C_2\rho^{m_0}v^{-M_0}V_k                                                         \leq
(1-3\delta_0)P_k+\delta_0^2C_2^{-1}V_k ,
\end{split}
\end{equation*}
where the first step uses \eqref{eq:simpler-hit-set-inclusion-noatom}; the second step uses Markov's inequality, the fact that \(E_k\) is independent of \(\mathcal F_{kT+m_0}\), and $\mathbb P(E_k)=\phi^{M_0}=4\delta_0$; the third step uses \eqref{eq:simpler-hit-cond-moment-noatom} with \(m=kT\), \(n=m_0\), and \(A=\{\tau_0>kT\}\); and the fourth step uses \eqref{eq:simpler-hit-parameter-choice-noatom}.
We also have
\begin{equation*}
    \begin{split}
V_{k+1}
 =
\mathbb E\left[v^{Y_{(k+1)T}}{\bf1}_{\left\{\tau_0>(k+1)T\right\}}\right]                  \leq
\mathbb E\left[v^{Y_{(k+1)T}}{\bf1}_{\left\{\tau_0>kT\right\}}\right]                      \leq
C_2P_k+C_2\rho^T V_k
\leq
C_2P_k+ \delta_0 V_k /2,
\end{split}
\end{equation*}
where the first step uses the definition of $V_{k+1}$; the second step uses inclusion; the third step uses \eqref{eq:simpler-hit-cond-moment-noatom} with \(m=kT\), \(n=T\), and \(A=\{\tau_0>kT\}\); and the fourth step uses \eqref{eq:simpler-hit-parameter-choice-noatom}. 
Combining the last two displays,
\[ 
P_{k+1}+(\delta_0/C_2) \cdot V_{k+1}\leq 
(1-2\delta_0)P_k+3\delta_0^2 /(2C_2)\cdot V_k       
 \leq(1-\delta_0)(P_k+(\delta_0/C_2) \cdot  V_k),
\] 
where the last inequality follows from \(0<\delta_0=\phi^{M_0}/4\leq1/4\). Iterating,
and using \(P_0\leq1\) and \(V_0\leq v^x\), gives
\[
        P_k
        \leq
        P_k+(\delta_0/C_2) \cdot V_k
        \leq
        (1-\delta_0)^k\left(1+(\delta_0/C_2) \cdot v^x\right)
        \leq
        (1+ \delta_0/C_2 )(1-\delta_0)^k v^x .
\]
Let $c= -\log(1-\delta_0)/T>0$ and $C=(1+\delta_0/C_2)(1-\delta_0)^{-1}>0$, we obtain $\mathbb P(\tau_0>\ell)\leq Ce^{-c\ell}v^x$ for all  $\ell,x\in\mathbb Z_{\geq0}$,
which follows from the last display with \(k=\lfloor \ell/T\rfloor\) and using \(\{\tau_0>\ell\}\subseteq\{\tau_0>kT\}\). 
Note that
\begin{equation*} 
\left\{
\mathscr{X}^{\boldsymbol{\mathsf S}^{H,\ell},x}(i)>0
\mbox{ for all }i\in\llbracket0,H_1+\cdots+H_\ell\rrbracket
\right\}
=
\{\tau_0>\ell\},
\end{equation*}
since zero-updates never decrease the chain.   The constants \(C,c,v\) were constructed from \(q\) and the law of \(H\) only. We conclude the proof of \eqref{eq:nryrbsbtn4y}.
\end{proof}

\begin{proof}[Proof of Proposition \ref{lem:expectation of TV distance}]
If \(m=0\), then \(\min\{\lfloor\varepsilon m\rfloor,N-m-1\}=0\), and the supremum in the statement is over \(x=0\) only; hence the left side is \(0\). In the rest of the proof we assume \(m\ge1\).

In this proof, we will write
$\mathbb{P}^0=\mathbb{P}^{(0,\mathsf{a},\mathsf{b})}$, $\mathbb{E}^0=\mathbb{E}^{(0,\mathsf{a},\mathsf{b})}$,
$\mathrm{P}^0=\mathrm{P}^{(0,\mathsf{a},\mathsf{b})}$, and $\mathrm{E}^0=\mathrm{E}^{(0,\mathsf{a},\mathsf{b})}$.

    Let $\mathrm{P}^0$ be the probability measure on $\lrw_N$ as in Definition \ref{def:approx TASEP measure}. Whenever this probability space is enlarged below, we keep the same notation $\mathrm{P}^0$ and $\mathrm{E}^0$ for the enlarged probability measure and expectation; its marginal on $\lrw_N$ is the measure in Definition \ref{def:approx TASEP measure}.
    By Lemma \ref{lem:regeneration indices MAIN}, on such an enlargement, there exist constants $C_1,c_1>0$ depending only on $\a$ and $\b$; $\mathbb{N}$-valued independent random variables $R_j$ for $j\in\mathbb{N}$ such that $R_j$ for $j\geq2$ are i.i.d.; a random integer $K=\max\{j\in\mathbb Z_{\geq0}:R_1+\cdots+R_j\leq N\}$; and, on the event $\{K\geq1\}$, random variables $1\leq \mathrm{n}_1<\dots<\mathrm{n}_K\leq N$, such that on the event $\{K\geq1\}$,  
\begin{equation}\label{eq:brtsbrtb22}
    \begin{split}
    &\mathsf{S}^{\bs(\bg)}(\mathrm{n}_j)=1\quad\mbox{ for all }j\in\llbracket1,K\rrbracket,\\
    &\mathrm{n}_j-\mathrm{n}_{j-1}=R_{K+2-j}\quad\mbox{ for all }j\in\llbracket2,K\rrbracket,\quad
        \mbox{and}\quad
        N-\mathrm{n}_K=R_1-1.
\end{split}
\end{equation}
Moreover, for all $y\in\mathbb{N}$,
\begin{equation}\label{eq:nrtgrsnrsynj22}
    \begin{split}
            & \mathrm{P}^0(K\geq 1, \mathrm{n}_1\geq y)\leq C_1e^{-c_1y},\quad \mathrm{P}^0(K\geq 1,  N-\mathrm{n}_K\geq y)\leq C_1e^{-c_1y},\\
             & \mathrm{P}^0(K=0)\leq C_1e^{-c_1N},\quad \mathrm{P}^0(R_i\geq y)\leq C_1e^{-c_1y}  \mbox{    for } i\in\mathbb{N},\quad\mbox{ and } \mathrm{P}^0(R_2=1)=\nu_{\b}(1).
\end{split}
\end{equation}
Increasing $C_1$ if necessary, the tail bound for $R_i$ in \eqref{eq:nrtgrsnrsynj22} holds for all $y\in\mathbb R_{\ge0}$.

We let 
\begin{equation*}
    A_m=\max\{j\in\mathbb Z_{\ge0}:R_1+\dots+R_j\le m+1\}\leq K 
\end{equation*}
and define
\begin{equation} \label{def:prm}
    \mathrm p_0=m+1,\qquad
\mathrm p_j=m+1-(R_1+\dots+R_j) \quad\mbox{ for }j\in\llbracket1,m+1\rrbracket.
\end{equation}
By \eqref{eq:brtsbrtb22}, we have
\[
\mathrm p_j=\mathrm n_{K+1-j}-(N-m) \quad\mbox{ for } j\in\llbracket1,A_m\rrbracket.
\] 
Let
$\bar{\boldsymbol{\mathsf S}}  \in\{0,1\}^{m+1}$
be the \(0\)-\(1\) sequence defined by
\begin{equation}\label{eq:mfyujkmjkmgfjfui}
\bar{\mathsf S} (i)=1\quad\mbox{ if and only if }i=\mathrm p_j\mbox{ for some }j\in\llbracket1,A_m\rrbracket,\quad\mbox{ for }
i\in\llbracket0,m\rrbracket,
\end{equation}
with the convention that \(\bar{\boldsymbol{\mathsf S}}\) is the all-zero sequence when \(A_m=0\). 
We split the rest of the proof into two steps.\\ 

\noindent\textbf{Step 1.} In this step, we conclude the proof of the proposition  assuming the following claim: there exist constants $\varepsilon, C_2,c_2>0$ depending only on \(q,\a,\b\), such that
\begin{equation}\label{eq:sufficessrtgbhjnyr1}
    \mathrm{E}^0\left[V\right]\leq C_2e^{-c_2m},
\end{equation}
where for \(m\in\llbracket0,N-1\rrbracket\), we write $h_m=\min\{\lfloor\varepsilon m\rfloor,N-m-1\}$, and set
\[
V
=
\mathbb P \left(
\mathscr X^{\bar{\boldsymbol{\mathsf S}},h_m}(i)>0
\mbox{ for all }i\in\llbracket0,m+1\rrbracket
\right).
\]
Here, $\bar{\boldsymbol{\mathsf S}}  \in\{0,1\}^{m+1}$ is the random $0$-$1$ sequence on the enlarged probability space as in \eqref{eq:mfyujkmjkmgfjfui}, and the probability above is taken only with respect to the transition randomness of the Markov chain $\mathscr X^{\bar{\boldsymbol{\mathsf S}},h_m}$, conditioning on a fixed $\bar{\boldsymbol{\mathsf S}}$.
 
By \eqref{eq:brtsbrtb22}, 
$\mathsf{S}^{\bs(\bg)}(\mathrm{n}_j)=1$ for all $j\in\llbracket1,K\rrbracket$. In view of the construction \eqref{eq:mfyujkmjkmgfjfui} of $\bar{\boldsymbol{\mathsf S}}$, we have

\[
\mathsf S^{\bs(\bg)}_{N,m}(i)=\mathsf S^{\bs(\bg)}(N-m+i)\ge \bar{\mathsf S} (i)
\quad\mbox{for all }i\in\llbracket0,m\rrbracket,
\]
where $\boldsymbol{\mathsf S}^{\bs(\bg)}_{N,m}\in\{0,1\}^{m+1}$ is the tail sequence of $\boldsymbol{\mathsf S}^{\bs(\bg)}$ as in Definition \ref{def:chain tail}. By Lemma \ref{lem:coupling birth death} $(\textup{iii})$, there exists a coupling satisfying
$\mathscr X^{\boldsymbol{\mathsf S}^{\bs(\bg)}_{N,m},h_m}(i)\leq\mathscr X^{\bar{\boldsymbol{\mathsf S}},h_m}(i)$ for all $i\in\llbracket0,m+1\rrbracket$. 
Hence
\begin{equation}\label{eq:sufficessrtgbhjnyr2}
    G(\bg)\leq V\quad \mathrm{P}^0\mbox{-almost surely},
\end{equation}
where
\[
G (\bg)
=
\mathbb P \left(
\mathscr X^{\boldsymbol{\mathsf S}^{\bs(\bg)}_{N,m},h_m}(i)>0
\mbox{ for all }i\in\llbracket0,m+1\rrbracket
\right),
\]
and the probability above is taken only with respect to the transition randomness of $\mathscr X^{\boldsymbol{\mathsf S}^{\bs(\bg)}_{N,m},h_m}$. Define
\[
F (\bg)
=
\sup_{x\in\llbracket0,h_m\rrbracket}
\TV{
K_{N,m}\left(\boldsymbol{\mathsf{S}}^{\bs(\bg)};x,\cdot\right)
-
K_{N,m}\left(\boldsymbol{\mathsf{S}}^{\bs(\bg)};0,\cdot\right)
}.
\]
Using Lemma \ref{lem:bound the TV distance by hitting} for \(\boldsymbol{\mathsf S}=\boldsymbol{\mathsf S}^{\bs(\bg)}\) and \(h=h_m\in\llbracket0,N-m-1\rrbracket\), we obtain
\begin{equation}\label{eq:sufficessrtgbhjnyr3}
    F(\bg)\leq G(\bg) \quad\mbox{ for }\bg\in\lrw_N,
\end{equation}
Since \(F:\lrw_N\to[0,1]\), we have
\begin{equation}\label{eq:sufficessrtgbhjnyr4}
\left|\mathbb E^0[F]-\mathrm E^0[F]\right|
=
\left|
\int_0^1
\left(
\mathbb P^0(F>t)-\mathrm P^0(F>t)
\right)\,dt
\right|
\le
\TV{\mathbb P^0-\mathrm P^0} \leq C_3e^{-c_3 N},
\end{equation}
where $C_3,c_3>0$ are constants depending only on $\a,\b$, and the last step above follows from Lemma \ref{lem:TV distance replaced measure}. Combining \eqref{eq:sufficessrtgbhjnyr1}, \eqref{eq:sufficessrtgbhjnyr2}, \eqref{eq:sufficessrtgbhjnyr3}, \eqref{eq:sufficessrtgbhjnyr4},
\[
\mathbb E^0[F]\leq \mathrm E^0[F]+C_3e^{-c_3 N}\leq \mathrm E^0[G]+C_3e^{-c_3 N}\leq\mathrm E^0[V]+C_3e^{-c_3 N}\leq C_2e^{-c_2m}+C_3e^{-c_3 N}.
\]
Choosing $C=C_2+C_3$ and $c=\min(c_2,c_3)$, we conclude the proof of the lemma.\\

\noindent\textbf{Step 2.} In this step, we prove the claim \eqref{eq:sufficessrtgbhjnyr1}, which we restate as follows: there exist constants $\varepsilon,C_2,c_2>0$ depending only on \(q,\a,\b\), such that
\begin{equation}\label{eq:sufficessrtgbhjnyr167546}
    \mathrm{E}^0\left[V\right]\leq C_2e^{-c_2m},
\end{equation}
where for \(m\in\llbracket0,N-1\rrbracket\), we recall $h_m=\min\{\lfloor\varepsilon m\rfloor,N-m-1\}$, and we let
\begin{equation}
V
=
\mathbb P \left(
\mathscr X^{\bar{\boldsymbol{\mathsf S}},h_m}(i)>0
\mbox{ for all }i\in\llbracket0,m+1\rrbracket
\right),
\end{equation}
where the probability above is taken only with respect to the transition randomness of $\mathscr X^{\bar{\boldsymbol{\mathsf S}},h_m}$. To this end, we will use Lemma \ref{lem:regeneration indices MAIN}, and the ingredients from its proof in the appendix.
 
We begin by choosing $\varepsilon>0$ in the claim. 
Recall that for Lemma \ref{lem:regeneration indices MAIN}, the law of \(R_2\) is the law of the first regeneration-time increment of the positive \(\nu_\b\)-random walk from Lemma \ref{lem:regeneration}; in particular, this law is independent of \(N\) and \(m\), and depends only on \(\b\).
By Lemma \ref{lem:simpler chain hitting 0}, in view of \eqref{eq:nrtgrsnrsynj22}, there exist constants $C_4,c_4>0$ and $v>1$, depending only on $q$ and the law of $R_2$, and hence only on \(q,\a,\b\), such that,  
\begin{equation}\label{eq:simpler block-survival-applied-corrected}
\mathbb P\left(
\mathscr X^{\boldsymbol{\mathsf S}^{H,\ell},x}(i)>0
\mbox{ for all }i\in\llbracket0,H_1+\cdots+H_\ell\rrbracket
\right)
\le C_4e^{-c_4\ell}v^x\quad\mbox{ for all }\ell\in\mathbb N,\, x\in\mathbb Z_{\ge0},
\end{equation}
where \(H_1,H_2,\ldots\) are i.i.d. copies of \(H=R_2\), independent of the transition randomness of the chain. 
By \eqref{eq:nrtgrsnrsynj22}, using Lemma \ref{lem:basic exponential tails}, there exist constants $C_5\in(0,\infty)$ and $c_5\in(0,c_1)$ depending only on $C_1,c_1$,  
\begin{equation}\label{eq:moment C5}
 \mathrm{E}^0\left[e^{c_5R_1}\right]\leq C_5,\quad \mathrm{E}^0\left[e^{c_5R_2}\right]\leq C_5,
\end{equation}
We choose first $\theta\in(0,1)$ and then $\varepsilon,\eta>0$, satisfying 
\begin{equation}\label{eq:thetaepsiloneta}
    \theta\log C_5<c_5/2,\quad (\varepsilon+\eta)\log v<c_4\theta/2.
\end{equation}
Choose \(m_0\in\mathbb N\) such that $\lfloor\theta m\rfloor\geq1$ and $\lfloor\eta m\rfloor\geq1$ when \(m\ge m_0\). The finitely
many \(m<m_0\) will be absorbed into the final constant, and hence it suffices to prove \eqref{eq:sufficessrtgbhjnyr167546} for $m\geq m_0$.
Define the bad event
\[
B_m=
\{A_m\le \lfloor\theta m\rfloor\}
\cup
\left(\{A_m\ge1\}\cap\{\mathrm p_{A_m} >\lfloor\eta m\rfloor\}\right).
\]
with $\mathrm p_{A_m}$ from \eqref{def:prm}. 
The structure of the remainder of the proof is as follows. We will first bound $\mathrm{P}^0(B_m)$ and then bound \(\mathrm E^0[V\cdot {\bf1}_{B_m^{\complement}}]\). The combination of these two estimates will give us \eqref{eq:sufficessrtgbhjnyr167546}.\\

We begin by estimating the probability of the event $B_m$. First, we have
\begin{equation}\label{eq:tjdmnfmyhng1}
    \begin{split}
        & \mathrm P^0\left(A_m\le  \lfloor\theta m\rfloor\right)
        =\mathrm P^0\left( R_1+\dots+  R_{ \lfloor\theta m\rfloor+1}>m+1\right)\\
  &      \le
e^{-c_5(m+1)}
\mathrm E^0\left[e^{c_5 R_1}\right]
\left(\mathrm E^0\left[e^{c_5 R_2}\right]\right)^{\lfloor\theta m\rfloor} 
 \leq e^{- c_5(m+1)}C_5^{ \lfloor\theta m\rfloor+1}
\le
C_5 e^{-\left(c_5-\theta\log C_5\right)m}
\le
C_5 e^{-c_5 m/2},
    \end{split}
\end{equation}
where the first step uses \(\{A_m\le \lfloor\theta m\rfloor\}=\{ R_1+\dots+  R_{ \lfloor\theta m\rfloor+1}>m+1\}\), the second step uses 
Markov's inequality, \(R_i\) for $i\in\mathbb{N}$ are independent and \(R_i\) for $i\geq2$ are i.i.d., the third step uses \eqref{eq:moment C5}, the fourth step uses $c_5>0$ and $\lfloor\theta m\rfloor\leq \theta m$, and the fifth step uses \eqref{eq:thetaepsiloneta}. 
Next,  
\begin{equation}\label{eq:tjdmnfmyhng2}
    \begin{split}
        &\mathrm P^0\left(A_m\ge1,\ \mathrm p_{A_m} >\lfloor\eta m\rfloor\right)\leq  
\sum_{a=1}^{m+1}
\mathrm P^0\left(R_1+\dots+R_a\le m+1,\ R_{a+1}>\lfloor\eta m\rfloor\right)\\
&=
\sum_{a=1}^{m+1}
\mathrm E^0\left[
{\bf1}_{\{R_1+\dots+R_a\le m+1\}}
\mathrm P^0\left(R_{a+1}>\lfloor\eta m\rfloor\mid R_1,\ldots,R_a\right)
\right]\le
(m+1)C_1e^{-c_1\left(\lfloor\eta m\rfloor+1\right)}
\le C_6e^{-c_6m},
    \end{split}
\end{equation}
where the constants \(C_6,c_6>0\) depend only on \(\a,\b,\eta\), hence only on \(q,\a,\b\). 
Here, the first step decomposes according to the value of $A_m=a\in\llbracket1,m+1\rrbracket$, noting that if we also have $\mathrm p_{A_m}>\lfloor\eta m\rfloor$, then $R_{a+1}>\lfloor\eta m\rfloor$; the second step uses the basic property of conditional expectation; the third step uses  \eqref{eq:nrtgrsnrsynj22}. 

Combining \eqref{eq:tjdmnfmyhng1} and \eqref{eq:tjdmnfmyhng2}, 
there exist constants
\(C_7,c_7>0\), depending only on \(q,\a,\b\), such that
\begin{equation}\label{eq:bad-event-exp-small-corrected}
\mathrm P^0(B_m)\le C_7e^{-c_7m}
\quad\mbox{for all }m\ge m_0.
\end{equation} 

Our next goal is to estimate \(\mathrm E^0[V\cdot {\bf1}_{B_m^{\complement}}]\). For deterministic \(u_1,\ldots,u_\ell\in\mathbb N\), let
\(\boldsymbol{\mathsf S}^{u_1,\ldots,u_\ell}\in\{0,1\}^{u_1+\cdots+u_\ell}\) be the $0$-$1$ sequence given by 
\[
{\mathsf S}^{u_1,\ldots,u_\ell}(i)=1\quad\mbox{if and only if } i=u_1+\dots+u_j-1\mbox{ for some }j\in\llbracket1,\ell\rrbracket.
\]
Recall the Markov chain $\mathscr{X}^{\bar{\boldsymbol{\mathsf S}},h_m}$, where $\bar{\boldsymbol{\mathsf S}}  \in\{0,1\}^{m+1}$ is the \(0\)-\(1\) sequence given by \eqref{eq:mfyujkmjkmgfjfui}, i.e., 
\[
\bar{\mathsf S} (i)=1\quad\mbox{ if and only if }i=\mathrm p_j\mbox{ for some }j\in\llbracket1,A_m\rrbracket,\quad\mbox{ for }
i\in\llbracket0,m\rrbracket.
\]
Fix \(a\in\llbracket \lfloor\theta m\rfloor+1,m+1\rrbracket\), and set
\[
E_a=\left\{A_m=a,\ \mathrm p_a\le \lfloor\eta m\rfloor\right\}.
\]
On \(E_a\), since \(\lfloor\theta m\rfloor\ge1\), we have \(a\ge2\). 
For every realization
of the transition randomness of  $\mathscr{X}^{\bar{\boldsymbol{\mathsf S}},h_m}$,
\[
\mathscr{X}^{\bar{\boldsymbol{\mathsf S}},h_m}\left(\mathrm p_a+1\right)\le h_m+\mathrm p_a\le h_m+\lfloor\eta m\rfloor.
\]
The symbols of \(\bar{\boldsymbol{\mathsf S}} \) read from time
\(\mathrm p_a+1\) through time \(\mathrm p_1\) form the complete block sequence
\[
R_a,\ R_{a-1},\ldots,R_2,
\]
because \(\mathrm p_{j-1}-\mathrm p_j=R_j\) for \(j\in \llbracket2,a\rrbracket\). Thus, by the Markov property at time \(\mathrm p_a+1\), and then by Lemma~\ref{lem:coupling birth death}~$(\textup{iii})$ applied to the same sequence and to initial states \(y\le h_m+\lfloor\eta m\rfloor\), with the probability on the right side of the following display taken conditionally on \(R_1,R_2,\ldots\) and only over the transition randomness,
\begin{equation}\label{eq:yumtdyj}
\begin{split}
V\cdot{\bf1}_{E_a}
&\le
{\bf1}_{E_a}\cdot
\mathbb P \left(
\mathscr X^{\boldsymbol{\mathsf S}^{R_a,\ldots,R_2},\,h_m+\lfloor\eta m\rfloor}(i)>0
\mbox{ for all }i\in\llbracket0,R_2+\cdots+R_a\rrbracket
\right).
\end{split}
\end{equation}
Taking \( \mathrm E^0 \), dropping the indicator \( {\bf1}_{E_a}\), and using that \((R_a,\ldots,R_2)\) has the same distribution as \(a-1\) i.i.d. copies of \(H=R_2\), \eqref{eq:simpler block-survival-applied-corrected}
gives
\begin{equation}\label{eq:fixed-a-good-bound-corrected} \mathrm E^0\left[
V\cdot {\bf1}_{E_a}
\right] \le
\mathbb P\left(
\mathscr X^{\boldsymbol{\mathsf S}^{H,a-1},\,h_m+\lfloor\eta m\rfloor}(i)>0
\mbox{ for all }i\in\llbracket0,H_1+\cdots+H_{a-1}\rrbracket
\right)
\le C_4e^{-c_4(a-1)}v^{h_m+\lfloor\eta m\rfloor}. 
\end{equation}
On \(B_m^{\complement}\), since \(\lfloor\theta m\rfloor\ge1\), we have
\(A_m\ge \lfloor\theta m\rfloor+1\) and \(\mathrm p_{A_m}\le \lfloor\eta m\rfloor\). Hence, 
\begin{equation}\label{eq:good-event-survival-bound-corrected} 
\mathrm{E}^0\left[
V {\bf1}_{B_m^\complement}
\right]
 \le
\sum_{a=\lfloor\theta m\rfloor+1}^{m+1}
C_4e^{-c_4(a-1)}v^{h_m+\lfloor\eta m\rfloor} \le
C_8 e^{-c_4\lfloor\theta m\rfloor}v^{(\varepsilon+\eta)m} 
\le
C_9e^{-c_9m}, 
\end{equation}
where \(c_9>0\) follows from \eqref{eq:thetaepsiloneta}. 
Combining \eqref{eq:bad-event-exp-small-corrected} and \eqref{eq:good-event-survival-bound-corrected},
\[ 
\mathrm{E}^0[V]
\le
\mathrm{P}^0(B_m)
+
\mathrm{E}^0\left[V\cdot{\bf1}_{B_m^\complement}\right]
\leq C_7e^{-c_7m}+C_9e^{-c_9m}
\le C_2e^{-c_2m}\quad\mbox{ for all }m\ge m_0,
\]
where $C_2=C_7+C_9$ and $c_2=\min(c_7,c_9)$ are constants depending only on \(q,\a,\b\). Finally, since $V\in[0,1]$, one can increase $C_2$ to cover the finitely many \(m<m_0\). We conclude   \eqref{eq:sufficessrtgbhjnyr167546} for all $m\in\mathbb{N}$.
\end{proof}

\subsection{Control of endpoint under conditioning: Proof of Lemma \ref{lem:conditioning does not change much endpoint}}\label{sec:LemmaSupSup}

We have now all tools in order to provide the deferred proof of Lemma \ref{lem:conditioning does not change much endpoint}. 

\begin{proof}[Proof of Lemma \ref{lem:conditioning does not change much endpoint}]
    In this proof, we will write
$\mathbb{P}^0=\mathbb{P}^{(0,\mathsf{a},\mathsf{b})}$, $\mathbb{E}^0=\mathbb{E}^{(0,\mathsf{a},\mathsf{b})}$,
$\mathrm{P}^0=\mathrm{P}^{(0,\mathsf{a},\mathsf{b})}$, and $\mathrm{E}^0=\mathrm{E}^{(0,\mathsf{a},\mathsf{b})}$.

We start by proving the following claim: for every fixed \(z>1\),
\begin{equation}\label{eq:uniform-intermediate-X-moment new}
\sup_{N\in\mathbb N}\sup_{i\in\llbracket0,N+1\rrbracket}
\mathbb E^0\left[
\mathbb E \left[
z^{X^{\boldsymbol{\mathsf S}^{\bs(\bg)}}(i)}
\right]\right]
<\infty ,
\end{equation}
is a finite constant depending only on \(q,\a,\b,z\),
where the inner expectation is taken only with respect to the transition randomness of the Markov chain, conditioning on the sampled path \(\bg\). 
For $\bg\in\lrw_N$, we write
$D_i(\bg)=
\mathbb E \left[
z^{\mathscr X^{\boldsymbol{\mathsf S}^{\bs(\bg)}}(i)}
\right]$.
By Lemma \ref{lem:apriori bound birth death chain}, there exists a finite constant $C_1$ depending only on \(q,z\), such that
\begin{equation}\label{eq:brbwtwryn,ui}
0\le D_i(\bg)\le C_1 z^{i^{3/4}} 
\quad\mbox{for all }\bg\in\lrw_N
 \mbox{ and } i\in\llbracket0,N+1\rrbracket .
\end{equation}
Therefore,
\begin{equation}\label{eq:brbwtwryne,ui}
\mathbb E^0\left[
\mathbb E \left[
z^{X^{\boldsymbol{\mathsf S}^{\bs(\bg)}}(i)}
\right]\right]\leq 
\mathbb E^0\left[D_i(\bg)\right]
\le
\mathrm E^0\left[D_i(\bg)\right]
+
2 C_1 z^{i^{3/4}}
\TV{\mathbb P^0-\mathrm P^0} 
\le
\mathrm E^0\left[D_i(\bg)\right]
+
2C_1C_2 z^{(N+1)^{3/4}}e^{-c_2 N},
\end{equation}
where the first step uses the coupling given by Lemma~\ref{lem:coupling birth death}~$(\textup{ii})$, the second step uses  
\eqref{eq:brbwtwryn,ui}, the third step uses Lemma \ref{lem:TV distance replaced measure}, which gives the constants $C_2,c_2>0$  depending only on $\a,\b$. By Lemma \ref{lem:moment bound simpler chain},
\begin{equation}\label{eq:brbwtweryne,ui}
\sup_{N\in\mathbb N}\sup_{i\in\llbracket0,N+1\rrbracket}
\mathrm E^0 \left[D_i(\bg)\right]<\infty ,
\end{equation}
is a finite constant depending only on \(q,\a,\b,z\).
Taking the supremum over $i\in\llbracket0,N+1\rrbracket$ and then over $N\in\mathbb N$ in \eqref{eq:brbwtwryne,ui}, combining with \eqref{eq:brbwtweryne,ui}, we conclude the proof of \eqref{eq:uniform-intermediate-X-moment new}.

We return to the proof of the lemma.
Let \(C_3,c_3,\delta_3,\kappa_3>0\) be the constants from Lemma \ref{lem:conditioning does not affect first compositions}, and
\(\varepsilon_4,C_4,c_4>0\) be the constants from Proposition~\ref{lem:expectation of TV distance}. Let
\begin{equation}\label{eq:brtwytnwymi,}
\eta= \min\{\kappa_3,1\}/8,
\quad
\varepsilon= \min\{\varepsilon_4,1\}/2,
\quad 
\delta
=
\min\left\{
\delta_3,\,
 c_4\eta/8,\,
 (\varepsilon\eta\log 2)/16
\right\}.
\end{equation}
Choose \(m_0\in\mathbb N\) depending only on \(\kappa_3\) and \(\varepsilon\), such that for every \(N\geq m\ge m_0\), we have
\begin{equation}\label{eq:ell-admissible-endpoint-conditioning new}
1\le \ell=\lfloor \eta m\rfloor\le N-1,\quad
\ell+1\le \lfloor\kappa_3m\rfloor,\quad
\ell\ge  \eta m/2,\quad
\lfloor\varepsilon\ell\rfloor\ge  \varepsilon\ell/2,\quad
\lfloor\varepsilon\ell\rfloor\le N-\ell-1 .
\end{equation}
 By
\eqref{eq:uniform-intermediate-X-moment new} with \(z=2\),
\begin{equation}\label{eq:Y-ell-tail new}
\mathbb E^0\left[
\mathbb P \left(X^{\boldsymbol{\mathsf S}^{\bs(\bg)}}(N-\ell)>\lfloor\varepsilon\ell\rfloor\,\middle|\,\bg\right)
\right]
\le
M \cdot 2^{-\lfloor\varepsilon\ell\rfloor},
\end{equation}
where $M$ is the left side of \eqref{eq:uniform-intermediate-X-moment new}, which a finite constant depending only on \(q,\a,\b,z\).
Denote
\[
F_\ell(\bg)
=
\sup_{x\in\llbracket0,\lfloor\varepsilon\ell\rfloor\rrbracket}
\TV{
K_{N,\ell}\left(\boldsymbol{\mathsf S}^{\bs(\bg)};x,\cdot\right)
-
K_{N,\ell}\left(\boldsymbol{\mathsf S}^{\bs(\bg)};0,\cdot\right)
}\quad\mbox{ for }\bg\in\lrw_N.
\]
Since \(\lfloor\varepsilon\ell\rfloor\le\min\{\lfloor\varepsilon_4\ell\rfloor,N-\ell-1\}\), by Proposition~\ref{lem:expectation of TV distance} with \(\ell\) in place of \(m\), 
\begin{equation}\label{eq:F-ell-average new}
\mathbb E^0[F_\ell]\le C_4e^{-c_4\ell}.
\end{equation}
Denote the probability distributions
\[
\mu^{\bg}=\mathbb{P}\left(X^{\boldsymbol{\mathsf S}^{\bs(\bg)}}(N+1)\in\,\cdot\,\right)\quad\mbox{and}\quad
\nu^{\bg} =K_{N,\ell}\left(\boldsymbol{\mathsf S}^{\bs(\bg)};0,\cdot\right),\quad\mbox{ for fixed }\bg\in\lrw_N.
\]
By the Markov property at time \(N-\ell\),
\begin{equation}\label{eq:dtmntdm}
\TV{\mu^{\bg} - \nu^{\bg}}
\le
F_\ell(\bg)
+
\mathbb P \left(X^{\boldsymbol{\mathsf S}^{\bs(\bg)}}(N-\ell)>\lfloor\varepsilon\ell\rfloor\,\middle|\,\bg\right)\quad
\mbox{ for each }\bg\in\lrw_N.
\end{equation}
When \(\mathbb P^0(A)\ge e^{-\delta m}\), averaging \eqref{eq:dtmntdm} under \(\mathbb P^0(\cdot\,|A)\) and \(\mathbb P^0\), respectively, combining with \eqref{eq:Y-ell-tail new} and \eqref{eq:F-ell-average new}, 
\begin{align}
\TV{
\mathbb E^0[\mu^{\bg}\,|\,A]
-
\mathbb E^0[\nu^{\bg}\,|\,A]
}
&\le
e^{\delta m}\left(C_4e^{-c_4\ell}+M\cdot 2^{-\lfloor\varepsilon\ell\rfloor}\right),
\label{eq:conditioned-starting-state-error new}
\\
\TV{
\mathbb E^0[\mu^{\bg}]
-
\mathbb E^0[\nu^{\bg}]
}
&\le
C_4e^{-c_4\ell}+M\cdot 2^{-\lfloor\varepsilon\ell\rfloor}.
\label{eq:unconditioned-starting-state-error new}
\end{align}
We next compare 
\(\mathbb E^0[{\nu}^{\bg}\,|\,A]\) with \(\mathbb E^0[{\nu}^{\bg}]\). We claim that the tail sequence
\begin{equation}\label{eq:NiceSequence}
    \boldsymbol{\mathsf S}_{N,\ell}=\left(
\mathsf S^{\bs(\bg)}(N-\ell),\dots,\mathsf S^{\bs(\bg)}(N)
\right)\in\{0,1\}^{\ell+1}
\end{equation}
is a deterministic function of
$\left(\sigma_0,\dots,\sigma_{\lfloor\kappa_3m\rfloor}\right)$.
Indeed, writing
\(\bs(\bg)=(\sigma_0,\dots,\sigma_r)\) and using the convention
\(\sigma_i=0\) for \(i>r\),
by Definition \ref{def:associated 01 seq},  we have
\[
\mathsf S^{\bs(\bg)}(N-h)=1\quad\mbox{ if and only if }\sigma_0+\cdots+\sigma_{j-1}=h+1\mbox{ for some }j\in\llbracket1,r\rrbracket,
\quad\mbox{ for }h\in\llbracket0,\ell\rrbracket.
\]
If $\sigma_0+\cdots+\sigma_{j-1}=h+1$, then \(j\le h+1\le \ell+1\le \lfloor\kappa_3m\rfloor\), since the nonzero composition parts are in $\mathbb{N}$. Thus the truth value of this condition is determined by $\left(\sigma_0,\dots,\sigma_{\lfloor\kappa_3m\rfloor}\right)$. This proves the claim in \eqref{eq:NiceSequence}.
Consequently,  
\[
K_{N,\ell}\left(\boldsymbol{\mathsf S}^{\bs(\bg)};0,B\right)=f_B\left(\sigma_0,\dots,\sigma_{\lfloor\kappa_3m\rfloor}\right)\in[0,1]
\quad\mbox{ for each }B\subset\mathbb Z_{\ge0},
\]
where \(f_B \) is a measurable function depending on $B$. 
By Lemma \ref{lem:conditioning does not affect first compositions} and \(\delta\le\delta_3\),
\begin{equation*} 
    \begin{split}
        &\left|
\mathbb E^0\left[
K_{N,\ell}\left(\boldsymbol{\mathsf S}^{\bs(\bg)};0,B\right)
\,\middle|\,A
\right]
-
\mathbb E^0\left[
K_{N,\ell}\left(\boldsymbol{\mathsf S}^{\bs(\bg)};0,B\right)
\right]
\right|\\
&=\left|
\mathbb E^0\left[f_B\left(\sigma_0,\dots,\sigma_{\lfloor\kappa_3m\rfloor}\right)\mid A\right]-\mathbb E^0\left[f_B\left(\sigma_0,\dots,\sigma_{\lfloor\kappa_3m\rfloor}\right)\right]
\right|\le C_3e^{-c_3m}.
    \end{split}
\end{equation*}
Taking the supremum over \(B\subset\mathbb Z_{\ge0}\), we get
\begin{equation}\label{eq:brintu}
\TV{
\mathbb E^0[\nu^{\bg}\,|\,A]
-
\mathbb E^0[\nu^{\bg}]
}
\le C_3e^{-c_3m}.
\end{equation}
Combining \eqref{eq:brintu}, \eqref{eq:conditioned-starting-state-error new}, \eqref{eq:unconditioned-starting-state-error new}, 
we obtain
\begin{equation}\label{eq:endpoint-conditioning-combined new}
\begin{split}
&\TV{
\mathbb{P}^{0}\left(
X^{\boldsymbol{\mathsf{S}}^{\bs(\bg)}}(N+1)\in\cdot\,\middle|\,A
\right)
-
\mathbb{P}^{0}\left(
X^{\boldsymbol{\mathsf{S}}^{\bs(\bg)}}(N+1)\in\cdot
\right)
}
\le
C_3e^{-c_3m}
+
\left(1+e^{\delta m}\right)
\left(C_4e^{-c_4\ell}+M\cdot 2^{-\lfloor\varepsilon\ell\rfloor}\right).
\end{split}
\end{equation}
For \(m\ge m_0\), \eqref{eq:ell-admissible-endpoint-conditioning new} gives
\[
e^{\delta m}e^{-c_4\ell}
\le
\exp\left(-\left( c_4\eta/2-\delta\right)m\right)
\le
e^{-c_4\eta m/4},
\]
where the last inequality follows from \eqref{eq:brtwytnwymi,}.  Also,
\[
e^{\delta m}\cdot 2^{-\lfloor\varepsilon\ell\rfloor}
\le
\exp\left(-\left((\varepsilon\eta\log 2)/4-\delta\right)m\right)
\le
\exp\left(-(\varepsilon\eta\log 2)/8\cdot m\right).
\]
Hence the right side of
\eqref{eq:endpoint-conditioning-combined new} is at most \(Ce^{-cm}\) for all \(m\ge m_0\), where \(C,c>0\) depend only on \(q,\a,\b\).
Finally, we increase \(C\) if necessary so that \(1\le Ce^{-cm}\) for the finitely many \(1\le m<m_0\). We conclude the proof.
\end{proof}

\section{Proofs of the applications of the approximation result}\label{sec:FurtherApplications}

In this section, we establish several consequences of our approximation results on the stationary distribution of the open ASEP in the shock region. 

\subsection{Moderate deviation estimates: Proof of Proposition~\ref{thm:moderate deviation}}

We start with moderate deviations for the total number of particles under the stationary distribution of the open ASEP in the shock region of the high density and the low density phase, which we state below again for clarity.

\begin{proposition}\label{thm:moderate deviation again}
Consider the open ASEP on the lattice \(\llbracket1,N\rrbracket\) with jump rates \(\alpha,\beta,\gamma,\delta\) and \(q\) satisfying
\eqref{eq:conditions open ASEP}. Assume that Liggett's condition \eqref{eq:Liggett's condition} holds and that the open ASEP is in the shock region of the high and low density phases, i.e.,  \(\a\b>1\) and $\a\neq\b$, where $\a=(1-\alpha)/\alpha$ and $\b= (1-\beta)/\beta$. We denote its stationary measure by \(\mu_N^{(q,\a,\b)}\). Define $\rho= \b/(1+\b)$ in the high density phase $\b>\max(\a,1)$ and  $\rho= 1/(1+\a)$ in the low density phase $\a>\max(\b,1)$. Then there exist constants $C,c\in(0,\infty)$, depending only on \(q,\a,\b\), such that, for every \(N\in\mathbb N\), every nonempty set $I  \subseteq\llbracket1,N\rrbracket$ and every  \(x\ge0\),   
\begin{equation*} 
\mu_N^{(q,\a,\b)}
\left(
\left|
\sum_{i\in I}\tau_i-\rho|I|
\right|
\ge x\sqrt{|I|}
\right)
\le
C e^{-cx^2 } .
\end{equation*} 
\end{proposition}

\begin{lemma}\label{lem:Bernoulli moderate}
Assume $d\in\mathbb N$, $\rho_1,\dots,\rho_d\in[0,1]$,  and $\varsigma_i\sim\ber(\rho_i)$, $i\in\llbracket1,d\rrbracket$ are independent. Then, 
\[ 
\mathbb P\left(
\left| (\varsigma_1-\rho_1)+\dots+(\varsigma_d-\rho_d )   \right|\geq s
\right)
\leq 2 e^{-2s^2/d}    \qquad\mbox{ for each } s\geq0.
\]
\end{lemma}
\begin{proof} This is a standard result but we provide the proof for completeness.
    Let $\rho\in[0,1]$, $\varsigma\sim\ber(\rho)$, and $g(\lambda)=\log\mathbb E [e^{\lambda(\varsigma-\rho)} ]$ for $\lambda\in\mathbb R$. We have $g(0)=g'(0)=0$ and 
    $g''(\lambda)=\rho e^\lambda(1-\rho)(1-\rho+\rho e^\lambda)^{-2}\leq 1/4$. By Taylor's theorem, $g(\lambda)\leq \lambda^2/8$ for all $\lambda\in\mathbb R$.
    By independence and the exponential Markov inequality, for $\lambda\geq0$,
 $\mathbb P ( (\varsigma_1-\rho_1)+\dots+(\varsigma_d-\rho_d ) \geq s )
\leq
e^{-\lambda s+ d\lambda^2/8 }$.
Taking \(\lambda=4s/d\) gives the bound by \(e^{-2s^2/d}\). Applying the same argument to the negative of the sum and taking a union bound proves the result.
\end{proof}

\begin{proof}[Proof of Proposition~\ref{thm:moderate deviation again}]
    Let \(C_0,c_0>0\) be the constants from Theorem \ref{thm:ApproxHighLow}. Then for each $N\geq m\geq 1$,
    \begin{equation}\label{eq:approx in proof}
        \begin{split}
              \left\|
\mu_N^{(q,\a,\b)}
\left((\tau_i)_{i\in\llbracket m,N\rrbracket}\in\cdot\right)
-\ber_{\llbracket m,N\rrbracket}(\rho)
\right\|_{\mathrm{TV}}
 \leq C_0e^{-c_0m}\qquad&\mbox{ if } \b>\max(\a,1),\\
\left\|
\mu_N^{(q,\a,\b)}
\left((\tau_i)_{i\in\llbracket1,m\rrbracket}\in\cdot\right)
-\ber_{\llbracket1,m\rrbracket}(\rho)
\right\|_{\mathrm{TV}}
\leq C_0e^{-c_0(N-m)}\qquad&\mbox{ if }\a>\max(\b,1).
        \end{split}
    \end{equation}
We choose $c=\min (1,c_0/8)$, $C=\max (e^{4c},2+C_0)$ and we will prove the theorem for these constants. Fix a nonempty set $I\subseteq\llbracket1,N\rrbracket$, write $I=\{i_1<\cdots<i_n\}$, and set $S_A=\sum_{i\in A}\tau_i$ for each $A\subseteq\llbracket1,N\rrbracket$. When \(0\leq x<2\), we have 
$\mu_N^{(q,\a,\b)} (|S_I-\rho n|\geq x\sqrt n )
\leq1
\leq Ce^{-cx^2}$. When $x>\sqrt n$, the event $\{|S_I-\rho n|\geq x\sqrt n\}=\emptyset$. We next assume $2\leq x\leq\sqrt n$ and set $h=\lfloor x^2/4\rfloor$. 
Note that $x^2/8\leq h\leq x^2/4\leq n/4$. Define
\[
J=
\begin{cases}
\{i_{h+1},\ldots,i_n\}&\qquad\mbox{ if }\b>\max(\a,1),\\
\{i_1,\ldots,i_{n-h}\}&\qquad\mbox{ if }\a>\max(\b,1).
\end{cases}
\]
  In the high density phase, $\min J=i_{h+1}\geq h+1$; while in the low density phase, $N-\max J=N-i_{n-h}\geq h$. Using \eqref{eq:approx in proof} with $m=i_{h+1}$ in the high density phase and with $m=i_{n-h}$ in the low density phase,  
\[
\left\|
\mu_N^{(q,\a,\b)}
\left((\tau_i)_{i\in J}\in\cdot\right)
-\ber_J(\rho)
\right\|_{\mathrm{TV}}
\leq C_0e^{-c_0h}.
\]
Since $\tau_i\in\{0,1\}$ and $\rho\in[0,1]$, we have $|S_{I\setminus J}-\rho h|\leq h$. Hence
\[
\left\{|S_I-\rho n|\geq x\sqrt n\right\}
\subseteq
\left\{|S_J-\rho(n-h)|\geq x\sqrt n-h\right\}.
\] 
Combining the above two displayed equations and Lemma \ref{lem:Bernoulli moderate}, in view of $|J|=n-h$,
\[ 
\mu_N^{(q,\a,\b)}
\left(|S_I-\rho n|\geq x\sqrt n\right)
\leq
2\exp\left(-\frac{2(x\sqrt n-h)^2}{n-h}\right)
+C_0e^{-c_0h}\leq 2e^{-9x^2/8}+C_0e^{-c_0x^2/8}\leq Ce^{-cx^2},
\]
where we used $2(x\sqrt n-h)^2/(n-h)\geq 9x^2/8$ since $h\leq x^2/4\leq x\sqrt n/4$.
We conclude the proof.
\end{proof}

\subsection{The open ASEP with light particles}

In the following, we provide a proof of Theorem~\ref{thm:mainHighLowLightIntro} and Theorem~\ref{thm:MainMixingTimeIntro} 
for the open ASEP with a finite number of light particles. Recall that we denote the state space of $r$ light particles on the segment $\llbracket 1,N \rrbracket$ by 
\[
\widehat{\Omega}_{N,r}:=\left\{\tau=(\tau_1,\dots,\tau_N)\, \colon \, \tau_i\in\{0, 1, 2\} \mbox{ for } i\in\llbracket 1,N \rrbracket,\quad \sum_{i=1}^{N} \mathds{1}{\{ \tau_i = 2 \} } = r\right\}.
\] 
and we let $\mu_{N,r}$ denote the unique stationary measure of the open ASEP with $r$ light particles, which is a probability measure on $\widehat{\Omega}_{N,r}$. Furthermore, we write 
\[
1\leq\lc_1<\dots<\lc_r\leq N.
\]
for the locations of the $r$ light particles within the lattice. 
Intuitively, the light particles behave as particles with respect to empty sites, and are treated like empty sites by all other particles; see Section~1.2 in~\cite{SY:openLight} for a formal definition. We recall the following result from \cite{SY:openLight} on the stationary distribution of the open ASEP with light particles in the fan region of the high and low density phase. 
\begin{theorem}[Theorem 1.3 in \cite{SY:openLight}]\label{thm:mainHighLowLight} Fix $r\in \mathbb{N}$. 
  Suppose $(a_N)_{N=1,2,\dots}$ and $(b_N)_{N=1,2,\dots}$ are two arbitrary sequences of integers such that  $1\leq a_N<b_N \leq N$ and $a_N,N-b_N\rightarrow\infty$. Then in the fan region of the high density phase where $\a\geq 1$, and $\a\b \leq 1$,  
\begin{equation}
    \lim_{N\rightarrow\infty}\mu_{N,r}(\lc_1,\dots,\lc_r\in \llbracket 1,a_N \rrbracket )=1. 
\end{equation}
Moreover, in the fan region of the low density phase where $C\geq 1$, and $AC \leq 1$,  
 \begin{equation}
          \lim_{N\rightarrow\infty}\mu_{N,r}( \lc_1,\dots,\lc_r\in \llbracket b_N,N\rrbracket )=1.
 \end{equation}
\end{theorem}
Next, recall from \eqref{def:MixingTime} that we quantify the speed of convergence of the open ASEP with light particles to the stationary measure $\widehat{\mu}_{N,r}$ with respect to the total-variation distance between probability measures on $\hat{\Omega}_{n,r}$. More precisely, we consider for all $\varepsilon \in (0,1)$ the mixing time for the open ASEP with $r$ many light particles given by
 \begin{equation}
     t^N_{\textup{mix}}(\varepsilon) := \inf \left\{ t \geq 0 \, \colon \, \max_{\tilde{\tau} \in \widehat{\Omega}_{N,r}} \TV{ \mathbb{P}(\tau(t) \in \, \cdot \, | \, \tau(0)=\tilde{\tau}) - \widehat{\mu}_{N,r} } \leq \varepsilon \right\},
 \end{equation}
where $\tau(t)=\lb\tau_1(t),\dots,\tau_N(t)\rb$ denote the occupation variables for the process at time $t\geq0$. We recall the following result on the mixing time of the open ASEP with light particles in the fan region from \cite{SY:openLight}.
\begin{theorem}\label{thm:MainMixingTime}[Theorem 1.6 in \cite{SY:openLight}]
Then there exist constants $c_0,C_0>0$, depending only on $\a$ and $\b$,  such that for every $\varepsilon\in(0,1)$,
\begin{equation}\label{eq:MixingEquation2}
   c_0 \leq  \liminf_{N \rightarrow \infty} \frac{t^N_{\textup{mix}}(\varepsilon)}{N} \leq   \limsup_{N \rightarrow \infty} \frac{t^N_{\textup{mix}}(\varepsilon)}{N} \leq C_0 . 
\end{equation} 
\end{theorem}

We have the following extension of Theorems~\ref{thm:mainHighLowLight} and \ref{thm:MainMixingTime} to the shock region where $\a > \b \geq 1$ and $\a>\b$, covering Theorems~\ref{thm:mainHighLowLightIntro} and \ref{thm:MainMixingTimeIntro}. 

\begin{proposition}\label{pro:mainHighLowLightShock}
    Theorem~\ref{thm:mainHighLowLight} and Theorem~\ref{thm:MainMixingTime} continue to holds in the shock region of the high density phase where $\a > \max(1,\b)$, and $\a \b>1$ as well as in the shock region of the low density phase where $\a > \max(1,\b)$, and $\a \b>1$. 
\end{proposition}

 We will in the following only describe the necessary adjustments in the proof of Theorems \ref{thm:mainHighLowLight} and~\ref{thm:MainMixingTime} in order to establish the corresponding results in the shock region. The key tool is to establish the following analogue of Lemma~4.11 in \cite{SY:openLight},  which controls the position of the light particles when the particles are initially sufficiently far away from the boundary. To this end, we consider an open ASEP $(\tau(t))_{t \geq 0}$ with $r$ many light particles at locations $(\lc_i(t))_{t \geq 0}$ for all $i \in \llb 1,r \rrb$. The following statement is Lemma~4.11 in \cite{SY:openLight}, but without the restriction that $\a\b \leq 1$.

\begin{lemma}\label{lem:NewLinearSpeed}
Let $\eta \sim \mu_N$ and set $\kappa=(1-q)(1-2\rho)$. Then there exists some $t_0=t_0(\a,\b,r)$ and $c=c(\a,\b,r)$ such that for all $t \geq t_0$, and all $0 < s_1 < s_2 < \dots <s_r$ with some $s_r \in \lBr 1, N-4t \rBr$, 
\begin{equation}\label{eq:Progress2}
    \mathbb{P}\left( \lc_r(t) \leq \max\left(s_r - \frac{|\kappa|}{2} t,4t\right) \, \Big| \,  \lc_i(0) = s_i \, \mbox{ }\forall i \in \llb 1,r \rrb ,\mbox{ } \,  \tau_x=\eta_x \, \mbox{ } \forall x \notin \{s_i\} \right) \geq 1  - \exp(-c t^{1/8})  ,
\end{equation} provided that $N$ is sufficiently large.
\end{lemma}
\begin{proof}
By the same arguments as in the proof of Lemma 4.11 in \cite{SY:openLight}, we assume without loss of generality that $s_1 \geq 2t$, since the event $\lc_r(t) \leq \max\left(s_r -  |\kappa| t/2,4t\right)$ does with probability at least $1-\exp(-c_1 t)$ for some constant $c_1>0$ not depend on the choice of $(\tau_x(0))_{x \leq 2t}$, and the fact that particles travel at most at speed $1$.   Let $\tilde{\zeta} \in \{0,1\}^{\mathbb{Z}}$ be chosen according to a Bernoulli-$\rho$-product law on $\{ 0,1 \}^{N}$, where we set $\rho=\a/(\a+1)$. Let $(\zeta(t))_{t \geq 0}$ denote the corresponding open ASEP with light particles with effective densities $\rho$, started from $\tilde{\zeta}$ and $r$ many light particles at locations $(s_i)_{i \in \llb 1,r \rrb}$. Using Theorem \ref{thm:ApproxHighLow}, there exists a coupling $\mathbf{P}$ such that uniformly in $1 \leq t \leq \frac{2}{3}N$, 
\begin{equation}\label{eq:InitialCoupled}
    \mathbf{P}\left( \tau_x(0)=\tilde{\zeta}_x \ \forall x\in \Big[\frac{3}{2}t, N-\frac{3}{2}t\Big] \right) \geq 1 - \exp(-c_1 t) , 
\end{equation} with some constants $c_1>0$, and $N$ sufficiently large. Moreover, using the basic coupling, and a standard Chernoff bound on the (at most linear) speed of particles, we can extend \eqref{eq:InitialCoupled} to a coupling $\tilde{\mathbf{P}}$ between $(\tau(t))_{t \geq 0}$ and $(\zeta(t))_{t \geq 0}$ such that 
\begin{equation}
     \tilde{\mathbf{P}}\left( \tau_x(t)=\zeta_x(t) \ \forall  x\in \Big[\frac{1}{4}t, N-\frac{1}{4}t\Big] \right) \geq 1 - \exp(-c_2 t)  
\end{equation} for some constants $c_2,C_2$, and all $N$ large enough, allowing us to conclude the desired bound \eqref{eq:Progress2} by Theorem \ref{thm:mainHighLowLight} for the process $(\zeta(t))_{t \geq 0}$.
\end{proof}

\begin{proof}[Proof of Proposition \ref{pro:mainHighLowLightShock}]
Note that the only place where the assumption $\a\b\leq 1$ is used to derive Theorems~\ref{thm:mainHighLowLight} and~\ref{thm:MainMixingTime} is Lemma~4.11 in \cite{SY:openLight}. Using Lemma \ref{lem:NewLinearSpeed} in place of Lemma 4.11 in \cite{SY:openLight}, we conclude.
\end{proof}

\appendix

\section{Regeneration times for random walks}\label{sec:Appendix}

In this appendix, we collect several results on regeneration times for random walks. A key difference to standard references on renewal theory is that we require subexponential tail bounds for the involved regeneration times, and therefore provide a self-contained independent proof. Recall the following standard definition of a subexponential random variable.

\begin{definition}
     We will say that a probability measure $\nu$ on $\mathbb{R}_{\geq0}$ is subexponential if there exist constants $C,c>0$ such that $\nu([y,\infty))\leq Ce^{-cy}$ for all $y\in\mathbb{R}_{\geq0}$. We will say that a random variable $X$ on $\mathbb{R}_{\geq0}$ is subexponential if its distribution is subexponential. For $C,c>0$ we write $X\in\mathsf{Subexp}(C,c)$ if $\mathbb{P}(X\geq y)\leq Ce^{-cy}$ for all $y\in\mathbb{R}_{\geq0}$. 
\end{definition}

We start by collecting some basic properties for subexponential random variables.

\begin{lemma}\label{lem:basic exponential tails}
Let $X\in\mathsf{Subexp}(C,c)$ for some $C,c>0$. Then $X$ has the following properties.
    \begin{enumerate}
        \item [\textup{(\romannum{1})}] A non-negative random variable $X$ is subexponential if and only if $\mathbb{E}[e^{\theta X}]<\infty$ for some $\theta>0$. Specifically, if $X\in\mathsf{Subexp}(C,c)$, then $\mathbb E[e^{\theta X}]\leq1+ \theta C/(c-\theta)<\infty$ for every $\theta\in(0,c)$; and conversely, if $\mathbb E[e^{\theta X}]\leq M$ for some $\theta,M\in(0,\infty)$, then $X\in\mathsf{Subexp}(M,\theta)$. 
        \item [\textup{(\romannum{2})}] If $0\leq X\leq Y$ almost surely and $Y$ is subexponential, then $X$ is also subexponential.
        Specifically, if $0\leq X\leq Y$ and $Y\in\mathsf{Subexp}(C,c)$, then $X\in\mathsf{Subexp}(C,c)$. 
        \item [\textup{(\romannum{3})}] If $X$ and $Y$ are both subexponential, then $X+Y$ is subexponential. Specifically, if $X\in\mathsf{Subexp}(C_1,c_1)$ and $Y\in\mathsf{Subexp}(C_2,c_2)$, then $X+Y\in\mathsf{Subexp}(C_1+C_2,\min(c_1,c_2)/2)$. 
        \item [\textup{(\romannum{4})}] If $X_1,X_2,\dots$ are i.i.d. subexponential variables, and $K$ is a subexponential variable with values in $\mathbb{Z}_{\geq0}$, not necessarily independent of $X_1,X_2,\dots$, then $X_1+\dots+X_K$ is also subexponential. Specifically, if $X_1,X_2,\dots$ are i.i.d., $X_1\in\mathsf{Subexp}(C_0,c_0)$, and $K$ is $\mathbb{Z}_{\geq0}$-valued with $K\in\mathsf{Subexp}(C_K,c_K)$, then there exist $C,c>0$ depending only on $C_0,c_0,C_K,c_K$, such that $X_1+\dots+X_K\in\mathsf{Subexp}(C,c)$.
    \end{enumerate} 
\end{lemma} 

\begin{proof}
 For Statement \textup{(\romannum{1})}, suppose first that $X\in\mathsf{Subexp}(C,c)$. For $\theta\in(0,c)$, 
    \[
        \mathbb{E}\left[e^{\theta X}\right]
        =
        1+\theta\int_0^\infty e^{\theta y}\mathbb{P}(X>y)\,dy
        \leq
        1+\theta C\int_0^\infty e^{-(c-\theta)y}\,dy
        =
        1+\frac{\theta C}{c-\theta}
        <\infty .
    \]
    Conversely, if $\mathbb{E}[e^{\theta X}]\leq M$ for \(\theta,M\in(0,\infty)\), then Markov's inequality gives
     $\mathbb{P}(X\geq y) 
    \leq
    Me^{-\theta y}$ for   $y\in\mathbb R_{\geq0}$.  

Statement \textup{(\romannum{2})} follows directly from the definition. 

For Statement \textup{(\romannum{3})}, since $X\in\mathsf{Subexp}(C_1,c_1)$ and $Y\in\mathsf{Subexp}(C_2,c_2)$,
\[
 \mathbb{P}(X+Y\geq y)\leq \mathbb{P}(X\geq y/2)+\mathbb{P}(Y\geq y/2)\leq C_1 e^{-c_1 y/2}+C_2 e^{-c_2 y/2}\leq (C_1+C_2)e^{-\frac12\min\{c_1,c_2\}y}\quad\mbox{for }y\in\mathbb{R}_{\geq0}.
\]  

For Statement \textup{(\romannum{4})}, let $S_0=0$, $S_n=\sum_{i=1}^n X_i$. We have $S_K=\sum_{n=0}^{\infty}S_n\mathbf{1}_{\{K=n\}}$.
Since $X_1\in\mathsf{Subexp}(C_0,c_0)$, part \textup{(\romannum{1})} gives $\theta,M\in(0,\infty)$ depending only on $C_0,c_0$ such that $\mathbb{E}[e^{\theta X_1}]\leq M$. 
Fix $y\geq0$, note that
\[
    \{S_K\geq y\}
    \subseteq
    \{K>\lfloor \upsilon y\rfloor\}\cup\{S_{\lfloor \upsilon y\rfloor}\geq y\},
\]
where we set $\upsilon= \theta/(2(1+\log M))>0$.
Since $K\in\mathsf{Subexp}(C_K,c_K)$, thus $\mathbb{P}(K>\lfloor \upsilon y\rfloor) 
    \leq
    C_Ke^{-c_K \upsilon y}$, and
\[
    \mathbb{P}\left(S_{\lfloor \upsilon y\rfloor}\geq y\right)
    =
    \mathbb{P}\left(e^{\theta S_{\lfloor \upsilon y\rfloor}}\geq e^{\theta y}\right)
    \leq
    e^{-\theta y}\mathbb{E}\left[e^{\theta S_{\lfloor \upsilon y\rfloor}}\right]\leq e^{-\theta y}M^{\lfloor \upsilon y\rfloor}\leq
    e^{-\theta y}e^{\upsilon y\log M}
    =
    e^{-\left(\theta-\upsilon \log M\right)y}.
\]
Hence $\mathbb{P}(S_K\geq y)
    \leq 
    (C_K+1)e^{-\min (c_K\upsilon,(\theta-\upsilon\log M) )y}$. Since $\theta-\upsilon\log M\geq\theta/2>0$, $S_K$ is subexponential. 
\end{proof}

\subsection{Regeneration times for the infinite biased random walk}

Recall the following standard definition of regeneration times for simple random walks.

\begin{definition}\label{def:regeneration}
Assume $N\in\mathbb{N}\cup\{\infty\}$. Let $\{Y_i\}_{i\in\llbracket0,N\rrbracket}$ be a path on $\mathbb{Z}$ with
$Y_i-Y_{i-1}\in\{-1,0,1\}$ for all $i\in\llbracket1,N\rrbracket$.
We call $n\in\llbracket0,N\rrbracket$ a \emph{regeneration time} for this path if
$Y_i<Y_n$ for all $i<n$ and $Y_i>Y_n$ for all $i>n$.
When $n\in\llbracket0,N\rrbracket$ is a regeneration time, we call $Y_n$ a \emph{regeneration level}.
\end{definition}

We have the following result on regeneration times of a positive $\nu_{\b}$-random walk. Let us stress that the overall proof strategy is reminiscent of the arguments in \cite{dembo1996tail} for random walks in random environments. However, since our regeneration times are in general not stopping times, we require a more refined analysis. 

\begin{proposition}\label{lem:regeneration}
    Assume $\b>1$. Let $\{Y_i\}_{i=0}^{\infty}$ be the positive $\nu_{\b}$-random walk started at $Y_0=0$ with infinitely many steps as in Definition \ref{def:positive lazy random walk}. Then, almost surely, there exist infinitely many regeneration times for this path given by random numbers $0=n_0<n_1<n_2<\cdots$. Furthermore, the increments $(n_i-n_{i-1},Y_{n_i}-Y_{n_{i-1}})$ for $i\in\mathbb{N}$ are i.i.d., and both coordinates are subexponential. More precisely, there exist constants $C,c>0$, depending only on $\b$, such that for each $y\in\mathbb{N}$, we have
    $\mathbb{P}(n_1\geq y)\leq Ce^{-cy}$ and $\mathbb{P}(Y_{n_1}\geq y)\leq Ce^{-cy}$.
    Moreover, we know that $\mathbb{P}((n_1, Y_{n_1})=(1,1))=\nu_{\b}(1)>0$.
\end{proposition}

\begin{proof} 
For clarity, we split the proof into four steps.
In the first step, we prove that the first regeneration level $Y_{n_1}$ exists and is subexponential. 
In the second step, we show that the first regeneration time $n_1$ is also subexponential. 
In the third step, we construct the regeneration times $n_i$, $i\in\mathbb{N}$ and show that the increments between $(n_i,Y_{n_i})$ are i.i.d. 
In the fourth step, we prove that the event $(n_1, Y_{n_1})=(1,1)$ has probability $\nu_{\b}(1)$.  \\

\noindent\textbf{Step 1.}
Let
\[
n_1=\inf\{n\ge 1:n\text{ is a regeneration time for }\boldsymbol Y\},
\]
with the convention that $n_1=\infty$ if this set is empty. Set
$V=Y_{n_1}$ on $\{n_1<\infty\}$ and $V=\infty$ on $\{n_1=\infty\}$.
In this step we show that there exist constants $C,c>0$ depending only on $\b$ such that
$\mathbb{P}(V\geq y)\leq Ce^{-cy}$ for $y\in\mathbb{N}$. 
Consider an ordinary $\nu_b$-random walk $\{S_i\}_{i=0}^{\infty}$ started at $S_0=0$. 
Denote $u_0=\mathbb{P}(S_i=-1\mbox{ for some }i\in\mathbb{N})=1/\b^2$ and $u_1=\mathbb{P}(S_i=0\mbox{ for some }i\in\mathbb{N})$. Then $u_0,u_1\in(0,1)$. We denote $u=\max(u_0,u_1)$ and choose $\rho\in(u,1)$.
We then choose $d\in\mathbb{N}$ sufficiently large, such that 
$\rho^{d-1}\leq (\rho-u)/(1-\rho)$ and $\rho+\rho^d<1$. 
To show that $V$ is subexponential, it suffices to prove that
\begin{equation}\label{eq:bereryhj57}
\mathbb{P}(V>kd)\leq\left(\rho+\rho^d\right)^{k-1}\quad\mbox{ for }k\in\mathbb{Z}_{\geq0}.
\end{equation}
We next prove \eqref{eq:bereryhj57} by induction.
When $k=0$, \eqref{eq:bereryhj57} holds since the right side is $>1$. 
We assume $k\geq1$ and that \eqref{eq:bereryhj57} holds for each $\ell\in\llbracket1,k-1\rrbracket$. For $y\in\mathbb{N}$, we denote by $\tau_y=\inf\{i\in\mathbb{Z}_{\geq0}: Y_i=y\}$ the hitting time of level $y$, and observe that $\tau_y<\infty$ almost surely. We next consider the process $\{Y_i\}_{i\leq\tau_y}$, i.e., the killed process after hitting level $y$, and denote its first regeneration level by $V^y$, using the same definition for the regeneration times as for the process without killing. By definition, we have $V^y\leq V$ and $V^y\in\llbracket1,y\rrbracket$. For each $w\in\llbracket1,y\rrbracket$,  we have
\begin{equation}\label{eq:ntndfy}
    \mathbb{P}(V>y|V^y=w)\leq\mathbb{P}(\exists i\in\mathbb{N} \, \colon \, Z_i=w)\leq\mathbb{P}(\exists i\in\mathbb{N} \, \colon \, S'_i=w)=u_0^{y-w}\cdot\mathbf{1}_{w\in\llbracket1,y-1\rrbracket}+u_1\cdot \mathbf{1}_{w=y},
\end{equation} 
where $\{Z_i\}_{i=0}^{\infty}$ is the $\nu_{\b}$-random walk started at $Z_0=y$ and conditioned on staying in $\mathbb{Z}_{\geq 1}$, and $\{S_i'\}_{i=0}^{\infty}$ is the ordinary $\nu_{\b}$-random walk started at $S_0'=y$.  
In the first step above, we use that conditioned on $\{V^y=w\}$, the event $\{V>y\}$ implies that the process $\{Y_i\}_{i>\tau_y}$ must visit $w$; in the second step we use the monotone coupling from Lemma \ref{lem:basic hitting monotone coupling}; and in the third step we use the definition of $u_0$ and $u_1$.
Therefore,
\begin{equation}\label{eq:reevgre}
    \begin{split}
        \mathbb{P}(V>kd)&=\sum_{\ell=1}^{kd} \mathbb{P}\left( V>kd|V^{kd}=\ell\right)\cdot \mathbb{P}\left(V^{kd}=\ell\right)\\
        &\leq\sum_{m=1}^{k-1}u_0^{kd-md}\cdot\mathbb{P}\left((m-1)d<V^{kd}\leq md\right)+u\cdot\mathbb{P}\left((k-1)d<V^{kd}\leq kd\right)\\
        &\leq\sum_{m=1}^{k-1}\rho^{kd-md}\cdot\mathbb{P}\left((m-1)d<V^{kd} \right)+u\cdot\mathbb{P}\left((k-1)d<V^{kd} \right)\\
        &\leq\sum_{m=1}^{k-1}\rho^{kd-md}\cdot \left(\rho+\rho^d\right)^{m-2}+u\cdot \left(\rho+\rho^d\right)^{k-2}\leq\left(\rho+\rho^d\right)^{k-1},
    \end{split}
\end{equation}
which completes the induction and concludes the proof that $V$ is subexponential. Here, in the first step above we condition on the value of $V^{kd}$; in the second step we use \eqref{eq:ntndfy} and that $u=\max(u_0,u_1)$; in the third step we use $u_0\leq u<\rho$; in the fourth step we use the induction hypothesis; and in the last step we use the condition $\rho^{d-1}\leq (\rho-u)/(1-\rho)$ when choosing $d$.
Specifically, when $k\geq2$ this condition implies
\[
\sum_{j=1}^{k-1}\left(\frac{\rho^d}{\rho+\rho^d}\right)^j+u\leq \sum_{j=1}^{\infty}\left(\frac{\rho^d}{\rho+\rho^d}\right)^j+u=\rho^{d-1}+u\leq \rho+\rho^d,
\]
which is equivalent to the last inequality of \eqref{eq:reevgre}; and when $k=1$, this last inequality of \eqref{eq:reevgre} is equivalent to $u\leq \rho+\rho^d$ which follows from $u<\rho$.\\ 

\noindent\textbf{Step 2.} In this step, we show that the first regeneration time $n_1$ is also subexponential.
For $y\in\mathbb{N}$, we recall that $\tau_y=\inf\{i\in\mathbb{Z}_{\geq0}: Y_i=y\}$ is the hitting time of level $y$. Observe that $\tau_y<\infty$ almost surely.
    Using that $n_1$ is a regeneration time and that each step of the path is in $\{-1,0,1\}$, one can check that $\{Y_{n_1}\geq y\}=\{n_1\geq\tau_y\}$. 
    Let $\kappa=(\b-1)/(2(\b+1))$, we have
    \[
    \mathbb{P}\left(n_1\geq y\right)\leq\mathbb{P}\left(\tau_{ \lceil\kappa y\rceil}\geq y \right)+\mathbb{P}\left(n_1\geq\tau_{ \lceil\kappa y\rceil}\right) \leq
    \mathbb{P}\left( Y_{y-1}<\lceil\kappa y\rceil \right)+\mathbb{P}\left(Y_{n_1}\geq  \lceil\kappa y\rceil\right)
    \leq C_1e^{-c_1y} + Ce^{-c\lceil\kappa y\rceil}.
    \]
    Hence we conclude that the first regeneration time $n_1$ is also subexponential.
    In the first step above, we use $\{n_1\geq y\}\subset\{\tau_{ \lceil\kappa y\rceil}\geq y\}\cup\{n_1\geq\tau_{ \lceil\kappa y\rceil}\}$; in the second step, we use $\{\tau_{ \lceil\kappa y\rceil}\geq y\}\subset\{Y_{y-1}<\lceil\kappa y\rceil\}$ (which follows from the definition of $\tau_y$) and $\{n_1\geq\tau_{ \lceil\kappa y\rceil}\}=\{Y_{n_1}\geq  \lceil\kappa y\rceil\}$; and in the third step, we use Lemma \ref{lem:basic hitting monotone coupling} to bound the first summand and the exponential tail bound of $Y_{n_1}$ from the first step to bound the second summand.\\

\noindent\textbf{Step 3.} 
In this step, we show that there almost surely exist infinitely many regeneration times $0=n_0<n_1<n_2<\cdots$, and that the increments $(n_i-n_{i-1},Y_{n_i}-Y_{n_{i-1}})$ for $i\in\mathbb{N}$ are i.i.d.
For $m\in\mathbb Z_{\ge0}$, define the shift operator $\Theta_m(\boldsymbol{\omega})=(\omega_{m+i}-\omega_m)_{i\ge0}$ for a sequence $\boldsymbol{\omega}=(\omega_i)_{i\geq0}$. We next show that $\Theta_{n_1}\boldsymbol{Y}$ is independent of the $\sigma$-field $\mathcal G=\sigma (n_1,(Y_i)_{i\in\llbracket0, n_1\rrbracket} )$ and has the same law as $\boldsymbol{Y}=(Y_i)_{i\geq0}$.
Let $\boldsymbol{S}=(S_i)_{i\geq0}$ be the ordinary $\nu_{\b}$-random walk started at $S_0=0$. Recall
$E_\infty=\{S_1=1 \text{ and } S_i\ge1\mbox{ for } i\in\mathbb{N}\}$, and note that by Definition~\ref{def:positive lazy random walk}, the law of $\boldsymbol{Y}=(Y_i)_{i\geq0}$ is the law of $\boldsymbol{S}=(S_i)_{i\geq0}$ conditioned on the event $E_{\infty}$.
For each measurable set $F$ (for the $\sigma$-algebra generated by the walk) and each $n\in\mathbb{N}$, $\bg\in\lrw_n$ satisfying 
$A_{n,\bg}=\{n_1=n,\, (Y_i)_{i\in\llbracket0,n\rrbracket}=\bg\}\neq\emptyset$,
we have
\begin{equation*}
    \begin{split}
        &\mathbb{P}\left( n_1=n,\, (Y_i)_{i\in\llbracket0,n\rrbracket}=\bg,\,\Theta_{n_1}\boldsymbol{Y}\in F \right)
        =\mathbb{P}\left( (S_i)_{i\in\llbracket0,n\rrbracket}=\bg ,\,\Theta_n \boldsymbol{S}\in E_\infty\cap F\right)/\mathbb P(E_\infty)\\
        &=\mathbb{P}\left( (S_i)_{i\in\llbracket0,n\rrbracket}=\bg \right)\cdot\mathbb{P}\left( \boldsymbol{S}\in E_\infty\cap F\right)/\mathbb P(E_\infty)
        =\mathbb P(n_1=n,\, (Y_i)_{i\in\llbracket0,n\rrbracket}=\bg)\cdot\mathbb P(\boldsymbol{Y}\in F).
    \end{split}
\end{equation*}
In the first step above we use the fact that 
\[
\{n_1(\boldsymbol{S})=n,\, (S_i)_{i\in\llbracket0,n\rrbracket}=\bg\}\cap E_\infty
=
\{(S_i)_{i\in\llbracket0,n\rrbracket}=\bg\}\cap\{\Theta_n \boldsymbol{S}\in E_\infty\},
\]
which follows from $A_{n,\bg}\neq\emptyset$, since the prefix $\bg$ already determines that $n$ is the first positive regeneration time, except for the requirement that after time $n$ the walk stays strictly above $\gamma_n$; in the second step we use that $\Theta_n \boldsymbol{S}$ is independent of $(S_i)_{i\in\llbracket0,n\rrbracket}$ and has the same law as $\boldsymbol{S}$; and in the third step we use 
\[
\mathbb P(n_1=n,\, (Y_i)_{i\in\llbracket0,n\rrbracket}=\bg)=\mathbb{P}\left((S_i)_{i\in\llbracket0,n\rrbracket}=\bg,\Theta_n\boldsymbol{S}\in E_{\infty}\right)/\mathbb P(E_\infty)
=\mathbb{P}\left( (S_i)_{i\in\llbracket0,n\rrbracket}=\bg \right),
\]
and $\mathbb{P}\left( \boldsymbol{S}\in E_\infty\cap F\right)/\mathbb P(E_\infty)=\mathbb P(\boldsymbol{Y}\in F)$.
Since every $\mathcal{G}$-measurable set is a disjoint union of subsets $A_{n,\bg}\neq\emptyset$ for $n\in\mathbb{N}$ and  $\bg\in\lrw_n$, it follows that $\Theta_{n_1}\boldsymbol{Y}$ is independent of $\mathcal{G}$ and has the same law as $\boldsymbol{Y}$.

Let $n_0=0$, and for each $k\in\mathbb{Z}_{\geq0}$, we inductively define $n_{k+1}-n_k$ as the first positive
regeneration time of the shifted path $\Theta_{n_k}\boldsymbol{Y}$. 
Note that if $n_k$ is a regeneration time of $\boldsymbol{Y}$ and $r$ is a regeneration time of
$\Theta_{n_k}\boldsymbol Y$, then $n_{k+1}=n_k+r$ is a regeneration time of
$\boldsymbol Y$. Since the first positive regeneration time of a positive $\nu_{\b}$-random walk is finite almost
surely by Step 2, this induction gives infinitely many regeneration times $0=n_0<n_1<\dots$ of $\boldsymbol{Y}$.
Moreover, by an induction on $k$, we know that $\Theta_{n_k}\boldsymbol{Y}$ is independent of the past up to time $n_k$ and has the law of a positive $\nu_{\b}$-random walk. Therefore
$(n_{k+1}-n_k,\ Y_{n_{k+1}}-Y_{n_k})$
is independent of the past and has the same law as $(n_1,Y_{n_1})$. Hence $(n_i-n_{i-1},\,Y_{n_i}-Y_{n_{i-1}})$ for $i\in\mathbb N$ are i.i.d.\\

\noindent\textbf{Step 4.} 
In this step, we show that $\mathbb{P}((n_1, Y_{n_1})=(1,1))=\nu_{\b}(1)$. Note that
\[
\{(n_1, Y_{n_1})=(1,1)\}=\{Y_1=1\mbox{ and }Y_i\geq2\mbox{ for all }i\geq2\}.
\]
With $\boldsymbol{S}$ and $E_{\infty}$ above, by the independence of increments and Lemma \ref{lem:basic hitting monotone coupling} with $h=1$,
\[
\mathbb{P}(S_1=1\mbox{ and }S_i\geq2\mbox{ for all }i\geq2)=\nu_{\b}(1)^2(1-\b^{-2}),\quad
\mathbb{P}(E_\infty)= \nu_{\b}(1)(1-\b^{-2}).
\]
Hence, we have $\mathbb{P}((n_1, Y_{n_1})=(1,1))=\mathbb{P}(S_1=1\mbox{ and }S_i\geq2\mbox{ for all }i\geq2\,|\,E_{\infty})=\nu_{\b}(1)$.
\end{proof}

\subsection{Regeneration times for the path measures}

We now provide similar results for the regeneration times of the open TASEP path measure. 

\begin{lemma}\label{lem:regeneration structure TASEP path measure}
    Assume $\b > \max(\a,1)$ and $\a\b>1$. Consider the probability measure $\mathrm{P}^{(0,\mathsf{a},\mathsf{b})}$ on $\lrw_N$ as in Definition \ref{def:approx TASEP measure}. Then, after enlarging the probability space if necessary, one can construct an infinite extension of a random path of length $N$  with law $\mathrm{P}^{(0,\mathsf{a},\mathsf{b})}$ together with random numbers 
    $0=m_0<m_1<m_2<\dots$,
    such that $m_i$ is a regeneration time of this extended path for each $i\in\mathbb{N}$. Furthermore, the increments $m_i-m_{i-1}$ for $i\in\mathbb{N}$ are independent, and the increments $m_i-m_{i-1}$ for $i\geq2$ are i.i.d. Moreover, there exist constants $C,c>0$ depending only on $\a$ and $\b$, such that $\mathbb{P}(m_1\geq y)\leq Ce^{-cy}$ and $\mathbb{P}(m_2-m_1\geq y)\leq Ce^{-cy}$ for $y\in\mathbb{N}$. We also have $\mathbb{P}(m_2-m_1=1)=\nu_{\b}(1)>0$.
\end{lemma}
\begin{proof}
    We construct the infinite random path $\bg=(\gamma_i)_{i\geq0}$ as follows. We first sample $Q_N\in\llbracket0,N\rrbracket$  from Definition~\ref{def:approx TASEP measure} and $(\gamma_i)_{i\in\llbracket0,Q_N\rrbracket}$  as under $\mathbb{P}^{(0,\a,\b)}$ and then attach an independent positive $\nu_{\b}$-random walk $\boldsymbol{Y}=\{Y_i\}_{i=0}^{\infty}$ so that $\gamma_{Q_N+i}=\gamma_{Q_N}+Y_i$ for each $i\in\mathbb{Z}_{\geq0}$. By Definition \ref{def:positive lazy random walk}, the restriction of this infinite path $\bg$ to $\llbracket0,N\rrbracket$ has law $\mathrm{P}^{(0,\mathsf{a},\mathsf{b})}$ on $\lrw_N$. We define the random variable 
    \begin{equation}\label{eq:etgetrg}
    L_N=\max_{i\in\llbracket0,Q_N\rrbracket}\gamma_i-\gamma_{Q_N},
    \end{equation}
    and we can observe that $0\leq L_N\leq Q_N$. As shown by \eqref{eq:rgregrt} in the proof of Lemma \ref{lem:TV distance replaced measure}, the random variable $Q_N$ (and hence also $L_N$) are subexponential. 

    We denote by $0=n_0<n_1<n_2<\dots$ the regeneration times of the positive $\nu_{\b}$-random walk $\boldsymbol{Y}$ as in Lemma \ref{lem:regeneration}.
    Since the path $\boldsymbol{Y}$ is independent of $(\gamma_i)_{i\in\llbracket0,Q_N\rrbracket}$, these regeneration times $n_i$, $i\in\mathbb{N}$ are also independent of $(\gamma_i)_{i\in\llbracket0,Q_N\rrbracket}$, and hence $n_i$ are also independent of $L_N$ given by \eqref{eq:etgetrg} and of $Q_N$.
    We write $(X_i,H_i)=(n_i-n_{i-1},Y_{n_i}-Y_{n_{i-1}})$ for $i\in\mathbb{N}$, which are i.i.d. and independent of $L_N$ and $Q_N$. 
    Define  
    \begin{equation}\label{eq:etgetrg2}
    K:=\min\{i\in\mathbb{N}:Y_{n_i}\geq L_N+1\}=\min\{i\in\mathbb{N}:H_1+\dots+H_i\geq L_N+1\}.
    \end{equation}
    One can observe that the positive regeneration times of the extended path $\bg=(\gamma_i)_{i\geq0}$ are precisely 
    \[
    m_1=Q_N+n_{K},\quad m_2=Q_N+n_{K+1},\quad m_3=Q_N+n_{K+2},\quad\dots
    \]
    We can check that the random number $K$ given by \eqref{eq:etgetrg2} is a stopping time with respect to filtration $\mathcal{F}_i=\sigma(Q_N,L_N,(X_r,H_r)_{r\in\llbracket1,i\rrbracket})$. 
    Hence $(X_{K+1},X_{K+2},\dots)$ is independent of $\mathcal{F}_K$ and has the same law as $(X_{1},X_{2},\dots)$. Since $m_1=Q_N+n_{K}$ is $\mathcal{F}_K$-measurable and $m_i-m_{i-1}=X_{K+i-1}$ for $i\geq2$, we conclude that the increments $m_i-m_{i-1}$ for $i\in\mathbb{N}$ are independent, and the increments $m_i-m_{i-1}$ for $i\geq2$ are i.i.d. 

    By Lemma \ref{lem:regeneration}, $m_2-m_1$ is subexponential and $\mathbb{P}(m_2-m_1=1)=\nu_{\b}(1)$.
    We next show that $m_1=Q_N+n_{K}$ is subexponential. As shown by \eqref{eq:rgregrt} in the proof of Lemma \ref{lem:TV distance replaced measure}, $Q_N$ is subexponential. By Lemma \ref{lem:basic exponential tails} (\romannum{3}), it suffices to show that $n_K$ is subexponential. Observe from \eqref{eq:etgetrg2} that $K\leq L_N+1$. Note 
    \[
    n_K=\sum_{i=1}^KX_i\leq \sum_{i=1}^{L_N+1}X_i.
    \]
    Since $X_1,X_2,\dots$ are i.i.d. and independent of $L_N+1$, and that both $X_1$ and $L_N+1$ are subexponential, in view of Lemma \ref{lem:basic exponential tails} (\romannum{2}), (\romannum{4}), we conclude that $n_K$ is subexponential. 
\end{proof}

We have the following result on the regeneration times under the random path measure $\mathrm{P}^{(0,\mathsf{a},\mathsf{b})}$.

\begin{lemma}\label{lem:regeneration indices}
    Assume $\b > \max(\a,1)$ and $\a\b>1$. Consider a random path $\bg\sim\mathrm{P}^{(0,\mathsf{a},\mathsf{b})}$, where $\mathrm{P}^{(0,\mathsf{a},\mathsf{b})}$ is the probability measure on $\lrw_N$ as in Definition \ref{def:approx TASEP measure}. 
    Consider the composition $\bs(\bg)$ associated to $\bg$ as in \eqref{eq:Signature} and the $0$-$1$ sequence $\boldsymbol{\mathsf{S}}^{\bs(\bg)}\in\{0,1\}^{N+1}$ associated to the composition $\bs(\bg)$ as in Definition \ref{def:associated 01 seq}. 
    Then, after enlarging the probability space if necessary, there exist a random integer $K\in\mathbb{Z}_{\geq0}$, $\mathbb{N}$-valued random variables $R_i$ for $i\in\mathbb{N}$, and, on the event $\{K\geq1\}$, $\mathbb{N}$-valued random variables
    \begin{equation*}
        1\leq \mathrm{n}_1<\dots<\mathrm{n}_K\leq N,
    \end{equation*}
    such that 
\[
K=\max\{j\in\mathbb Z_{\geq0}:R_1+\cdots+R_j\leq N\},
\]
and satisfying the following properties, with constants $C,c>0$ depending only on $\a$ and $\b$:
    \begin{enumerate}
        \item [(\romannum{1})] On the event $\{K\geq 1\}$, we have
\begin{align*}
    &\mathsf{S}^{\bs(\bg)}(\mathrm{n}_i)=1\quad\mbox{ for all }i\in\llbracket1,K\rrbracket,\\
    &\mathrm{n}_i-\mathrm{n}_{i-1}=R_{K+2-i}\quad\mbox{ for all }i\in\llbracket2,K\rrbracket,\quad
        \mbox{and}\quad
        N-\mathrm{n}_K=R_1-1.
\end{align*}
        \item [(\romannum{2})] We have $R_i$ for $i\in\mathbb{N}$ are independent and $R_i$ for $i\geq2$ are i.i.d. Moreover, for all $y\in\mathbb{N}$ we have
        \begin{align*}
            & \mathbb{P}(K\geq 1, \mathrm{n}_1\geq y)\leq Ce^{-cy},\quad \mathbb{P}(K\geq 1,  N-\mathrm{n}_K\geq y)\leq Ce^{-cy},\\
            & \mathbb{P}(K=0)\leq Ce^{-cN},\quad \mathbb{P}(R_i\geq y)\leq Ce^{-cy}  \mbox{    for } i\in\mathbb{N},\quad\mbox{ and } \mathbb{P}(R_2=1)=\nu_{\b}(1).
        \end{align*} 
\end{enumerate} 
\end{lemma}
\begin{proof}
    We choose regeneration times $0=m_0<m_1<m_2<\dots$ for the infinite extension of the path $\bg\sim\mathrm{P}^{(0,\mathsf{a},\mathsf{b})}$ from Lemma \ref{lem:regeneration structure TASEP path measure}. 
    Define 
    \begin{equation}\label{eq:bhrttsbht}
        \begin{split} 
        &R_1=m_1,\quad R_i=m_i-m_{i-1} \quad\mbox{ for }i\geq2,\\
        &K=\max\{j\in\mathbb{Z}_{\geq 0}:m_j\leq N\},\quad \mathrm{n}_i=N+1-m_{K+1-i}\quad\mbox{ for }i\in\llbracket1,K\rrbracket.
    \end{split}
    \end{equation}
On the event $\{K\geq1\}$, since $1\leq m_1<m_2<\dots<m_K\leq N$, we have $1\leq \mathrm{n}_1<\mathrm{n}_2<\dots<\mathrm{n}_K\leq N$. 
We also have
$\mathrm{n}_i-\mathrm{n}_{i-1}=m_{K+2-i}-m_{K+1-i}=R_{K+2-i}$ for $i\in\llbracket2,K\rrbracket$ and $N-\mathrm{n}_K=m_1-1=R_1-1$.
For each $i\in\llbracket1,K\rrbracket$, since $m_i$ is a regeneration time of $\bg$, we have
\[
    N+1-m_i=\#\{j\in\llbracket0,N\rrbracket:\gamma_j\geq \gamma_{m_i}\}=\sum_{\ell=\gamma_{m_i}-\min(\bg)}^{r}\sigma_{\ell},
\]
where we recall from \eqref{eq:Signature} that $r=\max(\gamma_i:i\in\llbracket0,N\rrbracket)-\min(\bg)$ and $\sigma_i=\#\{j\in\llbracket0,N\rrbracket:\gamma_j=\min(\bg)+i\}\in\mathbb{N}$ for $i\in\llbracket0,r\rrbracket$. By Definition \ref{def:associated 01 seq}, this implies $\mathsf{S}^{\bs(\bg)}(N+1-m_i)=1$, and hence 
$\mathsf{S}^{\bs(\bg)}(\mathrm{n}_i)=1$ for $i\in\llbracket1,K\rrbracket$.

Since the regeneration times $0=m_0<m_1<\dots$ are chosen from Lemma \ref{lem:regeneration structure TASEP path measure}, we know that $R_i$ for $i\in\mathbb{N}$ are independent,  $R_i$ for $i\geq2$ are i.i.d.,  $\mathbb{P}(R_i\geq y)\leq C_1e^{-c_1y}$ for each $i,y\in\mathbb{N}$, and $\mathbb{P}(R_2=1)=\nu_{\b}(1)$. By the definition \eqref{eq:bhrttsbht} of $K$, we have $\mathbb{P}(K=0)=\mathbb{P}(m_1>N)=\mathbb{P}(R_1>N)\leq C_1e^{-c_1N}$.
We also have
$\mathbb{P}(K\geq 1,\ N-\mathrm{n}_K\geq y)=\mathbb{P}(K\geq 1,\ R_1-1\geq y)\leq \mathbb{P}(R_1-1\geq y)\leq C_1e^{-c_1y}$ for $y\in\mathbb{N}$. Additionally,  

\begin{equation*}
    \begin{split}
        &\mathbb{P}(K\geq 1, \mathrm{n}_1\geq y)=\sum_{k=1}^{\infty}\mathbb{P}(K=k, \mathrm{n}_1\geq y)=\sum_{k=1}^{\infty}\mathbb{P}(K=k,    m_{k}\leq N+1- y)\\
        &=\sum_{k=1}^{\infty} \mathbb{P}(  m_{k}\leq N+1- y, m_{k+1}\geq N+1)
        =\sum_{k=1}^{\infty} \mathbb{P}(  m_{k}\leq N+1- y, R_{k+1}\geq N+1-m_k)\\
        &  
        =\sum_{k=1}^{\infty}\mathbb{E}\left[ {\bf1}_{m_{k}\leq N+1- y}\cdot \mathbb{P}\left( R_{k+1}\geq N+1-m_k| \sigma(R_1,\dots,R_k)\right)\right]\\
        & \leq C_1 \sum_{k=1}^{\infty}\mathbb{E}\left[ {\bf1}_{m_{k}\leq N+1- y}\cdot e^{-c_1(N-m_k)}\right]
        =  C_1 \sum_{k=1}^{\infty}\sum_{\ell=0}^{\infty}e^{-c_1\ell}\mathbb{P}\left(  m_k=N-\ell,\ell\geq y-1 \right)\\
        &=C_1\sum_{\ell=y-1}^{\infty}e^{-c_1\ell}\mathbb{E}\left[\#\{k\in\mathbb{N}:m_k=N-\ell\}   \right]\leq C_1\sum_{\ell=y-1}^{\infty}e^{-c_1\ell}\leq Ce^{-c y}.
    \end{split}
\end{equation*} 
Here, the inequality from the third line to the fourth line uses the fact that $m_k$ is $\sigma(R_1,\dots,R_k)$-measurable, the fact that $R_{k+1}$ is independent of $\sigma(R_1,\dots,R_k)$, and that $\mathbb{P}(R_{k+1}\geq y)\leq C_1e^{-c_1y}$ for $y\in\mathbb{N}$.    
\end{proof}

\bibliographystyle{goodbibtexstyle.bst}
\bibliography{RWApproximation}

@book{askey1985some,
  title={Some basic hypergeometric orthogonal polynomials that generalize Jacobi polynomials},
  author={Askey, Richard and Wilson, James Arthur},
  volume={319},
  year={1985},
  publisher={American Mathematical Soc.}
}

@article{uchiyama2004asymmetric,
  title={Asymmetric simple exclusion process with open boundaries and {Askey-Wilson} polynomials},
  author={Uchiyama, Masaru and Sasamoto, Tomohiro and Wadati, Miki},
  journal={J. Phys. A},
  volume={37},
  number={18},
  pages={4985},
  year={2004},
  publisher={IOP Publishing}
}

@article{enaud2004large,
  title={Large deviation functional of the weakly asymmetric exclusion process},
  author={Enaud, C and Derrida, B},
  journal={J. Stat. Phys.},
  volume={114},
  pages={537--562},
  year={2004},
  publisher={Springer}
}

@article{bryc2024two,
  title={{A two-line representation of stationary measure for open TASEP}},
  author={Bryc, W{\l}odzimierz and Zatitskii, Pavel},
  journal={Electron. J. Probab.},
  volume={29},
  pages={1--24},
  year={2024},
  publisher={The Institute of Mathematical Statistics and the Bernoulli Society}
}

@article{corteel2011tableaux,
  title={Tableaux combinatorics for the asymmetric exclusion process and {Askey--Wilson} polynomials},
  author={Corteel, S. and Williams, L. K},
  journal={Duke Math. J.},
  volume={159},
  number={3},
  pages={385--415},
  year={2011},
  publisher={Duke University Press}
}

@article{uchiyama2005correlation,
  title={Correlation function of asymmetric simple exclusion process with open boundaries},
  author={Uchiyama, M. and Wadati, M.},
  journal={J. Nonlinear Math. Phys. },
  volume={12},
  number={sup1},
  pages={676--688},
  year={2005},
  publisher={Taylor \& Francis}
}

@book{liggett1999stochastic,
  title={Stochastic interacting systems: contact, voter and exclusion processes},
  author={Liggett, Thomas M},
  volume={324},
  year={1999},
  publisher={Springer Science \& Business Media}
}

@book{gasper2011basic,
  title={Basic hypergeometric series},
  author={Gasper, George and Rahman, Mizan},
  volume={96},
  year={2011},
  publisher={Cambridge university press}
}

@Article{Bryc2017,
author={Bryc, Wlodek
and Weso{\l }owski, Jacek},
title={Asymmetric Simple Exclusion Process with Open Boundaries and Quadratic Harnesses},
journal={J. Stat. Phys.},
year={2017},
month={Apr},
day={01},
volume={167},
number={2},
pages={383-415},
issn={1572-9613},
doi={10.1007/s10955-017-1747-5},
url={https://doi.org/10.1007/s10955-017-1747-5}
}

@article{KPZ86,
  author  = {Kardar, Mehran and Parisi, Giorgio and Zhang, Yi-Cheng},
  title   = {Dynamic scaling of growing interfaces},
  journal = {Phys. Rev. Lett.},
  volume  = {56},
  number  = {9},
  pages   = {889--892},
  year    = {1986},
  doi     = {10.1103/PhysRevLett.56.889}
}

@article{derrida1993exact,
  title={Exact solution of a {1D} asymmetric exclusion model using a matrix formulation},
  author={Derrida, B. and Evans, M. R and Hakim, V. and Pasquier, V.},
  journal={J. Phys. A},
  volume={26},
  number={7},
  pages={1493},
  year={1993},
  publisher={IOP Publishing}
}

@article{bryc2019limit,
  title={Limit fluctuations for density of asymmetric simple exclusion processes with open boundaries},
  author={Bryc, W. and Wang, Y.},
  journal={Ann. Inst. Henri Poincaré  Probab. Stat.},  
  volume={55},
  number={4},
  pages={2169--2194},
  year={2019}
}

@article{macdonald1968kinetics,
  author     = {MacDonald, Carolyn T. and Gibbs, Julian H. and Pipkin, Allen C.},
  title      = {Kinetics of biopolymerization on nucleic acid templates},
  journal    = {Biopolymers},
  volume     = {6},
  number     = {1},
  pages      = {1--25},
  year       = {1968},
  publisher = {Wiley Online Library},
  doi        = {10.1002/bip.1968.360060102}
}

@article{SPITZER1970246,
title = {Interaction of Markov processes},
journal = {Adv. Math.},
volume = {5},
number = {2},
pages = {246-290},
year = {1970},
issn = {0001-8708},
doi = {https://doi.org/10.1016/0001-8708(70)90034-4},
url = {https://www.sciencedirect.com/science/article/pii/0001870870900344},
author = {Frank Spitzer}
}

@article{johansson2000shape,
  title={Shape fluctuations and random matrices},
  author={Johansson, Kurt},
  journal={Comm. Math. Phys.},
  volume={209},
  number={2},
  pages={437--476},
  year={2000},
  publisher={Springer}
}

@article{blythe2007nonequilibrium,
  title={Nonequilibrium steady states of matrix-product form: a solver's guide},
  author={Blythe, Richard A and Evans, Martin R},
  journal={J. Phys. A},
  volume={40},
  number={46},
  pages={R333},
  year={2007},
  publisher={IOP Publishing}
}

@article{corwin2022some,
  title={Some recent progress on the stationary measure for the open {KPZ} equation},
  author={Corwin, I.},
  journal={Toeplitz Operators and Random Matrices: In Memory of Harold Widom},
  pages={321--360},
  year={2022},
  publisher={Springer}
}

@article{rost1981non,
  title={Non-equilibrium behaviour of a many particle process: Density profile and local equilibria},
  author={Rost, H.},
  journal={Zeitschrift f{\"u}r Wahrscheinlichkeitstheorie und Verwandte Gebiete},
  volume={58},
  number={1},
  pages={41--53},
  year={1981},
  publisher={Springer}
}

@article{matetski2021kpz,
  title={The {KPZ} fixed point},
  author={Matetski, K. and Quastel, J. and Remenik, D.},
  journal={Acta Math.},
  volume={227},
  number={1},
  pages={115--203},
  year={2021},
  publisher={Lehigh University Bethlehem, Penn., USA}
}

@article{yang2022stationary,
  title={Stationary measure for six-vertex model on a strip},
  author={Yang, Zongrui},
  journal={Electron. J. Probab.},
  volume={29},
  pages={1--28},
  year={2024},
  publisher={The Institute of Mathematical Statistics and the Bernoulli Society}
}

@article{mallick1997finite,
  title={Finite-dimensional representations of the quadratic algebra: applications to the exclusion process},
  author={Mallick, Kirone and Sandow, Sven},
  journal={J. Phys. A},
  volume={30},
  number={13},
  pages={4513},
  year={1997},
  publisher={IOP Publishing}
}

@article{nestoridi2024approximating,
  title={Approximating the stationary distribution of the {ASEP} with open boundaries},
  author={Nestoridi, Evita and Schmid, Dominik},
  journal={Comm. Math. Phys. },
  volume={405},
  number={8},
  pages={176},
  year={2024},
  publisher={Springer}
}

@article{schuetz2022reverse,
  title={A reverse duality for the {ASEP} with open boundaries},
  author={Sch{\"u}tz, Gunter},
  journal={J. Phys. A},
  year={2022}
}

@article{bryc2024stationary,
  title={Stationary distribution of open asymmetric simple exclusion processes on an interval as a marginal of a two-layer ensemble},
  author={Bryc, W{\l}odek},
  journal={ALEA Lat. Am. J. Probab. Math. Stat.},
  volume={22},
  pages={131--148},
  year={2025}
}

@article{aggarwal2024scaling,
  title={{Scaling limit of the colored {ASEP} and stochastic six-vertex models}},
  author={Aggarwal, Amol and Corwin, Ivan and Hegde, Milind},
  journal={arXiv preprint arXiv:2403.01341},
  year={2024}
}

@article{aggarwal2024kpz,
  title={{KPZ fixed point convergence of the ASEP and stochastic six-vertex models}},
  author={Aggarwal, Amol and Corwin, Ivan and Hegde, Milind},
  journal={arXiv preprint arXiv:2412.18117},
  year={2024}
}

@article{barraquand2024stationary,
  title={Stationary measures for integrable polymers on a strip},
  author={Barraquand, Guillaume and Corwin, Ivan and Yang, Zongrui},
  journal={Invent. Math. },
  pages={1--75},
  year={2024},
  publisher={Springer}
}

@article{bryc2024limitCoex,
  title={{Limit Fluctuations of Stationary Measure of Totally Asymmetric Simple Exclusion Process with Open Boundaries on the Coexistence Line}},
  author={Bryc, W{\l}odzimierz and Najnudel, Joseph and Wang, Yizao},
  journal={Comm. Math. Phys.},
  volume={406},
  number={8},
  pages={197},
  year={2025},
  publisher={Springer}
}

@article{wang2023askey,
  title={{Askey--Wilson signed measures and open {ASEP} in the shock region}},
  author={Wang, Yizao and Weso{\l}owski, Jacek and Yang, Zongrui},
  journal={Int. Math. Res. Not. IMRN},
  pages={rnae116},
  year={2024},
  publisher={Oxford University Press}
}

@book{kipnis1999scaling,
  author    = {Kipnis, Claude and Landim, Claudio},
  title     = {Scaling Limits of Interacting Particle Systems},
  series    = {Grundlehren der mathematischen Wissenschaften},
  volume    = {320},
  publisher = {Springer-Verlag},
  address   = {Berlin},
  year      = {1999},
  doi       = {10.1007/978-3-662-03752-2}
}

@article{sasamoto2000density,
  title={Density profile of the one-dimensional partially asymmetric simple exclusion process with open boundaries},
  author={Sasamoto, Tomohiro},
  journal={J. Phys. Soc. Jpn.},
  volume={69},
  number={4},
  pages={1055--1067},
  year={2000},
  publisher={The Physical Society of Japan}
}

@article{brak2006combinatorial,
  title={A combinatorial derivation of the {PASEP} stationary state},
  author={Brak, Richard and Corteel, Sylvie and Essam, John and Parviainen, Robert and Rechnitzer, Andrew},
  journal={Electron. J. Combin. },
  volume={13},
  number={1},
  pages={R108},
  year={2006}
}

@article{liggett1975ergodic,
  title={Ergodic theorems for the asymmetric simple exclusion process},
  author={Liggett, Thomas M},
  journal={Trans. Amer. Math. Soc.},
  volume={213},
  pages={237--261},
  year={1975}
}

@article{derrida2002large,
  title={Large deviation of the density profile in the steady state of the open symmetric simple exclusion process},
  author={Derrida, Bernard and Lebowitz, JL and Speer, ER},
  journal={J. Stat. Phys.},
  volume={107},
  number={3},
  pages={599--634},
  year={2002},
  publisher={Springer}
}

@article{wood2020combinatorial,
  title={Combinatorial mappings of exclusion processes},
  author={Wood, Anthony J and Blythe, Richard A and Evans, Martin R},
  journal={J. Phys. A },
  volume={53},
  number={12},
  pages={123001},
  year={2020},
  publisher={IOP Publishing}
}

@article{hegde2024large,
  title={Large deviation principle for the stationary measures of open asymmetric simple exclusion processes},
  author={Hegde, Milind and Yang, Zongrui},
  journal={arXiv preprint arXiv:2412.12026},
  year={2024}
}

@article{denisov1983random,
  title={Random walk and the {W}iener process considered from a maximum point},
  author={Denisov, IV},
  journal={Teor. Veroyatnost. i Primenen},
  volume={28},
  number={4},
  pages={785--788},
  year={1983}
}

@article{SY:openLight,
  author  = {Schmid, Dominik and Yang, Zongrui},
  title   = {The open {ASEP} with light particles},
  journal = {Probab. Theory Related Fields},
  year    = {2026},
  doi     = {10.1007/s00440-026-01486-w},
  url     = {https://doi.org/10.1007/s00440-026-01486-w}
}

@article{jackson1909xi,
  title={XI.—{O}n {$q$}-functions and a certain difference operator},
  author={Jackson, Frederick H},
  journal={Earth and Environmental Science Transactions of the Royal Society of Edinburgh},
  volume={46},
  number={2},
  pages={253--281},
  year={1909},
  publisher={Royal Society of Edinburgh Scotland Foundation}
}

@incollection{koekoek2010hypergeometric,
  title={Hypergeometric orthogonal polynomials},
  author={Koekoek, Roelof and Lesky, Peter A and Swarttouw, Ren{\'e} F},
  booktitle={Hypergeometric Orthogonal Polynomials and Their $q$-Analogues},
  pages={183--253},
  year={2010},
  publisher={Springer}
}

@article{yang2025limits,
  title={Limits of open {ASEP} stationary measures near a boundary},
  author={Yang, Zongrui},
  journal={Ann. Appl. Probab.},
  volume={35},
  number={5},
  pages={3242--3270},
  year={2025},
  publisher={Institute of Mathematical Statistics}
}

@book{durrett2019probability,
  title={Probability: theory and examples},
  author={Durrett, Rick},
  volume={49},
  year={2019},
  publisher={Cambridge university press}
}

@article{dembo1996tail,
  title={Tail estimates for one-dimensional random walk in random environment},
  author={Dembo, Amir and Peres, Yuval and Zeitouni, Ofer},
  journal={Comm. Math. Phys.},
  volume={181},
  number={3},
  pages={667--683},
  year={1996},
  publisher={Springer}
}

@unpublished{SchmidYangCoexistence, 
title = {Sharp approximations for the stationary measures of the open {ASEP} {II}: The coexistence line}, 
author = {Schmid, Dominik and Yang, Zongrui},   
year = {2026} 
}

@unpublished{SchmidYangFan, 
title = {Sharp approximations for the stationary measures of the open {ASEP} {III}: The fan region}, 
author = {Schmid, Dominik and Yang, Zongrui},  
year = {2026} 
}

@unpublished{SchmidYangOpenKPZ,
title = {{Stationary measures for the open KPZ equation: a short proof}}, 
author = {Schmid, Dominik and Yang, Zongrui},  
year = {2026} 
}

\end{document}